\documentclass[11pt,english,openany,oneside]{article}

\usepackage[T1]{fontenc}
\usepackage[latin9]{inputenc}
\usepackage{fullpage}
\usepackage{units}
\usepackage{amsthm}
\usepackage{amssymb}
\usepackage{mathrsfs}
\usepackage{mathtools}
\usepackage[all]{xy}
\usepackage{wrapfig}
\usepackage[english]{babel}
\usepackage{indentfirst}
\usepackage{fancyhdr}
\usepackage{amssymb}
\usepackage{amsmath}
\usepackage{latexsym}
\usepackage{amsthm}
\usepackage{eucal}
\usepackage{eufrak}
\usepackage[pdftex]{graphicx}
\usepackage{amscd}
\usepackage{color}
\usepackage{lipsum}
\numberwithin{equation}{section} 
\numberwithin{figure}{section} 
\theoremstyle{plain}
\newtheorem{thm}{Theorem}[section]
  \theoremstyle{definition}
  \newtheorem{defn}[thm]{Definition}
  \theoremstyle{definition}
  
  \theoremstyle{remark}
  \newtheorem{rem}[thm]{Remark}
  \theoremstyle{plain}
  \newtheorem{prop}[thm]{Proposition}
  \theoremstyle{plain}
  \newtheorem{cor}[thm]{Corollary}
  \theoremstyle{plain}
  \newtheorem{lemma}[thm]{Lemma}
  \theoremstyle{remark}
  
  \theoremstyle{plain}
  
  \theoremstyle{plain}
  
  \theoremstyle{remark}
  \newtheorem{war}[thm]{Warning}

\newcommand\phantomarrow[2]{%
  \setbox0=\hbox{$\displaystyle #1\to$}%
  \hbox to \wd0{%
    $#2\mapstochar
     \cleaders\hbox{$\mkern-1mu\relbar\mkern-3mu$}\hfill
     \mkern-7mu\rightarrow$}%
  \,}
\makeatletter
\newcommand*{\doublerightarrow}[2]{\mathrel{
  \settowidth{\@tempdima}{$\scriptstyle#1$}
  \settowidth{\@tempdimb}{$\scriptstyle#2$}
  \ifdim\@tempdimb>\@tempdima \@tempdima=\@tempdimb\fi
  \mathop{\vcenter{
    \offinterlineskip\ialign{\hbox to\dimexpr\@tempdima+1em{##}\cr
    \rightarrowfill\cr\noalign{\kern.5ex}
    \rightarrowfill\cr}}}\limits^{\!#1}_{\!#2}}}
\newcommand*{\triplerightarrow}[1]{\mathrel{
  \settowidth{\@tempdima}{$\scriptstyle#1$}
  \mathop{\vcenter{
    \offinterlineskip\ialign{\hbox to\dimexpr\@tempdima+1em{##}\cr
    \rightarrowfill\cr\noalign{\kern.5ex}
    \rightarrowfill\cr\noalign{\kern.5ex}
    \rightarrowfill\cr}}}\limits^{\!#1}}}
\newcommand\blfootnote[1]{%
  \begingroup
  \renewcommand\thefootnote{}\footnote{#1}%
  \addtocounter{footnote}{-1}%
  \endgroup
}

\makeatother

\title{Derived Moduli of Complexes and Derived Grassmannians}
\author{Carmelo Di Natale}
\date{}

\begin{document}

\maketitle

\abstract{In the first part of this paper we construct a model structure for the category of filtered cochain complexes of modules over some commutative ring $R$ and explain how the classical Rees construction relates this to the usual projective model structure over cochain complexes. The second part of the paper is devoted to the study of derived moduli of sheaves: we give a new proof of the representability of the derived stack of perfect complexes over a proper scheme and then use the new model structure for filtered complexes to tackle moduli of filtered derived modules. As an application, we construct derived versions of Grassmannians and flag varieties.}

\tableofcontents{}

\section*{Introduction}
Recent\blfootnote{This work was supported by the Engineering and Physical Sciences Research Council [grant number EP/I004130/1].} developments in Derived Algebraic Geometry have lead many mathematicians to revise their approach to Moduli Theory: in particular one of the most striking results in this area is certainly \emph{Lurie Representability Theorem} -- proved by Lurie in 2004 as the main result of his PhD thesis \cite{Lu1} -- which provides us with an explicit criterion to check whether a simplicial presheaf over some $\infty$-category of derived algebras gives rise to a derived geometric stack. Unfortunately the conditions in Lurie's result are often quite complicated to verify in concrete derived moduli problems involving algebro-geometrical structures, so for several years a rather narrow range of derived algebraic stacks have actually been constructed: in particular the most significant example known was probably the locally geometric derived stack of perfect complexes over a smooth proper scheme $X$, which was firstly studied by To\" en and Vaqui\' e in 2007 (see \cite{TVa}) without using any representability result \emph{\` a la} Lurie. Nonetheless a few years later Pridham developed in \cite{Pr5} several new representability criteria for derived geometric stacks which have revealed to be more suitable to tackle moduli problems arising in Algebraic Geometry, as he himself showed in \cite{Pr1} where he used such criteria to construct a variety of derived moduli stacks for schemes and (complexes of) sheaves. In \cite{Ha-LePr} Halpern-Leistner and Preygel have also recovered To\" en and Vaqui\' e's result by using some generalisation of \emph{Artin Representability Theorem} for ordinary algebraic stacks (see \cite{Art}), though their approach is not based on Pridham's theory, while in \cite{Pa} Pandit generalised it to non-necessarily smooth schemes by studying the derived moduli stack of compact objects in a perfect symmetric monoidal infinity-category.\\
In this paper we give a third proof of existence and local geometricity of derived moduli for perfect complexes by means of Pridham's representability and then look at derived moduli of filtered perfect complexes: our main result is Theorem \ref{filt Perf repr}, which essentially shows that filtered perfect complexes of $\mathscr O_X$-modules -- where $X$ is a proper scheme -- are parametrised by a locally geometric derived stack. In our strategy a key ingredient to tackle derived moduli of filtrations -- in addition to Pridham's representability -- is a good Homotopy Theory of filtered modules in complexes: for this reason the first part of this paper is devoted to construct a satisfying model structure on the category $\mathfrak{FdgMod}_R$, which is probably an interesting matter in itself. Theorem \ref{fdg model thm} shows that $\mathfrak{FdgMod}_R$ is endowed with a natural cofibrantly generated model structure and Theorem \ref{Rprops} proves that this is nicely related to the standard projective model structure on $\mathfrak{dgMod}_R$ via the Rees construction. In the end, we conclude this paper by constructing derived versions of Grassmannians and flag varieties, which are obtained as suitable homotopy fibres of the derived stack of filtrations over the derived stack of complexes. \\
$\mathbf{Acknowledgements}$ --- The author does wish to thank his PhD supervisor Jonathan P. Pridham for suggesting the problem and for his constant support and advise along all the preparation of this paper. The author is also deeply indebted to Domenico Fiorenza, Ian Grojnowski, Julian V. S. Holstein, Donatella Iacono, Dominic Joyce, Marco Manetti and Elena Martinengo for several inspiring discussions about Grassmannians and flag varieties.

\section{Homotopy Theory of Filtered Structures}

This chapter is devoted to the construction of a good homotopy theory for filtered cochain complexes; for this reason we will first recall the standard projective model structure on cochain complexes and then use it to define a suitable one for filtered objects. At last, we will also study the Rees functor from a homotopy-theoretic viewpoint and see that it liaises coherently dg structures with filtered cochain ones.
\subsection{Background on Model Categories}
This section is devoted to review a few complementary definitions and results in Homotopy Theory which will be largely used in this paper; we will assume that the reader is familiar with the notions of model category, simplicial category and differential graded category: references for them include \cite{DS}, \cite{GJ}, \cite{Hir}, \cite{Ho}, \cite{Qui} and \cite{Toe2}, while \cite{GS} provides a very clear and readable overview.\\
Let $\mathfrak C$ be a complete and cocomplete category and $I$ a class of morphisms in $\mathfrak C$; recall from \cite{Ho} that:
\begin{enumerate}
\item a map is $I$\emph{-injective} if it has the right lifting property with respect to every map in $I$ (denote by $I$-inj the class of $I$-injective morphisms in $\mathfrak C$);
\item a map is $I$\emph{-projective} if it has the left lifting property with respect to every map in $I$ (denote by $I$-proj the class of $I$-projective morphisms in $\mathfrak C$);
\item a map is an $I$\emph{-cofibration} if it has the left lifting property with respect to every $I$-injective map (denote by $I$-cof the class of $I$-cofibrations in $\mathfrak C$);
\item a map is an $I$\emph{-fibration} if it has the right lifting property with respect to every $I$-projective map (denote by $I$-fib the class of $I$-fibrations in $\mathfrak C$);
\item a map is a \emph{relative }$I$\emph{-cell complex} if it is a transfinite composition of pushouts of elements of $I$ (denote by $I$-cell the class of $I$-cell complexes).
\end{enumerate} 
The above classes of morphisms satisfy a number of comparison relations: in particular we have that:
\begin{itemize}
\item $I\text{-cof}=\left(I\text{-inj}\right)\text{-proj}$ and $I\text{-fib}=\left(I\text{-proj}\right)\text{-inj}$;
\item $I\subseteq I$-cof and $I\subseteq I$-fib;
\item $\left(I\text{-cof}\right)\text{-inj}=I\text{-inj}$ and $\left(I\text{-fib}\right)\text{-proj}=I\text{-proj}$;
\item $I\text{-cell}\subseteq I\text{-cof}$ (see \cite{Ho} Lemma 2.1.10);
\item if $I\subseteq J$ then $I\text{-inj}\supseteq J\text{-inj}$ and $I\text{-proj}\supseteq J\text{-proj}$, thus $I\text{-cof}\supseteq J\text{-cof}$ and $I\text{-fib}\supseteq J\text{-fib}$.
\end{itemize}
Fix some class $S$ of morphisms in $\mathfrak C$ and recall that an object $A\in\mathfrak C$ is said to be \emph{compact\footnote{In the language of \cite{Ho} compact objects are called $\aleph_0$-\emph{small}.} relative to $S$} if for all sequences
\begin{equation*}
C_0\longrightarrow C_1\longrightarrow\cdots\longrightarrow C_{n}\longrightarrow C_{n+1}\longrightarrow\cdots
\end{equation*}
such that each map $C_{n}\rightarrow C_{n+1}$ is in $S$, the natural map
\begin{equation*}
\underset{\underset{n}{\longrightarrow}}{\mathrm{lim}}\,\mathrm{Hom}_{\mathfrak C}\left(A,C_n\right)\longrightarrow \mathrm{Hom}_{\mathfrak C}\left(A,\underset{\underset{n}{\longrightarrow}}{\mathrm{lim}}\,C_{n}\right)
\end{equation*}
is an isomorphism; moreover $A$ is said to be \emph{compact} if it is compact relative to $\mathfrak C$.
\begin{defn} \label{cof gen mod cat}
A model category $\mathfrak C$ is said to be \emph{(compactly) cofibrantly generated}\footnote{The definition of cofibrantly generated model category as found in \cite{Ho} Section 2.1 is slightly more general than the one provided by Definition \ref{cof gen mod cat}, as it involves small objects rather than compact ones; anyway the proper definition requires some non-trivial Set Theory and moreover all examples we consider in this paper fit into the weaker notion determined by Definition \ref{cof gen mod cat}, so we will stick to this.} if there are sets $I$ and $J$ of maps such that:
\begin{enumerate}
\item the domains of the maps in $I$ are compact relative to $I$-cell;
\item the domains of the maps in $J$ are compact relative to $J$-cell;
\item the class of fibrations is $J$-inj;
\item the class of trivial fibrations is $I$-inj.
\end{enumerate}
$I$ is said to be the set of \emph{generating cofibrations}, while $J$ is said to be the set of \emph{generating trivial cofibrations}.
\end{defn}
Cofibrantly generated model categories are very useful as they come with a quite explicit characterisation of (trivial) fibrations and (trivial) cofibrations: this is exactly the content of the next result.
\begin{prop}
Let $\mathfrak C$ be a cofibrantly generated model category with $I$ and $J$ respectively being the set of generating cofibrations and generating trivial cofibrations. We have that:
\begin{enumerate}
\item the cofibrations form the class $I$-cof;
\item every cofibration is a retract of a relative $I$-cell complex;
\item the domains of $I$ are compact relative to the class of cofibrations;
\item the trivial cofibrations form the class $J$-cof;
\item every trivial cofibration is a retract of a relative $J$-cell complex;
\item the domains of $J$ are compact relative to the trivial cofibrations.
\end{enumerate} 
\end{prop} 
\begin{proof}
See \cite{Ho} Proposition 2.1.18, which in turn relies on \cite{Ho} Corollary 2.1.15 and \cite{Ho} Proposition 2.1.16.
\end{proof}
The main reason we are interested in cofibrantly generated model categories is that they fit into a very powerful existence criterion -- essentially due to Kan and Quillen and then developed by many more authors -- which will be repeatedly used along this paper.
\begin{thm} \label{criterion model} \emph{(Kan, Quillen)}
Let $\mathfrak C$ be a complete and cocomplete category and $W$, $I$, $J$ three sets of maps. Then $\mathfrak C$ is endowed with a cofibrantly generated model structure with $W$ as the set of weak equivalences, $I$ as the set of generating cofibrations and $J$ as the set of generating trivial cofibrations if and only if:
\begin{enumerate}
\item the class $W$ has the two-out-of-three property and is closed under retracts;
\item the domains of $I$ are compact relative to $I$-cell;
\item the domains of $J$ are compact relative to $J$-cell;
\item $J$-cell$\,\subseteq W\cap I$-cof;
\item $I$-inj$\,\subseteq W\cap J$-inj;
\item either $W\cap I$-cof$\,\subseteq J$-cof or $W\cap J$-inj$\,\subseteq I$-inj.
\end{enumerate}
\end{thm}
\begin{proof}
See \cite{Ho} Theorem 2.1.19.
\end{proof}
Theorem \ref{criterion model} is a great tool in order to construct new model categories; moreover if we are given a cofibrantly generated model category we can often induce a good homotopy theory over other categories: this is the content of the following result, which again is essentially due to Kan and Quillen.
\begin{thm} \label{Hirschhorn} \emph{(Kan, Quillen)}
Let $\mathbf F:\mathfrak C\rightleftarrows\mathfrak D:\mathbf G$ be an adjoint pair of functors and assume $\mathfrak C$ is a cofibrantly generated model category, with $I$ as set of generating cofibrations, $J$ as set of generating trivial fibrations and $W$ as set of weak equivalences. Suppose further that:
\begin{enumerate}
\item $\mathbf G$ preserves sequential colimits;
\item $\mathbf G$ maps relative $\mathbf FJ$-cell\footnote{Of course, if $S$ is a set of morphisms in $\mathfrak C$, $\mathbf FS$ will denote the set $$\mathbf FS:=\left\{\mathbf Ff\text{ s.t. }f\in S\right\}.$$} complexes to weak equivalences.
\end{enumerate}
Then the category $\mathfrak D$ is endowed with a cofibrantly generated model structure where $\mathbf FI$ is the set of generating cofibrations, $\mathbf FJ$ is the set of generating trivial cofibrations and $\mathbf FW$ as set of weak equivalences. Moreover $\left(\mathbf F,\mathbf G\right)$ is a Quillen pair with respect to these model structures.
\end{thm}
\begin{proof}
See \cite{Hir} Theorem 11.3.2.
\end{proof}
The end of this section is devoted to review a famous comparison result due to Dold and Kan establishing an equivalence between the category of non-negatively graded chain complexes of $k$-vector spaces and that of simplicial $k$-vector spaces, which has very profound consequences in Homotopy Theory.
\begin{war}
Be aware that in the end of this section we will deal with (non-negatively graded) differential graded chain structures, while in the rest of the paper we will mostly be interested in cochain objects.
\end{war}
First of all, recall that the \emph{normalisation} of a simplicial $k$-vector space $\left(V,\partial_i,\sigma_j\right)$ is defined to be the non-negatively graded chain complex of $k$-vector spaces $\left(\mathbf NV,\delta\right)$ where
\begin{equation*}
\left(\mathbf NV\right)_n:=\bigcap_i\mathrm{ker}\left(\partial_i:V_n\rightarrow V_{n-1}\right)
\end{equation*}
and $\delta_n:=\left(-1\right)^n\partial_n$. Such a procedure defines a functor
\begin{equation*}
\mathbf N:\mathfrak{sVect}_k\longrightarrow\mathfrak{Ch}_{\geq 0}\left(\mathfrak{Vect}_k\right).
\end{equation*}
On the other hand, let $V$ be a chain complex of $k$-vector spaces and recall that its \emph{denormalisation} is defined to be the simplicial vector space $\left(\left(\mathbf KV\right),\partial_i,\sigma_j\right)$ given in level $n$ by the vector space 
\begin{equation*}
\left(\mathbf KV\right)_n:=\underset{\eta\text{ surjective}}{\underset{\eta\in\mathrm{Hom}_{\Delta}\left(\left[p\right],\left[n\right]\right)}{\prod}}V_p\left[\eta\right]\qquad\qquad\qquad \left(V_p\left[\eta\right]\simeq V_p\right).
\end{equation*}
\begin{rem}
Notice that
\begin{equation*}
\left(\mathbf KV\right)_n\simeq V_0\oplus V_1^{\oplus n}\oplus V_2^{\oplus\binom{n}{2}}\oplus\cdots\oplus V_k^{\oplus\binom{n}{k}}\oplus\cdots\oplus V_n^{\oplus\binom{n}{n}}.
\end{equation*}
\end{rem}
In order to complete the definition of the denormalisation of $V$ we need to define face and degeneracy maps: we will describe a combinatorial procedure to determine all of them. For all morphisms $\alpha:\left[m\right]\rightarrow\left[n\right]$ in $\Delta$, we want to define a linear map $\mathbf K\left(\alpha\right):\left(\mathbf KV\right)_n\rightarrow\left(\mathbf KV\right)_m$; this will be done by describing all restrictions $\mathbf K\left(\alpha,\eta\right):V_p\left[\eta\right]\rightarrow\left(\mathbf KV\right)_m$, for any surjective non-decreasing map $\eta\in\mathrm{Hom}_{\Delta}\left(\left[p\right],\left[n\right]\right)$. \\
For all such $\eta$, take the composite $\eta\circ\alpha$ and consider its epi-monic factorisation\footnote{The existence of such a decomposition is one of the key properties of the category $\Delta$.} $\epsilon\circ\eta'$, as in the diagram
\begin{equation*}
\tiny{\xymatrix{\left[m\right]\ar[r]^{\alpha}\ar[d]_{\eta'} & \left[n\right]\ar[d]^{\eta} \\
\left[q\right]\ar[r]^{\epsilon} & \left[p\right].} }
\end{equation*}
Now
\begin{itemize}
\item if $p=q$ (in which case $\epsilon$ is just the identity map), then set $\mathbf K\left(\alpha,\eta\right)$ to be the natural identification of $V_p\left[\eta\right]$ with the summand $V_p\left[\eta'\right]$ in $\left(\mathbf KV\right)_m$;
\item if $p=q+1$ and $\epsilon$ is the unique injective non-decreasing map from $\left[p\right]$ to $\left[p+1\right]$ whose image misses $p$, then set $\mathbf K\left(\alpha,\eta\right)$ to be the differential $d_p:V_p\rightarrow V_{p-1}$;
\item in all other cases set $\mathbf K\left(\alpha,\eta\right)$ to be the zero map.
\end{itemize}
The above setting characterises all the structure of the simplicial vector space $\left(\left(\mathbf KV\right),\partial_i,\sigma_j\right)$; again, such a procedure defines a functor 
\begin{equation*}
\mathbf K:\mathfrak{Ch}_{\geq 0}\left(\mathfrak{Vect}_k\right)\longrightarrow\mathfrak{sVect}_k.
\end{equation*}
\begin{thm} \label{Dold-Kan} \emph{(Dold, Kan)}
The functors $\mathbf N$ and $\mathbf K$ form an equivalence of categories between $\mathfrak{sMod}_k$ and $\mathfrak{Ch}_{\geq 0}\mathfrak{Mod}_k$. 
\end{thm}
\begin{proof}
See \cite{GJ} Corollary 2.3 or \cite{We} Theorem 8.4.1.
\end{proof}
The Dold-Kan correspondence described in Theorem \ref{Dold-Kan} is known to induce a number of very interesting $\infty$-equivalences: for more details see for example \cite{dN}, \cite{GJ}, \cite{Tab2} and \cite{We}.
\subsection{Homotopy Theory of Cochain Complexes}
Fix a commutative unital ring $R$: in this section we will review the standard model structure by which one usually endows the category of (unbounded) cochain complexes $R$, i.e. the so-called \emph{projective model structure}; all the section is largely based on \cite{Ho} Section 2.3, where the homotopy theory of chain complexes over a commutative unital ring is extensively studied: actually all results, constructions and arguments we are about to discuss are essentially dual versions of the ones given there. \\
Recall that the category $\mathfrak{dgMod}_R$ of \emph{cochain complex of $R$-modules} (also referred as \emph{$R$-module in complexes}) is one of the main examples of abelian category: as a matter of fact it is complete and cocomplete (limits and colimits are taken degreewise), the complex $\left(0,0\right)$ defined by the trivial module in each degree gives the zero object and short exact sequences are defined degreewise; for more details see \cite{We}. \\
Let $\left(M,d\right)\in\mathfrak{dgMod}_R$: as usual, we define its $R$-module of $n$-cocycles to be $Z^n\left(M\right):=\ker\left(d^n\right)$, its $R$-module of $n$-coboundaries to be $B^n\left(M\right):=\mathrm{Im}\;d^{n-1}\leq Z^n\left(M\right)$ and its $n^{\mathrm{th}}$ cohomology $R$-module to be $H^n\left(M\right):=\nicefrac{Z^n\left(M\right)}{B^n\left(M\right)}$; $\left(M,d\right)$ is said to be \emph{acyclic} if $H^n\left(M\right)=0$ $\forall n\in\mathbb Z$; cocycles, coboundaries and cohomology define naturally functors from the category $\mathfrak{dgMod}_R$ to the category $\mathfrak{Mod}_R$ of $R$-modules. Finally, recall that a cochain map $f:\left(M,d\right)\rightarrow\left(N,\delta\right)$ is said to be a \emph{quasi-isomorphism} if it is a cohomology isomorphism, i.e. if $H^n\left(f\right)\text{ }\mathrm{is}\text{ }\mathrm{an}\text{ }\mathrm{isomorphism}\text{ }\forall n\in\mathbb Z$. In the following, we will not explicitly mention the differential of a complex whenever it is clear from the context. \\
Now define the complexes 
\begin{equation*}
D_R\left(n\right):=\begin{cases}
                 R & \text{if }k=n,n+1\\
                 0 & \text{otherwise}
                 \end{cases}
\qquad\qquad
S_R\left(n\right):=\begin{cases}
                 R & \text{if }k=n\\
                 0 & \text{otherwise}
                 \end{cases}
\end{equation*}
and the only non-trivial connecting map (the one between $D_R\left(n\right)^{n}$ and $D_R\left(n\right)^{n+1}$) is the identity.
\begin{rem}
Observe that $D_R\left(n\right)$ and $S_R\left(n\right)$ are compact for all $n$.
\end{rem}
\begin{thm} \label{dg model thm}
Consider the sets
\begin{eqnarray} \label{dg model}
&I_{\mathfrak{dgMod}_R}:=\left\{S_R\left(n+1\right)\rightarrow D_R\left(n\right)\right\}_{n\in\mathbb Z}\quad\,&\nonumber \\
&J_{\mathfrak{dgMod}_R}:=\left\{0\rightarrow D_R\left(n\right)\right\}_{n\in\mathbb Z}\qquad\;& \nonumber \\
&W_{\mathfrak{dgMod}_R}:=\left\{f:M\rightarrow N|f\text{ }\mathrm{is}\text{ }\mathrm{a}\text{ }\mathrm{quasi}\text{-}\mathrm{isomorphism}\right\}.&
\end{eqnarray}
The classes \eqref{dg model} define a cofibrantly generated model structure on $\mathfrak{dgMod}_R$, where $I_{\mathfrak{dgMod}_R}$ is the set of generating cofibrations, $J_{\mathfrak{dgMod}_R}$ is the set of generating trivial cofibrations and $W_{\mathfrak{dgMod}_R}$ is the set of weak equivalences. 
\end{thm}
The proof of Theorem \ref{dg model thm} (which corresponds to \cite{Ho} Theorem 2.3.11) relies on Theorem \ref{criterion model}, thus it amounts to explicitly describe fibrations, trivial fibrations, cofibrations and trivial cofibrations determined by the sets \eqref{dg model}, which we do in the following propositions.
\begin{prop} \label{DG fib}
$p\in\mathrm{Hom}_{\mathfrak{dgMod}_R}\left(M,N\right)$ is a fibration if and only if it is a degreewise surjection.
\end{prop}
\begin{proof}
We want to characterise diagrams
\begin{equation} \label{dg D^n lift}
\tiny{\xymatrix{0\ar[r]\ar[d] & M\ar[d]^p \\
D_R\left(n\right)\ar[r] & N} }
\end{equation}
in $\mathfrak{dgMod}_R$ admitting a lifting. A diagram like \eqref{dg D^n lift} is equivalent to choosing an element $y$ in $N^n$, while a lifting is equivalent to a pair $\left(x,y\right)\in M^n\times N^n$ such that $p^n\left(x\right)=y$: it follows that $p\in J_{\mathfrak{dgMod}_R}$-inj if and only if $p^n$ is surjective for all $n\in\mathbb Z$.
\end{proof}
\begin{prop} \label{DG tr fib}
$p\in\mathrm{Hom}_{\mathfrak{dgMod}_R}\left(M,N\right)$ is a trivial fibration if and only if it is in $I_{\mathfrak{dgMod}_R}$-inj; in particular $W_{\mathfrak{dgMod}_R}\cap J_{\mathfrak{dgMod}_R}\text{-inj}=I_{\mathfrak{dgMod}_R}\text{-inj}$.
\end{prop}
\begin{proof}
First of all, observe that any diagram in $\mathfrak{dgMod}_R$ of the form
\begin{equation} \label{dg S^n lift}
\tiny{\xymatrix{S_R\left(n+1\right)\ar[r]\ar[d] & M\ar[d]^p \\
D_R\left(n\right)\ar[r] & N}}
\end{equation}
is uniquely determined by an element in 
\begin{equation*}
X:=\left\{\left(x,y\right)\in N^n\oplus Z^{n+1}\left(M\right)|p^{n+1}\left(y\right)=d^n\left(x\right)\right\}.
\end{equation*}
Moreover, there is a bijection between the set of diagrams like \eqref{dg S^n lift} admitting a lifting and 
\begin{equation*}
X':=\left\{\left(x,z,y\right)\in N^n\oplus M^n\oplus Z^{n+1}\left(M\right)|p^n\left(z\right)=x,d^n\left(z\right)=y,p^{n+1}\left(y\right)=d^n\left(x\right)\right\}.
\end{equation*}
Now suppose that $p\in I$-inj: we want to prove that it is degreewise surjective (because of Proposition \ref{DG fib}) and a cohomology isomorphism. For any cocycle $y\in Z^n\left(N\right)$, the pair $\left(y,0\right)$ defines a diagram like \eqref{dg S^n lift}, therefore, as $p\in I$-inj, $\exists z\in M^n$ such that $p^n\left(z\right)=y$ and $d^n\left(z\right)=0$, so the induced map $Z^n\left(p\right):Z^n\left(M\right)\rightarrow Z^n\left(N\right)$ is surjective; in particular the map $H^n\left(p\right):H^n\left(M\right)\rightarrow H^n\left(N\right)$ is surjective as well. We now show that $p^n$ itself is surjective: fix $x\in N^n$ and consider $d^n\left(x\right)\in Z^{n+1}\left(N\right)$; as the map $Z^{n+1}\left(p\right)$ is surjective, $\exists y\in Z^{n+1}\left(M\right)$ such that $p^{n+1}\left(y\right)=d^n\left(x\right)$, thus by the assumption $\exists z\in M^n$ such that $p^n\left(z\right)=x$, hence $p$ is a degreewise surjection. It remains to prove that $H^n\left(p\right)$ is injective: fix $x\in N^{n-1}$ and consider $d^{n-1}\left(x\right)\in B^n\left(N\right)\leq Z^n\left(N\right)$; by the surjectivity of $Z^n\left(p\right)$ $\exists y\in Z^{n}\left(M\right)$ such that $p^n\left(y\right)=d^{n-1}\left(x\right)$, so $\left[y\right]\in\ker\left(H^n\left(p\right)\right)$. The pair $\left(x,y\right)$ defines a diagram of the form \eqref{dg S^n lift}, so the assumption on $p$ implies the existence of $z\in M^{n-1}$ such that $d^{n-1}\left(z\right)=y$ and $p^{n-1}\left(z\right)=x$; in particular, we have that $\ker\left(H^n\left(p\right)\right)=0$, so $H^n\left(p\right)$ is injective. \\
Now assume that $p$ is a trivial fibration, i.e. a degreewise surjection with acyclic kernel; consider $\left(x,y\right)\in X$: we want to find $z\in M^n$ such that $\left(x,z,y\right)\in X'$. The hypotheses on $p$ are equivalent to the existence of a short exact sequence in $\mathfrak{dgMod}_R$
\begin{equation*}
\xymatrix{0\ar[r] & K\ar[r] & M\ar[r]^p & N\ar[r] & 0}
\end{equation*}
such that $H^n\left(K\right)=0$ $\forall n\in\mathbb Z$. Take any $w\in M^n$ such that $p^n\left(w\right)=x$; an immediate computation shows that $dw-y\in Z^{n+1}\left(K\right)$ and, as $K$ is acyclic, $\exists v\in K^n$ such that $dv=dw-y$. Now define $z:=w-v$ and the result follows.
\end{proof}
The next step is describing cofibrations and trivial cofibrations generated by the sets \eqref{dg model}, but we need to understand cofibrant objects in order to do that; in the following for any $R$-module $P$ call $D_R\left(n,P\right)$ the cochain complex defined by
\begin{equation*}
D_R\left(n,P\right):=\begin{cases}
                 P & \text{if }k=n,n+1\\
                 0 & \text{otherwise}
                 \end{cases}
\end{equation*}
and in which the only non-trivial connecting map is the identity.
\begin{prop} \label{DG cofibrant}
If $A\in\mathfrak{dgMod}_R$ is cofibrant, then $A^n$ is projective for all $n$. Conversely, any bounded above cochain complex of projective $R$-modules is cofibrant.
\end{prop}
\begin{proof}
Suppose $A$ is a cofibrant object in $\mathfrak{dgMod}_R$ and let $p:M\twoheadrightarrow N$ be a surjection between two $R$-modules; the $R$-linear map $p:M\twoheadrightarrow N$ induces a morphism $\dot p:D_R\left(n,M\right)\rightarrow D_R\left(n,N\right)$ (given by $p$ itself in degree $n$ and $n-1$ and by the zero map elsewhere), which is immediately seen to be degreewise surjective with acyclic kernel, hence a trivial fibration by Proposition \ref{DG tr fib}. Moreover any $R$-linear map $f:A^n\rightarrow N$ defines a cochain morphism $\dot f:A\rightarrow D_R\left(n,N\right)$ which is given by $f$ in degree $n+1$, $fd$ in degree $n$ and $0$ elsewhere. By assumption the diagram in $\mathfrak{dgMod}_R$
\begin{equation*} \
\tiny{\xymatrix{& D_R\left(n,M\right)\ar[d]^{\dot p} \\
A\ar[r]^{\dot f}\ar[r]\ar@{-->}[ur]^g & D_R\left(n,N\right)}}
\end{equation*}
admits a lifting $g$: now it suffices to look at the above diagram in degree $n$ to see that $A^n$ is projective. \\
Now suppose $A$ is a bounded above cochain complex of projective $R$-modules (i.e. $A^n=0$ for $n\gg 0$) and fix a trivial fibration in $\mathfrak{dgMod}_R$ $p:M\rightarrow N$ and a cochain map $g:A\rightarrow N$: we want to prove that $g$ lifts to a morphism $h\in\mathrm{Hom}_{\mathfrak{dgMod}_R}\left(A,M\right)$, so we construct $h^n$ by (reverse) induction. The base of the induction is guaranteed by the fact that $A$ is assumed to be bounded above, so suppose that $h^k$ has been defined for all $k>n$; by Proposition \ref{DG tr fib} $p^n$ is surjective and has an acyclic kernel $K$, so since $A^n$ is projective $\exists f\in\mathrm{Hom}_{\mathfrak{Mod}_R}\left(A^n,M^n\right)$ lifting $g^n$. Consider the $R$-linear map $F:A^n\rightarrow M^{n+1}$ defined as $F:=d^{n}f-h^{n+1}d^n$, which measures how far $f$ is to fit into a cochain map: an easy computation shows that $p^{n+1}F=d^{n+1}F=0$, so $F:A^n\rightarrow Z^{n+1}\left(K\right)$, but, as $K$ is acyclic, we get that $F:A^n\rightarrow B^n\left(K\right)$. Of course the map $d^{n+1}$ gives a surjective $R$-linear morphism from $K^n$ to $B^{n+1}\left(K\right)$, so by the projectiveness of $A^n$ $F$ lifts to a map $G\in\mathrm{Hom}_{\mathfrak{dgMod}_R}\left(A^n,K^n\right)$. Now define $h^n:=f-G$ and the result follows.
\end{proof}
\begin{rem}
A complex of projective $R$-modules is not necessarily cofibrant (get a counterexample by adapting \cite{Ho} Remark 2.3.7); it is possible to give a complete characterisation of cofibrant objects in $\mathfrak{dgMod}_R$ in terms of \emph{dg-projective complexes} (see \cite{AFH}).
\end{rem}
\begin{prop} \label{DG cofib}
$i\in\mathrm{Hom}_{\mathfrak{dgMod}_R}\left(M,N\right)$ is a cofibration if and only if it is a degreewise injection with cofibrant cokernel.
\end{prop}
\begin{proof}
Suppose $i$ is a cofibration, i.e. a map having the left lifting property with respect to degreewise surjections with acyclic kernel, by Proposition \ref{DG tr fib}; there is an obvious morphism $M\rightarrow D_R\left(n-1,M^n\right)$ given by $d^{n-1}$ in degree $n-1$ and the identity in degree $n$, while the canonical map $D_R\left(n-1,M^n\right)\rightarrow 0$ is a trivial fibration, as $D_R\left(n-1,M^n\right)$ is clearly acyclic. As a consequence we get a diagram in $\mathfrak{dgMod}_R$
\begin{equation*} 
\tiny{\xymatrix{M\ar[d]_i\ar[r] & D_R\left(n-1,M^n\right)\ar[d] \\
N\ar[r] & 0}}
\end{equation*}
which admits a lifting as $i$ is a cofibration; in particular this implies that $i^n$ is an injection. At last recall that the class of cofibration in a model category is closed under pushouts: in particular $0\rightarrow\mathrm{coker}\left(i\right)$ is a cofibration as it is the pushout of $i$, thus $\mathrm{coker}\left(i\right)$ is cofibrant. \\
Now suppose that $i$ is a degreewise injection with cofibrant cokernel $C$ and we are given a diagram of cochain complexes 
\begin{equation} \label{diag cofib}
\tiny{\xymatrix{M\ar[r]^f\ar[d]_i & X\ar[d]^p \\
N\ar[r]^g & Y}}
\end{equation}
where $p$ is a degreewise surjection with acyclic kernel $K$ (let $j:K\rightarrow X$ be the kernel morphism): we want to construct a lifting in such a diagram. First of all notice that $N^n\simeq M^n\oplus C^n$, as $C^n$ is projective by Proposition \ref{DG cofibrant}, so we have
\begin{eqnarray*}
\xymatrix{d:N^n\ar[rr] & & N^{n+1}}\qquad\qquad\quad \\
\qquad\qquad\left(x,z\right)\longmapsto\left(d^n\left(x\right)+\tau^n\left(z\right),d^n\left(y\right)\right)
\end{eqnarray*} 
where $\tau^n:C^n\rightarrow A^{n+1}$ is some $R$-linear map such that $d^n\tau^n=\tau^nd^n$, and
\begin{eqnarray*}
\xymatrix{g^n:N^n\ar[rr] & & Y^n}\qquad\quad \\
\qquad\qquad\left(x,z\right)\longmapsto p^nf^n\left(x\right)+\sigma^n\left(z\right)
\end{eqnarray*}
where $\sigma^n:C^n\rightarrow Y^n$ satisfies the relation $d^n\sigma^n=p^nf^n\tau^n+\sigma^nd^n$. A lifting in the diagram \eqref{diag cofib} then consists of a collection $\left\{\nu^n\right\}_{n\in\mathbb Z}$ of $R$-linear morphisms such that $p^n\nu^n=\sigma^n$ and $d^n\nu^n=\nu^nd^n+f^n\tau^n$. As $C^n$ is projective, fix $G^n\in\mathrm{Hom}_{\mathfrak{Mod}_R}\left(C^n,X^n\right)$ lifting $\sigma^n$ and consider the map $F^n:C^n\rightarrow X^{n+1}$ defined as $F^n:=d^nG-Gd^n-f^n\tau^n$. It is easily seen that $p^{n+1}F^n=0$ and $d^{n+1}F^n=-d^{n+1}G^nd^n+f^{n+1}\tau^nd^{n-1}$, so there is an induced cochain map $s:C\rightarrow\Sigma K$, where $\Sigma K$ is the \emph{suspension complex} defined by the relations $\left(\Sigma K\right)^n=K^{n+1}$ and $d_{\Sigma K}=-d_K$. As $K$ is acyclic, observe that $s$ is cochain homotopic to $0$ (see \cite{Ho} Lemma 2.3.8 for details), thus there is $h^n\in\mathrm{Hom}_{\mathfrak{dgMod}_R}\left(C^n,K^n\right)$ such that $s=-d^nh^n+h^{n+1}d^n$; define $\nu^n:=G^n+j^nh^n$ and the result follows.
\end{proof}
\begin{prop} \label{DG tr cofib}
$i\in\mathrm{Hom}_{\mathfrak{dgMod}_R}\left(M,N\right)$ is in $J_{\mathfrak{dgMod}_R}$-cof if and only if it is a degreewise injection with projective\footnote{Here projective means projective as a cochain complex.} cokernel; in particular $J_{\mathfrak{dgMod}_R}\text{-cof}\subseteq W_{\mathfrak{dgMod}_R}\cap I_{\mathfrak{dgMod}_R}\text{-cof}$.
\end{prop}
\begin{proof}
Suppose $i\in J_{\mathfrak{dgMod}_R}$-cof, i.e. it has the left lifting property with respect to all fibrations; in particular it is a cofibration so by Proposition \ref{DG cofib} it is a degreewise injection with cofibrant cokernel $C$ and let $c:N\rightarrow C$ be the cokernel morphism: we want to show that $C$ is projective as a cochain complex. Fix a fibration $p:X\rightarrow Y$ and consider the diagram
\begin{equation} \label{DG tr cofib diag}
\tiny{\xymatrix{M\ar[r]^0\ar[d]_i & X\ar[d]^p & \\
N\ar[r]^{fc}\ar[r] & Y}}
\end{equation}
where $0:M\rightarrow 0\rightarrow X$ is the zero morphism and $f\in\mathrm{Hom}_{\mathfrak{dgMod}_R}\left(C,N\right)$ is an arbitrary cochain map. By assumption diagram \eqref{DG tr cofib diag} admits a lifting, which is a cochain map $h$ such that $hi=0$ and $ph=fc$; it follows that $h$ factors through a map $g\in\mathrm{Hom}_{\mathfrak{dgMod}_R}\left(C,M\right)$ lifting $f$, so $C$ is a projective cochain complex. \\
Now assume $i$ is a degreewise injection with projective cokernel $C$: again, let $c:N\rightarrow C$ denote the cokernel morphism and consider a diagram
\begin{equation*}
\tiny{\xymatrix{M\ar[r]^f\ar[d]_i & X\ar[d]^p & \\
N\ar[r]^{g}\ar[r] & Y}}
\end{equation*}
where $p$ is a fibration, i.e. a degreewise surjection because of Proposition \ref{DG fib}. Since $C$ is projective, there is a retraction $r:N\rightarrow M$ and it is easily seen that $\left(pfr-g\right)i=0$, so the map $pfr-g$ lifts to a map $s\in\mathrm{Hom}_{\mathfrak{dgMod}_R}\left(C,Y\right)$. Again, the projectiveness of $C$ implies that there is a map $t\in\mathrm{Hom}_{\mathfrak{dgMod}_R}\left(C,X\right)$ lifting $s$; now the map $fr-tc$ gives a lifting in diagram \eqref{DG tr cofib diag}. \\
The last claim of the statement follows immediately by the fact that any projective cochain complex is also acyclic (see for example \cite{We} or \cite{Ho}).
\end{proof}
\begin{prop} \label{DG equiv}
The set $W_{\mathfrak{dgMod}_R}$ of quasi-isomorphisms in $\mathfrak{dgMod}_R$ has the $2$-out-of-$3$ property and is closed under retracts.
\end{prop}
\begin{proof}
This is a classical result in Homological Algebra: for a detailed proof see \cite{HilSt} Lemma 1.1 (apply it in cohomology).
\end{proof}
The above results (especially Proposition \ref{DG tr fib}, Proposition \ref{DG tr cofib} and Proposition \ref{DG equiv}) say that the category $\mathfrak{dgMod}_R$ endowed with the structure \eqref{dg model} fits into the hypotheses of Theorem \ref{criterion model}, so Theorem \ref{dg model thm} has been proved. \\
Now assume $R$ is a (possibly differential graded) commutative $k$-algebra, where $k$ is a field of characteristic $0$: under such hypothesis there is also a canonical simplicial enrichment on $\mathfrak{dgMod}_R$ (all the rest of the section is adapted from \cite{Pr1}). \\
For all $R$-modules in complexes $M$, $N$ consider the chain complex $\left(\mathrm{HOM}_{\mathfrak{dgMod}_R}\left(M,N\right),\delta\right)$ defined by the relations
\begin{eqnarray} \label{dg Hom DG}
\mathrm{HOM}_{\mathfrak{dgMod}_R}\left(M,N\right)_n:=\mathrm{Hom}\left(M,N[-n]\right)\qquad\qquad\qquad\qquad \nonumber \\
\forall f\in {\mathrm{HOM}_{\mathfrak{dgMod}_R}\left(M,N\right)_n}\qquad\delta_n\left(f\right):=\bar d^n\circ f-\left(-1\right)^nf\circ d^n\in {\mathrm{HOM}_{\mathfrak{dgMod}_R}\left(M,N\right)_{n-1}}.
\end{eqnarray}
Formulae \eqref{dg Hom DG} make $\mathfrak{dgMod}_R$ into a differential graded category over $k$, thus the simplicial structure on $\mathfrak{dgMod}_R$ will be given by setting
\begin{equation} \label{simpl DG}
\underline{\mathrm{Hom}}_{\mathfrak{dgMod}_R}\left(M,N\right):=\mathbf K\left(\tau_{\geq 0}\mathrm{HOM}_{\mathfrak{dgMod}_R}\left(M,N\right)\right)
\end{equation}
where $\mathbf K$ is the simplicial denormalisation functor giving the Dold-Kan correspondence (see Section 1.1) and $\tau_{\geq 0}$ is good truncation.
\subsection{Homotopy Theory of Filtered Cochain Complexes}
Let $R$ be any commutative unital ring: in this section we will endow the category of filtered cochain complexes with a model structure which turns to be compatible (in a sense which will be clarified in Section 1.6) with the projective model structure on $\mathfrak{dgMod}_R$.\\
Recall that a \emph{filtered cochain complex of $R$-modules} (also referred as \emph{filtered $R$-module in complexes}) consists of a pair $\left(M,F\right)$, where $M\in\mathfrak{dgMod}_R$ and $F$ is a decreasing filtration on it, i.e. a collection $\left\{F^kM\right\}_{k\in\mathbb N}$ of subcomplexes of $M$ such that $F^{k+1}M\subseteq F^kM$ and $F^0M=M$; as a consequence an object $\left(M,F\right)\in\mathfrak{FdgMod}_R$ looks like a diagram of the form
\begin{equation*}
\tiny{\xymatrix{\cdots\ar[r] & M^{n-1}\ar[r] & M^n\ar[r] & M^{n+1}\ar[r] & \cdots \\
\cdots\ar[r] & F^1M^{n-1}\ar[r]\ar@{^{(}->}[u] & F^1M^n\ar[r]\ar@{^{(}->}[u] & F^1M^{n+1}\ar[r]\ar@{^{(}->}[u] & \cdots \\
\cdots\ar[r] & F^2M^{n-1}\ar[r]\ar@{^{(}->}[u] & F^2M^n\ar[r]\ar@{^{(}->}[u] & F^2M^{n+1}\ar[r]\ar@{^{(}->}[u] & \cdots \\
& \vdots\ar@{^{(}->}[u] & \vdots\ar@{^{(}->}[u] & \vdots\ar@{^{(}->}[u] & &} }
\end{equation*}
A morphism of filtered complexes is a cochain map preserving filtrations\footnote{There are more general notions of filtrations in the literature, but we are not caring about them in this paper.}, so denote by $\mathfrak{FdgMod}_R$ the category made of filtered $R$-modules in complexes and their morphisms. \\
The category $\mathfrak{FdgMod}_R$ is both complete and cocomplete: as a matter of fact let $\left(M_{\alpha},F_{\alpha}\right)_{\alpha\in I}$ and $\left(N_{\beta},F_{\beta}\right)$ be respectively an inverse system and a direct system in $\mathfrak{FdgMod}_R$: we have that
\begin{eqnarray*}
\underset{\underset{\alpha}{\longrightarrow}}{\mathrm{lim}}\left(M_{\alpha},F_{\alpha}\right)=\left(M,F\right)\text{ where }F^kM:=\underset{\underset{\alpha}{\longrightarrow}}{\mathrm{lim}}\,F_{\alpha}^kM_{\alpha} \\
\underset{\underset{\beta}{\longleftarrow}}{\mathrm{lim}}\left(N_{\beta},F_{\beta}\right)=\left(N,F\right)\text{ where }F^kN:=\underset{\underset{\beta}{\longleftarrow}}{\mathrm{lim}}\,F_{\beta}^kN_{\beta}.
\end{eqnarray*}
In particular the filtered complex $\left(0,T\right)$, where $0$ is the zero cochain complex and $T$ is the trivial filtration over it, is the zero object of the category $\mathfrak{FdgMod}_R$. \\
Define the filtered complexes
\begin{eqnarray*}
\left(D_R\left(n,p\right),F\right)\quad\text{where}\quad F^kD_R\left(n,p\right):=\begin{cases}
                                                                    D_R\left(n\right) & \text{if }k\leq p\\
                                                                    0 & \text{otherwise}
                                                                    \end{cases}
\\
\left(S_R\left(n,p\right),F\right)\quad\text{where}\quad F^kS_R\left(n,p\right):=\begin{cases}
                                                                    S_R\left(n\right) & \text{if }k\leq p\\
                                                                    0 & \text{otherwise.}
                                                                    \end{cases}\;
\end{eqnarray*}
\begin{rem}
Observe that $\left(D_R\left(n,p\right),F\right)$ and $\left(S_R\left(n,p\right),F\right)$ are compact for all $n$ and all $p$.
\end{rem}
In the following we will sometimes drop explicit references to filtrations if the context makes them clear.
\begin{thm} \label{fdg model thm}
Consider the sets
\begin{eqnarray} \label{fdg model}
&I_{\mathfrak{FdgMod}_R}:=\left\{S_R\left(n+1,p\right)\rightarrow D_R\left(n,p\right)\right\}_{n\in\mathbb Z}\quad\,&\nonumber \\
&J_{\mathfrak{FdgMod}_R}:=\left\{0\rightarrow D_R\left(n,p\right)\right\}_{n\in\mathbb Z}\qquad\qquad\quad& \nonumber \\
&W_{\mathfrak{FdgMod}_R}:=\left\{f:\left(M,F\right)\rightarrow \left(N,F\right)|H^n\left(F^pf\right)\text{ }\mathrm{is}\text{ }\mathrm{an}\text{ }\mathrm{isomorphism}\text{ }\forall n\in\mathbb Z,\forall p\in\mathbb N\right\}.&
\end{eqnarray}
The classes \eqref{fdg model} define a cofibrantly generated model structure on $\mathfrak{FdgMod}_{R}$, where $I_{\mathfrak{FdgMod}_{R}}$ is the set of generating cofibrations, $J_{\mathfrak{FdgMod}_{R}}$ is the set of generating trivial cofibrations and $W_{\mathfrak{FdgMod}_{R}}$ is the set of weak equivalences. 
\end{thm}
As done in Section 1.1, proving Theorem \ref{fdg model thm} amounts to provide a precise description of fibrations, trivial fibrations, cofibrations and trivial cofibrations determined by the sets \eqref{fdg model}, which we do in the following propositions.
\begin{prop} \label{FDG fib}
$p\in\mathrm{Hom}_{\mathfrak{FdgMod}_R}\left(\left(M,F\right),\left(N,F\right)\right)$ is a fibration if and only if $F^kp^n$ is surjective $\forall k\in\mathbb N,\forall n\in\mathbb Z$.
\end{prop}
\begin{proof}
We want to characterise diagrams 
\begin{equation} \label{fdg D^n lift}
\tiny{\xymatrix{0\ar[r]\ar[d] & \left(M,F\right)\ar[d]^p \\
D_R\left(n,p\right)\ar[r] & \left(N,F\right)}}
\end{equation}
in $\mathfrak{FdgMod}_R$ admitting a lifting. A diagram like \eqref{fdg D^n lift} corresponds to the sequence of diagrams in $\mathfrak{dgMod}_R$
\begin{equation} \label{fdg D^n sequence}
\tiny{\xymatrix{0\ar[r]\ar[d] & M\ar[d]^p \\
D_R\left(n\right)\ar[r] & N} \quad
\xymatrix{0\ar[r]\ar[d] & F^1M\ar[d]^{F^1p} \\
D_R\left(n\right)\ar[r] & F^1N} \qquad \cdots \qquad
\xymatrix{0\ar[r]\ar[d] & F^p M\ar[d]^{F^pp} \\
D_R\left(n\right)\ar[r] & F^p N} \quad
\xymatrix{0\ar[r]\ar[d] & F^{p+1}M\ar[d]^{F^{p+1}p} \\
0\ar[r] & F^{p+1}N} }
\end{equation}
and -- as we did in the proof of Proposition \ref{DG fib} -- we see that the sequence \eqref{fdg D^n sequence} corresponds bijectively to an element $\left(x^0,x^1,\ldots,x^p\right)\in N^n\times F^1N^n\times\cdots F^pN^n$, where $x^p\in F^pN^n$ determines $x^k\in F^kN^n$ for all $k\leq p$ through the inclusion maps defining the filtration $F$. Again, Proposition \ref{DG fib} ensures that a diagram like \eqref{fdg D^n lift} admits a lifting if and only if maps $F^kp^n$ are surjective $\forall k\leq p$ if and only if the map $F^pp^n$ is surjective (as observed above, what happens in level $p$ determines the picture in lower levels), so the result follows letting $n$ and $p$ vary.
\end{proof}
\begin{prop} \label{FDG tr fib}
$p\in\mathrm{Hom}_{\mathfrak{DG}_R}\left(\left(M,F\right),\left(N,F\right)\right)$ is a trivial fibration if and only if $F^kp$ is degreewise surjective with acyclic kernel; in particular $W_{\mathfrak{FdgMod}_R}\cap J_{\mathfrak{FdgMod}_R}\text{-inj}=I_{\mathfrak{FdgMod}_R}\text{-inj}$.
\end{prop}
\begin{proof}
We want to characterise diagrams 
\begin{equation} \label{fdg S^n lift}
\tiny{\xymatrix{S_R\left(n+1,p\right)\ar[r]\ar[d] & \left(M,F\right)\ar[d]^p \\
D_R\left(n,p\right)\ar[r] & \left(N,F\right)} }
\end{equation}
in $\mathfrak{FdgMod}_R$ admitting a lifting. A diagram like \eqref{fdg D^n lift} corresponds to the sequence of diagrams in $\mathfrak{dgMod}_R$
\begin{equation} \label{fdg S^n sequence}
\tiny{\xymatrix{S_R\left(n+1\right)\ar[r]\ar[d] & M\ar[d]^p \\
D_R\left(n\right)\ar[r] & N} \quad
\xymatrix{S_R\left(n+1\right)\ar[r]\ar[d] & F^1M\ar[d]^{F^1p} \\
D_R\left(n\right)\ar[r] & F^1N} \quad \cdots \quad
\xymatrix{S_R\left(n+1\right)\ar[r]\ar[d] & F^p M\ar[d]^{F^pp} \\
D_R\left(n\right)\ar[r] & F^p N} \quad
\xymatrix{0\ar[r]\ar[d] & F^{p+1}M\ar[d]^{F^{p+1}p} \\
0\ar[r] & F^{p+1}N} }
\end{equation}
and -- as we did in the proof of Proposition \ref{DG tr fib} -- we see that the sequence \eqref{fdg S^n sequence} corresponds bijectively to an element $\left(\left(x_0,y_0\right),\left(x_1,y_1\right),\ldots,\left(x_p,y_p\right)\right)\in X_0\times X_1\times\cdots\times X_p$, where 
\begin{equation*}
X_k:=\left\{\left(x_k,y_k\right)\in F^kN^n\oplus F^kZ^{n+1}\left(M\right)|F^kp^{n+1}\left(y_k\right)=F^kd^n\left(x_k\right)\right\}
\end{equation*}
and moreover the pair $\left(x_p,y_p\right)$ determines all the previous ones through the inclusion maps defining the filtration $F$. Now by Proposition \ref{DG tr fib} a diagram like \eqref{fdg S^n lift} admits a lifting if and only if $F^pp$ is degreewise surjective and induces an isomorphism in cohomology, thus the result follows letting $n$ and $p$ vary.
\end{proof}
As done in Section 1.2, we study cofibrant objects defined by the structure \eqref{fdg model}.
\begin{prop} \label{fdg cofibrant}
Let $\left(A,F\right)$ be a filtered complex of $R$-modules. If $\left(A,F\right)$ is cofibrant then $F^kA^n$ is a projective $R$-module $\forall n\in\mathbb Z,\forall k\in\mathbb N$; conversely if $F^kA$ is cofibrant as an object in $\mathfrak{dgMod}_R$ and the filtration $F$ is bounded above then $\left(A,F\right)$ is cofibrant.
\end{prop}
\begin{proof}
Suppose $\left(A,F\right)$ is cofibrant and consider a trivial fibration $p\in\mathrm{Hom}_{\mathfrak{FdgMod}_R}\left(\left(M,F\right),\left(N,F\right)\right)$ and any morphism $g\in\mathrm{Hom}_{\mathfrak{FdgMod}_R}\left(\left(A,F\right),\left(N,F\right)\right)$. By assumption, there exists a morphism $h$ lifting $g$, so the diagram
\begin{equation*}
\tiny{\xymatrix{ & \left(M,F\right)\ar[d]^p \\
\left(A,F\right)\ar[r]^g\ar[ur]^h & \left(N,F\right)} }
\end{equation*}
commutes. In particular this means that the big diagram 
\begin{equation*}
\tiny{\xymatrix{ & & & M\ar[ddd]^p \\ 
& & F^1M\ar[d]^{F^1p}\ar@{^{(}->}[ur]\ar@{}[dl] |{\cdots\;\cdots} &\\
& F^1 A\ar[r]^{F^1g}\ar@/^1pc/[ur]^{F^1h}\ar@{^{(}->}[dl] & F^1N\ar@{^{(}->}[dr] &\\
A\ar[rrr]^g\ar@/^3pc/[uuurrr]^h & & & N}}
\end{equation*}
in $\mathfrak{dgMod}_R$ commutes; now it suffices to apply Proposition \ref{DG cofibrant} to show that $F^kA^n$ is a projective $R$-module $\forall n\in\mathbb Z,\forall k\in\mathbb N$.\\
Now assume that $F^kA$ is a cofibrant cochain complex (which in particular implies that $F^kA^n$ is a projective $R$-module $\forall n\in\mathbb Z$ by Proposition \ref{DG cofibrant}) and $F$ is bounded above: we want to prove that $\left(A,F\right)$ is cofibrant as a filtered module in complexes. Let $p\in\mathrm{Hom}_{\mathfrak{FdgMod}_R}\left(\left(M,F\right),\left(N,F\right)\right)$ be a trivial fibration and pick a morphism $g\in\mathrm{Hom}_{\mathfrak{FdgMod}_R}\left(A,M\right)$: we want to show that there is a morphism $h$ lifting $g$. By reverse induction, assume that $F^ph:F^pA\rightarrow F^pM$ has been defined for all $p\geq k$ (the boundedness of $F$ ensures that we can get started): we want to construct a lifting in level $k-1$. Consider the diagram 
\begin{equation*}
\tiny{\xymatrix{ & & & F^{k-1}M\ar[ddd]^p \\ 
& & F^kM\ar[d]^{F^1p}\ar@{^{(}->}[ur] &\\
& F^k A\ar[r]^{F^kg}\ar@/^1pc/[ur]^{F^kh}\ar@{^{(}->}[dl] & F^kN\ar@{^{(}->}[dr] &\\
F^{k-1}A\ar[rrr]^g\ar@{-->}@/^3pc/[uuurrr]^f & & & F^{k-1}N}}
\end{equation*}
and observe that a lifting $f\in\mathrm{Hom}_{\mathfrak{dgMod}_R}\left(F^{k-1}A,F^{k-1}M\right)$ does exist because $F^{k-1}A$ is cofibrant as an object in $\mathfrak{dgMod}_R$; moreover since $F^{k-1}A$ is projective in each degree we are allowed to choose $f$ such that $f|_{F^kA}=F^kh$, thus the result follows.
\end{proof}
\begin{rem}
The assumption on the filtration in Proposition \ref{fdg cofibrant} is probably too strong: it can be substituted with any hypothesis giving the base of the above inductive argument.
\end{rem}
\begin{prop} \label{FDG tr cofib}
There is an inclusion $J_{\mathfrak{FdgMod}_R}\text{-cof}\subseteq W_{\mathfrak{FdgMod}_R}\cap I_{\mathfrak{FdgMod}_R}\text{-cof}$.
\end{prop}
\begin{proof}
Suppose $i\in\mathrm{Hom}_{\mathfrak{FdgMod}_R}\left(\left(M,F\right),\left(N,F\right)\right)$ is a $J_{\mathfrak{FdgMod}_R}\text{-cofibration}$, i.e. it has the left lifting property with respect to fibrations; in particular it lies in $I_{\mathfrak{FdgMod}_R}\text{-cof}$, so we only need to prove that $H^n\left(F^ki\right)$ is an isomorphism $\forall n\in\mathbb Z$, $\forall k\in\mathbb N$. Let $p\in\mathrm{Hom}_{\mathfrak{FdgMod}_R}\left(\left(X,F\right),\left(Y,F\right)\right)$ be any fibration, so by Proposition \ref{FDG fib} $F^kp^n$ is surjective $\forall n\in\mathbb Z$, $\forall k\in\mathbb N$: by assumption the diagram 
\begin{equation*}
\tiny{\xymatrix{\left(M,F\right)\ar[r]\ar[d]_i & \left(X,F\right)\ar[d]^p \\
\left(N,F\right)\ar[r] & \left(Y,F\right)}}
\end{equation*}
admits a lifting and, unfolding it, we get that the diagram in $\mathfrak{dgMod}_R$
\begin{equation*}
\tiny{\xymatrix{F^kM\ar[r]\ar[d]_{F^ki} & X\ar[d]^{F^kp} \\
F^kN\ar[r] & F^kY}}
\end{equation*}
lifts as well. Letting $p$ vary among all fibrations in $\mathfrak{FdgMod}_R$ we see that $F^ki$ has the right lifting property with respect to all degreewise surjections in $\mathfrak{dgMod}_R$, so by Proposition \ref{DG fib} and Proposition \ref{DG tr cofib} it is a trivial cofibration in $\mathfrak{dgMod}_R$; in particular this means that $H^n\left(F^ki\right)$ is an isomorphism $\forall n\in\mathbb Z$, $\forall k\in\mathbb N$, so the result follows.
\end{proof}
\begin{prop} \label{FDG equiv}
The set $W_{\mathfrak{FdgMod}_R}$ has the $2$-out-of-$3$ property and is closed under retracts.
\end{prop}
\begin{proof}
The result follows immediately by applying Proposition \ref{DG equiv} levelwise in the filtration.
\end{proof}
The above results (especially Proposition \ref{FDG tr fib}, Proposition \ref{FDG tr cofib} and Proposition \ref{FDG equiv}) say that the category $\mathfrak{FdgMod}_R$ endowed with the structure \eqref{fdg model} fits into the hypotheses of Theorem \ref{criterion model}, so Theorem \ref{fdg model thm} has been proved. 
\begin{rem}
We have not provided a complete description of cofibrations as this is not really needed in order to establish that data \eqref{fdg model} endow $\mathfrak{FdgMod}_R$ with a model structure; clearly all morphism $f:\left(M,F\right)\rightarrow\left(N,F\right)$ for which $F^kf:F^kM\rightarrow F^kN$ is a cofibration in $\mathfrak{dgMod}_R$ for all $k$ are cofibrations for such model structure, but it is not clear (nor expected) that these are all of them. Actually we believe that a careful characterisation of cofibrations should be quite complicated.
\end{rem}
Now assume $R$ is a $k$-algebra, where $k$ is a field of characteristic $0$: we now endow $\mathfrak{FdgMod}_R$ with the structure of a simplicially enriched category.\\
For all $\left(M,F\right),\left(N,F\right)\in\mathfrak{FdgMod}_R$ consider the chain complex $\left(\mathrm{HOM}\left(\left(M,F\right),\left(N,F\right)\right),\delta\right)$ defined as
\begin{eqnarray} \label{dg stru FDG}
&\mathrm{HOM}_{\mathfrak{FdgMod}_R}\left(\left(M,F\right),\left(N,F\right)\right)_n:=\mathrm{Hom}\left(\left(M,F\right),\left(N[-n],F\right)\right)& \nonumber \\
&\forall\left(f,F\right)\in {\mathrm{HOM}_{\mathfrak{FdgMod}_R}\left(\left(M,F\right),\left(N,F\right)\right)_n}\;\;\delta_n\left(\left(f,F\right)\right)\in {\mathrm{HOM}_{\mathfrak{FdgMod}_R}\left(\left(M,F\right),\left(N,F\right)\right)_{n-1}}& \nonumber\\
&\text{defined by } F^p\left(\delta_n\left(\left(f,F\right)\right)\right):=F^p\bar d^n\circ F^pf-\left(-1\right)^nF^pf\circ F^pd^n&
\end{eqnarray} 
where, by a slight abuse of notation, we mean that $F^pM\left[k\right]:=\left(F^pM\right)\left[k\right]$. \\
Formulae \eqref{dg stru FDG} make $\mathfrak{FdgMod}_R$ into a differential graded category over $k$, so we can naturally endow it with a simplicial structure by taking denormalisation, i.e. by setting
\begin{equation} \label{simpl FDG}
\underline{\mathrm{Hom}}_{\mathfrak{FdgMod}_R}\left(\left(M,F\right),\left(N,F\right)\right):=\mathbf K\left(\tau_{\geq 0}\mathrm{HOM}_{\mathfrak{FdgMod}_R}\left(\left(M,F\right),\left(N,F\right)\right)\right).
\end{equation}
\subsection{The Rees Functor}
Let $R$ be a commutative unital ring; the model structure over $\mathfrak{FdgMod}_R$ given by Theorem \ref{fdg model thm} is really modelled on the unfiltered situation: unsurprisingly, the homotopy theories of filtered modules in complexes and unfiltered ones are closely related, and the functor connecting them is given by the classical Rees construction. \\
Recall that the \emph{Rees module} associated to a filtered $R$-module $\left(M,F\right)$ is defined to be the graded $R\left[t\right]$-module given by
\begin{equation} \label{Rees defn}
\mathrm{Rees}\left(\left(M,F\right)\right):=\underset{p=0}{\overset{\infty}{\bigoplus}}F^pM\cdot t^{-p}
\end{equation}
so the Rees construction transforms filtrations into grading with respect to the polynomial algebra $R\left[t\right]$. Also, it is quite evident from formula \eqref{Rees defn} that the Rees construction is functorial, so there is a functor
\begin{equation*}
\mathrm{Rees}:\mathfrak{FMod}_R\longrightarrow\mathfrak{gMod}_{R\left[t\right]}\footnote{There is some abuse of notation in this formula.}
\end{equation*}
at our disposal, which in turn induces a functor
\begin{equation} \label{dg Rees}
\mathrm{Rees}:\mathfrak{FdgMod}_R\longrightarrow\mathfrak{gdgMod}_{R\left[t\right]}\footnote{There is some abuse of notation in this formula.}
\end{equation}
to the category of graded dg-modules over $R\left[t\right]$; in particular we like to view the latter as the category $\mathbb G_m$-$\mathfrak{dgMod}_{R\left[t\right]}$ of $R\left[t\right]$-modules in complexes equipped with an extra action of the multiplicative group compatible with the canonical action
\begin{eqnarray} \label{Gm action A1}
\mathbb G_m\times\mathbb A^1_R&\xrightarrow{\hspace*{1cm}}&\;\;\mathbb A^1_R \nonumber \\
\left(\lambda,s\right)\quad&\longmapsto&\lambda^{-1}s
\end{eqnarray}
The projective model structure on $\mathfrak{dgMod}_{R\left[t\right]}$ admits a natural $\mathbb G_m$-equivariant version. 
\begin{thm} \label{Gm dg model thm}
Consider the sets
\begin{eqnarray} \label{Gm dg model}
&I_{\text{$\mathbb G_m$-$\mathfrak{dgMod}_{R\left[t\right]}$}}:=\left\{f:t^iS_{R\left[t\right]}\left(n+1\right)\rightarrow t^iD_{R\left[t\right]}\left(n\right)\right\}_{i,n\in\mathbb Z}\quad\,&\nonumber \\
&J_{\text{$\mathbb G_m$-$\mathfrak{dgMod}_{R\left[t\right]}$}}:=\left\{f:0\rightarrow t^iD_{R\left[t\right]}\left(n\right)\right\}_{i,n\in\mathbb Z}\qquad\;& \nonumber \\
&W_{\text{$\mathbb G_m$-$\mathfrak{dgMod}_{R\left[t\right]}$}}:=\left\{f:M\rightarrow N|f \text{ is a $\mathbb G_m$-equivariant quasi-isomorphism}\right\}.&
\end{eqnarray}
The classes \eqref{Gm dg model} determine a cofibrantly generated model structure over $\mathbb G_m$-$\mathfrak{dgMod}_{R\left[t\right]}$, in which $I_{\text{$\mathbb G_m$-$\mathfrak{dgMod}_{R\left[t\right]}$}}$ is the set of generating cofibrations, $J_{\text{$\mathbb G_m$-$\mathfrak{dgMod}_{R\left[t\right]}$}}$ is the set of generating trivial cofibrations and $W_{\text{$\mathbb G_m$-$\mathfrak{dgMod}_{R\left[t\right]}$}}$ is the set of weak equivalences. 
\end{thm}
Arguments and lemmas discussed in Section 1.2 to prove Theorem \ref{dg model} carry over to this context once we restrict to $\mathbb G_m$-equivariant objects and maps. 
\begin{rem}
Notice that maps in $I_{\text{$\mathbb G_m$-$\mathfrak{dgMod}_{R\left[t\right]}$}}$ and $J_{\text{$\mathbb G_m$-$\mathfrak{dgMod}_{R\left[t\right]}$}}$ are $\mathbb G_m$-equivariant, therefore all cofibrations are $\mathbb G_m$-equivariant.
\end{rem}
Fibrations in the model structure determined by Theorem \ref{Gm dg model} are very nicely described: this is the content of the next proposition.
\begin{prop} \label{Gm DG fib}
$p\in\mathrm{Hom}_{\text{$\mathbb G_m$-$\mathfrak{dgMod}_{R\left[t\right]}$}}\left(M,N\right)$ is a fibration if and only if it is a $\mathbb G_m$-equivariant degreewise surjection.
\end{prop}
\begin{proof}
The proof of Proposition \ref{DG fib} adapts to the $\mathbb G_m$-equivariant context. 
\end{proof}
The following result collects various properties of functor \eqref{dg Rees}: all claims are well-known, we only state them in homotopy-theoretical terms.
\begin{thm} \label{Rprops}
The Rees functor 
\begin{equation*}
\mathrm{Rees}:\mathfrak{FdgMod}_R\longrightarrow\mathbb G_m\text{-}\mathfrak{dgMod}_{R\left[t\right]}.
\end{equation*}
has the following properties:
\begin{enumerate}
\item it has a left adjoint functor, given by 
\begin{eqnarray*}
&\varphi:\mathbb G_m\text{-}\mathfrak{dgMod}_{R\left[t\right]}\xrightarrow{\hspace*{0.7cm}}\mathfrak{FdgMod}_R\qquad\qquad\qquad\qquad\qquad& \\
&M\xmapsto{\hspace*{1.5cm}}\left(M_{\varphi},F_{\varphi}\right)\qquad\qquad\qquad& \\
&\qquad\qquad\qquad M_{\varphi}:=\nicefrac{M}{\left(1-t\right)M}& \\
&\qquad\qquad\qquad\qquad\qquad\quad F^n M_{\varphi}:=\mathrm{Im}\left(M^{\bullet,n}\rightarrow M^{\bullet}_{\varphi}\right)&
\end{eqnarray*}
where the complex $M$ is seen as a bigraded $R\left[t\right]$-module;
\item for all pairs $\left(M,F\right)$, $\left(N,F\right)$ there is bijection
\begin{equation} \label{G_m invariance}
\mathrm{Hom}_R\left(\left(M,F\right),\left(N,F\right)\right)\simeq\mathrm{Hom}_{R\left[t\right]}\left(\mathrm{Rees}\left(\left(M,F\right)\right),\mathrm{Rees}\left(\left(N,F\right)\right)\right)^{\mathbb G_m}\footnote{Here $\mathrm{Hom}_{R\left[t\right]}\left(\mathrm{Rees}\left(\left(M,F\right)\right),\mathrm{Rees}\left(\left(N,F\right)\right)\right)^{\mathbb G_m}$ stands for the set of $\mathbb G_m$-equivariant morphisms of $R\left[t\right]$-modules in complexes between $\mathrm{Rees}\left(\left(M,F\right)\right)$ and $\mathrm{Rees}\left(\left(N,F\right)\right)$.}
\end{equation}
which is natural in all variables;
\item its essential image consists of the full subcategory of $t$-torsion-free $R\left[t\right]$-modules in complexes;
\item it induces an equivalence on the homotopy categories;
\item it preserves compact objects;\footnote{In particular this means that the Rees functor maps filtered perfect complexes to perfect complexes: we will be more precise about this in Section 2.2 and Section 2.3.}
\item it maps fibrations to fibrations.
\end{enumerate}
In particular the Rees construction provides a Quillen equivalence between the categories $\mathfrak{FdgMod}_R$ and $\mathbb G_m\text{-}\mathfrak{dgMod}_{R\left[t\right]}$, both endowed with the projective model structure.
\end{thm}
\begin{proof}
We give references for most of the claims enunciated: the language we are using might be somehow different from the one therein, but the results and arguments we quote definitely apply to our statements.
\begin{enumerate}
\item Claim $\left(1\right)$ follows from \cite{Hin2} Section 4.3: more specifically it is given by Comment 4.3.3;
\item Claim $\left(2\right)$ follows from \cite{Pr2} Lemma 1.6;
\item Claim $\left(3\right)$ follows from \cite{Hin2} Section 4.3: more specifically it is given again by Comment 4.3.3;
\item Claim $\left(4\right)$ follows from \cite{SchaSchn} Theorem 3.16 and \cite{SchaSchn} Theorem 4.20: for a naiver explanation see \cite{BrNT} Section 3.1;
\item Claim $\left(5\right)$ follows from \cite{BrNT} Section 3.1
\item In order to prove Claim $\left(6\right)$, let $f:\left(M,F\right)\rightarrow\left(N,F\right)$ be a fibration in $\mathfrak{FdgMod}_R$, i.e. by Proposition \ref{FDG fib} assume that $F^kf:F^kM\rightarrow F^kN$ is degreewise surjective for all $k\in\mathbb N$; this in turn implies that
\begin{eqnarray*}
\mathrm{Rees}\left(f\right):\qquad\;\mathrm{Rees}\left(\left(M,F\right)\right)\qquad\quad&\rightarrow&\qquad\qquad\qquad\mathrm{Rees}\left(\left(N,F\right)\right) \\
\underline x^0\oplus\underline x^1\cdot t^{-1}\oplus\underline x^2\cdot t^{-2}\oplus\cdots&\mapsto&\underline f\left(x^0\right)\oplus F^1f\left(\underline x^1\right)\cdot t^{-1}\oplus F^2f\left(\underline x^2\right)\cdot t^{-2}\oplus\cdots
\end{eqnarray*}
is degreewise surjective as a map of $\mathbb G_m$-equivariant $R\left[t\right]$-modules in complexes, thus the statement follows because of Proposition \ref{Gm DG fib}.
\end{enumerate}
In particular Claim $\left(1\right)$, Claim $\left(4\right)$ and Claim $\left(6\right)$ can be rephrased by saying that 
\begin{equation*}
\mathrm{Rees}:\mathfrak{FdgMod}_R\longrightarrow\mathbb G_m\text{-}\mathfrak{dgMod}_{R\left[t\right]}.
\end{equation*}
is a right Quillen equivalence.
\end{proof}
\begin{rem}
We can say that the model structure on $\mathfrak{FdgMod}_R$ defined by Theorem \ref{fdg model thm} is precisely the one making the Rees functor into a right Quillen functor; more formally consider the pair given by the Rees functor and its left adjoint described in Theorem \ref{Rprops}.1: than such a pair satisfies the assumption of Theorem \ref{Hirschhorn} and moreover the model structure induced on $\mathfrak{FdgMod}_R$ through the latter criterion is the one determined by Theorem \ref{fdg model thm}.
\end{rem}
\begin{rem}
Relation \eqref{G_m invariance} descends to $\mathrm{Ext}$ groups: $\forall i\in\mathbb Z$, $\forall\left(M,F\right),\left(N,F\right)\in\mathfrak{FdgMod}_R$  we have that
\begin{equation*}
\mathrm{Ext}_R^i\left(\left(M,F\right),\left(N,F\right)\right)\simeq\mathrm{Ext}^i_{R\left[t\right]}\left(\mathrm{Rees}\left(\left(M,F\right)\right),\mathrm{Rees}\left(\left(N,F\right)\right)\right)^{\mathbb G_m}
\end{equation*}
where the object on the left-hand side is the $\mathrm{Ext}$ group in the category $\mathfrak{FdgMod}_R$, i.e.
\begin{equation*}
\mathrm{Ext}_R^{n-i}\left(\left(M,F\right),\left(N,F\right)\right):=\pi_i\underline{\mathrm{Hom}}_{\mathfrak{FdgMod}_R}\left(\left(M,F\right),\left(N\left[-n\right],F\right)\right).
\end{equation*}
\end{rem}

\section{Derived Moduli of Filtered Complexes}

From now on $k$ will always denote a field of characteristic $0$ and $R$ a (possibly differential graded) commutative algebra over $k$; let $X$ be a smooth proper scheme over $k$: the main goal of this chapter is to study derived moduli of filtered perfect complexes of $\mathscr O_X$-modules. In order to do this we will first recall some generalities about representability of derived stacks -- following the work of Lurie and Pridham -- and then we will use these tools to construct derived geometric stacks classifying perfect complexes and filtered perfect complexes. Such stacks are related by a canonical forgetful map: as we will see in the last section of the chapter, the homotopy fibre of this map will help us define a coherent derived version of the Grassmannian.
\subsection{Background on Derived Stacks and Representability}
This section is devoted to collect some miscellaneous background material on derived geometric stacks which will be largely used in the other sections of this chapter: in particular we will review a few representability results -- due to Lurie and Pridham -- giving conditions for a simplicial presheaf on $\mathfrak{dgAlg}^{\leq 0}_R$ to give rise to a (truncated) derived geometric stack. We will assume that the reader is familiar with the notion of derived geometric $n$-stack and the basic tools of Derived Algebraic Geometry as they appear in the work of Lurie, To\" en and Vezzosi: foundational references on this subject include \cite{Lu1}, \cite{Lu2}, \cite{Toe1} and \cite{TVe}; in any case along most of the paper it will be enough to think of a derived geometric stack as a functor $\mathbf F:\mathfrak{dgAlg}^{\leq 0}_R\rightarrow\mathfrak{sSet}$ satisfying hyperdescent and some technical geometricity assumption -- i.e. the existence of some sort of higher atlas -- with respect to affine hypercovers. These two conditions are precisely those turning a completely abstract functor to some kind of {} ``geometric space'', where the usual tools of Algebraic Geometry -- such as  quasi-coherent modules, formalism of the six operations, Intersection Theory -- make sense. Also note that the case of derived schemes is much easier to figure out: as a matter of fact by \cite{Pr4} Theorem 6.42 a derived scheme $\mathcal X$ over $k$ can be seen as a pair $\left(\pi^0\mathcal X,\mathscr O_{\mathcal X,*}\right)$, where $\pi^0\mathcal X$ is an honest $k$-scheme and $\mathscr O_{\mathcal X,*}$ is a presheaf of differential graded commutative algebras in non-positive degrees on the site of affine opens of $\pi^0\mathcal X$ such that:
\begin{itemize}
\item the (cohomology) presheaf $\mathcal H^0\left(\mathscr O_{\mathcal X,*}\right)\simeq\mathscr O_{\pi^0\mathcal X}$;
\item the (cohomology) presheaves $\mathcal H^n\left(\mathscr O_{\mathcal X,*}\right)$ are quasi-coherent $\mathscr O_{\pi^0\mathcal X}$-modules.
\end{itemize}
\begin{war}
Be aware that there are some small differences between the definition of derived geometric stack given in \cite{Lu1} -- which is the one we refer to in this paper -- and the one given in \cite{TVe}: for a comparison see the explanation provided in \cite{Pr4} and \cite{Toe1}.
\end{war}
Now we are to recall representability for derived geometric stacks: all contents herein are adapted from \cite{Pr5} and \cite{Pr1}. \\ 
Recall that a functor $\mathbf F:\mathfrak{dgAlg}^{\leq 0}_R\rightarrow\mathfrak{sSet}$ is said to be \emph{homotopic} or \emph{homotopy-preserving} if it maps quasi-isomorphisms in $\mathfrak{dgAlg}^{\leq 0}_R$ to weak equivalences in $\mathfrak{sSet}$, while it is called \emph{homotopy-homogeneous} if for any morphism $C\rightarrow B$ and any square-zero extension
\begin{equation*}
0\longrightarrow I\longrightarrow A\longrightarrow B\longrightarrow 0
\end{equation*}
in $\mathfrak{dgAlg}^{\leq 0}$ the natural map of simplicial sets
\begin{equation*}
\mathbf F\left(A\times_BC\right)\longrightarrow\mathbf F\left(A\right)\times^h_{\mathbf F\left(B\right)}\mathbf F\left(B\right)\,\footnote{The symbol $-\times_-^h-$ denotes the homotopy fibre product in $\mathfrak{sSet}$.}
\end{equation*}
is a weak equivalence. \\
Let $\mathbf F:\mathfrak{dgAlg}^{\leq 0}_R\rightarrow\mathfrak{sSet}$ be a homotopy-preserving homotopy-homogeneous functor and take a point $x\in\mathbf F\left(A\right)$, where $A\in\mathfrak{dgAlg}_R^{\leq 0}$; recall from \cite{Pr5} that the \emph{tangent space} to $\mathbf F$ at $x$ is defined to be the functor
\begin{eqnarray*}
T_x\mathbf F:&\mathfrak{dgMod}^{\leq 0}_A&\xrightarrow{\hspace*{1.5cm}}\mathfrak{sSet} \\
&M&\longmapsto\mathbf F\left(M\oplus A\right)\times^h_{\mathbf F\left(A\right)}\left\{x\right\}
\end{eqnarray*}
and define for any differential graded $A$-module $M$ and for all $i>0$ the groups
\begin{equation*} \label{coho the}
\mathrm D^{n-i}_x\left(\mathbf F,M\right):=\pi_{i}\left(T_x\mathbf F\left(M\left[-n\right]\right)\right).
\end{equation*}
\begin{prop} \label{coho} \emph{(Pridham)}
In the notations of formula \eqref{coho the} we have that:
\begin{enumerate}
\item $\pi_{i}\left(T_x\mathbf F\left(M\right)\right)\simeq\pi_{i+1}\left(T_x\mathbf F\left(M\left[-1\right]\right)\right)$, so $\mathrm D^j_x\left(\mathbf F,M\right)$ is well-defined for all $m$;
\item $\mathrm D^j_x\left(\mathbf F,M\right)$ is an abelian group and the abelian structure is natural in $M$ and $\mathbf F$;
\item there is a local coefficient system $\mathrm D^*\left(\mathbf F,M\right)$ on $\mathbf F\left(A\right)$ whose stalk at $x$ is $\mathrm D^*_x\left(\mathbf F,M\right)$;
\item for any map $f:A\rightarrow B$ in $\mathfrak{dgAlg}^{\leq 0}_R$ and any $P\in\mathfrak{dgMod}^{\leq 0}_B$ there is a natural isomorphism $\mathrm D^j_x\left(\mathbf F,f_*P\right)\simeq\mathrm D^j_{f_*x}\left(\mathbf F,P\right)$;
\item let 
\begin{equation*}
0\longrightarrow I\overset{e}{\longrightarrow} A\overset{f}{\longrightarrow}B\longrightarrow 0
\end{equation*}
be a square-zero extension in $\mathfrak{dgAlg}^{\leq 0}_R$ and set $y:=f_*x$: there is a long exact sequence of groups and sets
\begin{eqnarray*}
&\cdots\overset{e_*}{\longrightarrow}\pi_n\left(\mathbf F\left(A\right),x\right)\overset{f_*}{\longrightarrow}\pi_n\left(\mathbf F\left(B\right),y\right)\overset{o_e}{\longrightarrow}\mathrm D^{1-n}_x\left(\mathbf F,I\right)\overset{e_*}{\longrightarrow}\pi_{n-1}\left(\mathbf F\left(A\right),x\right)\overset{f_*}{\longrightarrow}\cdots& \\
&\cdots\overset{f^*}{\longrightarrow}\pi_1\left(\mathbf F\left(B\right),y\right)\overset{o_e}{\longrightarrow}\mathrm D^0_x\left(\mathbf F,I\right)\overset{-*x}{\longrightarrow}\pi_0\left(\mathbf F\left(A\right)\right)\overset{f_*}{\longrightarrow}\pi_0\left(\mathbf F\left(B\right)\right)\overset{o_e}{\longrightarrow}\Gamma\left(\mathbf F\left(B\right),\mathrm D^1\left(\mathbf F,I\right)\right).&
\end{eqnarray*}
\end{enumerate}
\end{prop}
\begin{proof}
Claim 1 and Claim 2 correspond to \cite{Pr5} Lemma 1.12, Claim 3 to \cite{Pr5} Lemma 1.16, Claim 4 to \cite{Pr5} Lemma 1.15 and Claim 5 to \cite{Pr5} Proposition 1.17. 
\end{proof}
\begin{rem}
Proposition \ref{coho} says that the sequence of abelian groups $\mathrm D^*_x\left(\mathbf F,M\right)$ should be thought morally as some sort of pointwise cohomology theory for the functor $\mathbf F$; such a statement is actually true -- in a rigorous mathematical sense -- whenever $\mathbf F$ is a derived geometric $n$-stack over $R$ and $x:\mathbb R\mathrm{Spec}\left(A\right)\rightarrow\mathbf F$ is a point on it: as a matter of fact in this case
\begin{equation*}
\mathrm D^j_x\left(\mathbf F,M\right)=\mathrm{Ext}^j_A\left(x^*\mathbb L^{\mathbf F/R},M\right).
\end{equation*}
\end{rem}
At last, recall that a simplicial presheaf on $\mathfrak{dgAlg}^{\leq 0}_R$ is said to be \emph{nilcomplete} if for all $A\in\mathfrak{dgAlg}^{\leq 0}_R$ the natural map 
\begin{equation*}
\mathbf F\left(A\right)\longrightarrow\underset{\underset{r}{\longleftarrow}}{\mathrm{holim}}\,\mathbf F\left(P^rA\right)
\end{equation*}
is a weak equivalence, where $\left\{P^rA\right\}_{r>0}$ stands for the Moore-Postnikov tower of $A$ (see \cite{GJ} for a definition). \\
Now we are ready to state \emph{Lurie-Pridham Representability Theorem} for derived geometric stacks.
\begin{thm} \label{L-P Rep} \emph{(Lurie, Pridham)}
A functor $\mathbf F:\mathfrak{dgAlg}^{\leq 0}_R\rightarrow\mathfrak{sSet}$ is a derived geometric $n$-stack almost of finite presentation if and only if the following conditions hold:
\begin{enumerate}
\item $\mathbf F$ is $n$-truncated;
\item $\mathbf F$ is homotopy-preserving;
\item $\mathbf F$ is homotopy-homogeneous;
\item $\mathbf F$ is nilcomplete;
\item $\pi^0\mathbf F$ is a hypersheaf (for the \' etale topology);
\item $\pi^0\mathbf F$ preserves filtered colimits;
\item for finitely generated integral domains $A\in H^0\left(R\right)$ and all $x\in\mathbf F\left(A\right)$, the groups $\mathrm D^j_x\left(\mathbf F,A\right)$ are finitely generated $A$-modules;
\item for finitely generated integral domains $A\in H^0\left(R\right)$, all $x\in\mathbf F\left(A\right)$ and all \' etale morphisms $f:A\rightarrow A'$, the maps
\begin{equation*}
\mathrm D^*_x\left(\mathbf F,A\right)\otimes_AA'\longrightarrow\mathrm D^*_{f_*x}\left(\mathbf F,A'\right)
\end{equation*}
are isomorphisms;
\item for all finitely generated integral domains $A\in\mathfrak{Alg}_{H^0\left(R\right)}$ and all $x\in\mathbf F\left(A\right)$ the functors $\mathrm D^j\left(\mathbf F,-\right)$ preserve filtered colimits for all $j>0$;
\item for all complete discrete local Noetherian $H^0\left(R\right)$-algebras $A$ the map
\begin{equation*}
\mathbf F\left(A\right)\longrightarrow\underset{\underset{r}{\longleftarrow}}{\mathrm{holim}}\,\mathbf F\left(\nicefrac{A}{\mathfrak m_A^r}\right)
\end{equation*}
is a weak equivalence.
\end{enumerate}
\end{thm}
\begin{proof}
See \cite{Pr5} Corollary 1.36 and lemmas therewith, which largely rely on \cite{Lu1} Theorem 7.5.1.
\end{proof}
\begin{rem} \label{expl}
As we have already mentioned, a derived geometric $n$-stack roughly corresponds to a $n$-truncated homotopy-preserving simplicial presheaf on $\mathfrak{dgAlg}^{\leq 0}_R$ which is a hypersheaf for the (homotopy) \' etale topology and which is obtained from an affine hypercover by taking successive smooth quotients. Theorem \ref{L-P Rep} says that in order to ensure that some given homotopy-homogeneous functor $\mathbf F:\mathfrak{dgAlg}^{\leq 0}_R\rightarrow\mathfrak{sSet}$ is a derived geometric stack it suffices to verify that its underived truncation $\pi^0\mathbf F:\mathfrak{Alg}_{H^0\left(R\right)}\rightarrow\mathfrak{sSet}$ is a $n$-truncated stack (in the sense of \cite{HirSi} and \cite{Sim1}) and that for all $x\in\mathbf F\left(A\right)$ the cohomology theories $\mathrm D^*_x\left(\mathbf F,-\right)$ satisfy some mild finiteness conditions. 
\end{rem}
The most technical assumption in Theorem \ref{L-P Rep} is probably Condition (4), i.e. nilcompleteness: this is actually avoided when working with nilpotent algebras. Consider the full subcategory $\mathfrak{dg}_b\mathfrak{Nil}^{\leq 0}_R$ of $\mathfrak{dgAlg}^{\leq 0}_R$ made of bounded below differential graded commutative $R$-algebras in non-positive degrees such that the canonical map $A\rightarrow H^0\left(A\right)$ is nilpotent: the following result is \emph{Pridham Nilpotent Representability Criterion}.
\begin{thm} \label{Nilp Rep} \emph{(Pridham)}
A functor $\mathbf F:\mathfrak{dg}_b\mathfrak{Nil}^{\leq 0}_R\rightarrow\mathfrak{sSet}$ is the restriction of an almost finitely presented derived geometric $n$-stack $\mathcal F:\mathfrak{dgAlg}^{\leq 0}_R\rightarrow\mathfrak{sSet}$ if and only if the following conditions hold:
\begin{enumerate}
\item $\mathbf F$ is $n$-truncated;
\item $\mathbf F$ is homotopy-preserving\footnote{When dealing with functors defined on $\mathfrak{dg}_b\mathfrak{Nil}^{\leq 0}_R$ actually it suffices to check that tiny acyclic extensions are mapped to weak equivalences.};
\item $\mathbf F$ is homotopy-homogeneous;
\item $\pi^0\mathbf F$ is a hypersheaf (for the \' etale topology);
\item $\pi^0\mathbf F$ preserves filtered colimits;
\item for finitely generated integral domains $A\in H^0\left(R\right)$ and all $x\in\mathbf F\left(A\right)$, the groups $\mathrm D^j_x\left(\mathbf F,A\right)$ are finitely generated $A$-modules;
\item for finitely generated integral domains $A\in H^0\left(R\right)$, all $x\in\mathbf F\left(A\right)$ and all \' etale morphisms $f:A\rightarrow A'$, the maps
\begin{equation*}
\mathrm D^*_x\left(F,A\right)\otimes_AA'\longrightarrow\mathrm D^*_{f_*x}\left(\mathbf F,A'\right)
\end{equation*}
are isomorphisms;
\item for all finitely generated integral domains $A\in\mathfrak{Alg}_{H^0\left(R\right)}$ and all $x\in\mathbf F\left(A\right)$ the functors $\mathrm D^j\left(\mathbf F,-\right)$ preserve filtered colimits for all $j>0$;
\item for all complete discrete local Noetherian $H^0\left(R\right)$-algebras $A$ the map
\begin{equation*}
\mathbf F\left(A\right)\longrightarrow\underset{\underset{r}{\longleftarrow}}{\mathrm{holim}}\,\mathbf F\left(\nicefrac{A}{\mathfrak m_A^r}\right)
\end{equation*}
is a weak equivalence.
\end{enumerate}
Moreover $\mathcal F$ is uniquely determined by $\mathbf F$ up to weak equivalence.
\end{thm}
\begin{proof}
See \cite{Pr5} Theorem 2.17.
\end{proof}
In the last part of this section we will recall from \cite{Pr1} a few criteria ensuring homotopicity, homogeneity and underived hyperdescent of a functor $\mathbf F:\mathfrak{dgAlg}_R^{\leq 0}\rightarrow\mathfrak{sSet}$, which from now on will always be thought of as an abstract derived moduli problem. Most definitions and results below will involve $\mathfrak{sCat}$-valued derived moduli functors rather than honest simplicial presheaves on $\mathfrak{dgAlg}_R^{\leq 0}$: the reason for this lies in the fact that it is often easier to tackle a derived moduli problem by considering a suitable $\mathfrak{sCat}$-valued functor $\mathbf F:\mathfrak{dg}_b\mathfrak{Nil}_R^{\leq 0}\rightarrow\mathfrak{sCat}$ and then use Theorem \ref{Nilp Rep} to prove that the diagonal of its simplicial nerve $\mathrm{diag}\left(B\mathbf F\right):\mathfrak{dg}_b\mathfrak{Nil}_R^{\leq 0}\rightarrow\mathfrak{sSet}$ gives rise to a honest truncated derived geometric stack; we will see instances of such a procedure in Section 2.2 and Section 2.3, for more examples see \cite{Pr1} Section 3 and Section 4. Moreover Cegarra and Remedios showed in \cite{CR} that the diagonal of the simplicial nerve is weakly equivalent to the functor $\bar W$ obtained as the right adjoint of Illusie's total d\' ecalage functor (see \cite{GJ} or \cite{Il} for a definition), so we can substitute $\mathrm{diag}\left(B\mathbf F\right)$ with $\bar W\mathbf F$ in the above considerations: for more details see \cite{Pr1}. \\
Let
\begin{equation*}
\mathfrak C\overset{\mathbf F}{\longrightarrow}\mathfrak B\overset{\mathbf G}{\longleftarrow}\mathfrak D
\end{equation*}
be a diagram of simplicial categories; recall that the \emph{2-fibre product} $\mathfrak C\times_{\mathfrak B}^{\left(2\right)}\mathfrak D$ is defined to be the simplicial category for which
\begin{eqnarray}
\mathrm{Ob}\left(\mathfrak C\times_{\mathfrak B}^{\left(2\right)}\mathfrak D\right):=\left\{\left(c,\theta,d\right)\big|c\in\mathfrak C,d\in\mathfrak D,\theta:\mathbf F\left(c\right)\rightarrow\mathbf G\left(d\right)\text{ is an isomorphism in $\mathfrak B_0$}\right\}\quad\nonumber \\
\qquad\underline{\mathrm{Hom}}_{\mathfrak C\times_{\mathfrak B}^{\left(2\right)}\mathfrak D}\left(\left(c_1,\theta_1,d_1\right),\left(c_2,\theta_2,d_2\right)\right):=\left\{\left(f_1,f_2\right)\in\underline{\mathrm{Hom}}_{\mathfrak C}\times\underline{\mathrm{Hom}}_{\mathfrak D}\big|\mathbf Gf_2\circ\theta_1=\theta_2\circ\mathbf Ff_1\right\}. \nonumber
\end{eqnarray}
\begin{defn} \label{2-FIB}
A morphism $\mathbf F:\mathfrak C\rightarrow\mathfrak D$ of simplicial categories is said to be a \emph{$2$-fibration} if the following conditions hold:
\begin{enumerate}
\item $\forall c_1,c_2\in\mathfrak C$, the induced map $\underline{\mathrm{Hom}}_{\mathfrak C}\left(c_1,c_2\right)\rightarrow\underline{\mathrm{Hom}}_{\mathfrak D}\left(\mathbf F\left(c_1\right),\mathbf F\left(c_2\right)\right)$ is a fibration in $\mathfrak{sSet}$;
\item for any $c_1\in\mathfrak C$, $d\in\mathfrak D$ and homotopy equivalence $h:\mathbf F\left(c_1\right)\rightarrow d$ in $\mathfrak C$ there exist $c_2\in\mathfrak C$, a homotopy equivalence $k:c_1\rightarrow c_2$ in $\mathfrak C$ and an isomorphism $\theta:\mathbf F\left(c_2\right)\rightarrow d$ such that $\theta\circ \mathbf Fk=h$. 
\end{enumerate}
\end{defn}
\begin{defn} \label{trivial 2-FIB}
A morphism $\mathbf F:\mathfrak C\rightarrow\mathfrak D$ of simplicial categories is said to be a \emph{trivial $2$-fibration} if the following conditions hold:
\begin{enumerate}
\item $\forall c_1,c_2\in\mathfrak C$, the induced map $\underline{\mathrm{Hom}}_{\mathfrak C}\left(c_1,c_2\right)\rightarrow\underline{\mathrm{Hom}}_{\mathfrak D}\left(\mathbf F\left(c_1\right),\mathbf F\left(c_2\right)\right)$ is a trivial fibration in $\mathfrak{sSet}$;
\item $\mathbf F_0:\mathfrak C_0\rightarrow\mathfrak D_0$ is essentially surjective. 
\end{enumerate}
\end{defn}
\begin{defn}
Fix two functors $\mathbf F,\mathbf G:\mathfrak{dg}_b\mathfrak{Nil}_R^{\leq 0}\rightarrow\mathfrak{sCat}$; a natural transformation $\eta:\mathbf F\rightarrow\mathbf G$ is said to be \emph{$2$-homotopic} if for all tiny acyclic extensions $A\rightarrow B$, the natural map
\begin{equation*}
\mathbf F\left(A\right)\longrightarrow\mathbf F\left(B\right)\times_{\mathbf G\left(B\right)}^{\left(2\right)}\mathbf G\left(A\right)
\end{equation*}
is a trivial $2$-fibration. The functor $\mathbf F$ is said to be $2$-homotopic if so is the morphism $\mathbf F\rightarrow\bullet$.
\end{defn}
\begin{defn}
Fix two functors $\mathbf F,\mathbf G:\mathfrak{dg}_b\mathfrak{Nil}_R^{\leq 0}\rightarrow\mathfrak{sCat}$; a natural transformation $\eta:\mathbf F\rightarrow\mathbf G$ is said to be \emph{formally $2$-quasi-presmooth} if for all square-zero extensions $A\rightarrow B$, the natural map
\begin{equation*}
\mathbf F\left(A\right)\longrightarrow\mathbf F\left(B\right)\times_{\mathbf G\left(B\right)}^{\left(2\right)}\mathbf G\left(A\right)
\end{equation*}
is a $2$-fibration. If $\eta$ is also $2$-homotopic, it is said to be \emph{formally $2$-quasi-smooth}. \\
The functor $\mathbf F$ is said to be formally $2$-quasi-(pre)smooth if so is the morphism $\mathbf F\rightarrow\bullet$.
\end{defn}
\begin{defn}
A functor $\mathbf F:\mathfrak{dg}_b\mathfrak{Nil}_R^{\leq 0}\rightarrow\mathfrak{sCat}$ is said to be \emph{$2$-homogeneous} if for all square-zero extensions $A\rightarrow B$ and all morphisms $C\rightarrow B$ the natural map
\begin{equation*}
\mathbf F\left(A\times_BC\right)\longrightarrow\mathbf F\left(A\right)\times_{\mathbf F\left(B\right)}^{(2)}\mathbf F\left(C\right)
\end{equation*}
is essentially surjective on objects and an isomorphism on $\underline{\mathrm{Hom}}$ spaces.
\end{defn}
Now given a simplicial category $\mathfrak C$, denote by $\mathcal W\left(\mathfrak C\right)$ the full simplicial subcategory of $\mathfrak C$ in which morphisms are maps whose image in $\pi_0\mathfrak C$ is invertible (in particular this means that $\pi_0\mathcal W\left(\mathfrak C\right)$ is the core of $\pi_0\mathfrak C$). Also denote by $c\left(\pi_0\mathfrak C\right)$ the set of isomorphism classes of the (honest) category $\pi_0\mathfrak C$.\\
The following result relates quasi-smoothness to homogeneity and will be very useful in the rest of the paper.
\begin{prop} \label{qsmooth homog}
Let $\mathbf F:\mathfrak{dg}_b\mathfrak{Nil}_R^{\leq 0}\rightarrow\mathfrak{sCat}$ be $2$-homogeneous and formally $2$-quasi-smooth; then
\begin{enumerate}
\item $\mathrm{diag}\left(B\mathbf F\right)$ is homotopy-preserving;
\item $\mathrm{diag}\left(B\mathbf F\right)$ is homotopy-homogeneous;
\item the map $\mathcal W\left(\mathbf F\right)\rightarrow\mathbf F$ is formally \' etale, meaning that for any square-zero extension $A\rightarrow B$ the induced map 
\begin{equation*}
\mathcal W\left(\mathbf F\left(A\right)\right)\longrightarrow\mathbf F\left(A\right)\times_{\mathbf F\left(B\right)}\mathcal W\left(\mathbf F\left(B\right)\right) 
\end{equation*}
is an isomorphism;
\item $\mathcal W\left(\mathbf F\right)$ is $2$-homogeneous and formally $2$-quasi smooth, as well.
\end{enumerate}
\end{prop}
\begin{proof}
See \cite{Pr1}, Section 2.3.
\end{proof}
At last, let us recall for future reference the notions of openness and (homotopy \' etaleness) for a $\mathfrak{sCat}$-valued presheaf.
\begin{defn} \label{ffss}
Fix a presheaf $\mathbf C:\mathfrak{Alg}_{H^0\left(R\right)}\rightarrow\mathfrak{sCat}$ and a subfunctor $\mathbf M\subset\mathbf C$; $M$ is said to be a \emph{functorial full simplicial subcategory} of $\mathbf C$ if $\forall A\in\mathfrak{Alg}_{H^0\left(R\right)},\forall X,Y\in\mathbf M\left(A\right)$, the map 
\begin{equation} \label{hom ffss}
\underline{\mathrm{Hom}}_{\mathbf M\left(A\right)}\left(X,Y\right)\longrightarrow\underline{\mathrm{Hom}}_{\mathbf C\left(A\right)}\left(X,Y\right)
\end{equation}
is a weak equivalence.
\end{defn}
\begin{rem}
In the notations of Definition \ref{ffss}, denote $M:=\bar W\mathcal W\left(\mathbf M\right)$ and $C:=\bar W\mathcal W\left(\mathbf C\right)$; formula \eqref{hom ffss} implies that the induced morphism $M\rightarrow C$ is injective on $\pi_0$ and bijective on all homotopy groups.
\end{rem}
\begin{defn}
Given a functor $\mathbf C:\mathfrak{Alg}_{H^0\left(R\right)}\rightarrow\mathfrak{sCat}$ and a functorial simplicial subcategory $\mathbf M\subset\mathbf C$, say that $\mathbf M$ is an \emph{open} simplicial subcategory of $\mathbf C$ if
\begin{enumerate}
\item $\mathbf M$ is a full simplicial subcategory;
\item the map $\mathbf M\rightarrow \mathbf C$ is \emph{homotopy formally \' etale}, meaning that for any square-zero extension $A\rightarrow B$, the map 
\begin{equation*}
\pi_0\mathbf M\left(A\right)\longrightarrow\pi_0\mathbf C\left(A\right)\times^{\left(2\right)}_{\pi_0\mathbf C\left(B\right)}\pi_0\mathbf M\left(B\right)
\end{equation*}
is essentially surjective on objects.
\end{enumerate}
\end{defn}
\begin{prop} \label{hypersheaf} \emph{(Pridham)}
Let $\mathbf C:\mathfrak{Alg}_{H^0\left(R\right)}\rightarrow\mathfrak{sCat}$ be a functor for which 
\begin{equation*}
\bar W\mathcal W\left(\mathbf C\right):\mathfrak{Alg}_{H^0\left(R\right)}\rightarrow\mathfrak{sSet}
\end{equation*} 
is an \' etale hypersheaf and let $\mathbf M\subset\mathbf C$ be functorial full simplicial subcategory. Then $\bar W\mathcal W\left(\mathbf M\right)$ is an \' etale hypersheaf if and only if for any $A\in\mathfrak{Alg}_{H^0\left(R\right)}$ and any \' etale cover $\left\{f_{\alpha}:A\rightarrow B_{\alpha}\right\}_{\alpha}$ the map
\begin{equation*}
c\left(\pi_0\mathbf M\left(A\right)\right)\longrightarrow c\left(\pi_0\mathbf C\left(A\right)\right)\times_{\underset{\alpha}{\prod}c\left(\pi_0\mathbf C\left(B_{\alpha}\right)\right)}\underset{\alpha}{\prod}c\left(\pi_0\mathbf M\left(B_{\alpha}\right)\right)
\end{equation*}
is surjective.
\end{prop}
\begin{proof}
See \cite{Pr1} Proposition 2.32.
\end{proof}
\subsection{Derived Moduli of Perfect Complexes}
Let $X$ be a smooth proper scheme over $k$ and recall that a complex $\mathcal E$ of $\mathscr O_X$-modules is said to be \emph{perfect} if it is compact as an object in the derived category $\mathrm D\left(X\right)$; in simpler terms $\mathcal E$ is perfect if it is locally quasi-isomorphic to a bounded complex of vector bundles. A key example of perfect complex is given by the derived push-forward of the relative De Rham complex associated to a morphism of schemes: more clearly, if $f:Y\rightarrow Z$ is a proper morphism of (semi-separated quasi-compact) $k$-schemes, then $\mathbb Rf_*\Omega_{Y/Z}$ is perfect as an object in $\mathrm D\left(Z\right)$. Perfect complexes play a very important role in several parts of Algebraic Geometry -- such as Hodge Theory, Deformation Theory, Enumerative Geometry, Symplectic Algebraic Geometry and Homological Mirror Symmetry -- so it is very natural to ask whether they can be classified by some moduli stack; for this reason consider the functor
\begin{eqnarray} \label{Perf}
&\mathcal Perf_X^{\geq 0}:\mathfrak{Alg}_k\xrightarrow{\hspace*{5cm}}\mathfrak{Grpd}\qquad\qquad\qquad\qquad\qquad\qquad\qquad\qquad\qquad& \nonumber \\
&\qquad A\longmapsto\mathcal Perf_X^{\geq 0}\left(A\right):=\text{groupoid of perfect $\left(\mathscr O_X\otimes A\right)$-modules $\mathscr E$ in complexes}& \nonumber \\ 
&\,\,\qquad\qquad\text{such that $\mathrm{Ext}^i\left(\mathscr E,\mathscr E\right)=0$ for all $i<0$.}&
\end{eqnarray}
\begin{thm} \label{Lieb} \emph{(Lieblich)}
Functor \eqref{Perf} is an (underived) Artin stack over $k$ locally of finite presentation.
\end{thm}
\begin{proof}
See \cite{Lie} Theorem 4.2.1 and results therein.
\end{proof}
The assumptions on the base scheme $X$ in Theorem \ref{Lieb} -- whose proof relies on \emph{Artin Representability Theorem} (see \cite{Art}) -- can be relaxed, but the key condition of Lieblich's result remains the vanishing of all negative $\mathrm{Ext}$ groups\footnote{In \cite{Lie} a perfect complex $\mathcal E\in\mathrm D\left(X\right)$ such that $\mathrm{Ext}^i\left(\mathscr E,\mathscr E\right)=0$ for all $i<0$ is called \emph{universally gluable}; also in that paper the stack $\mathcal Perf_X^{\geq 0}$ is denoted by $\mathcal D^b_{\mathrm{pug}}\left(X/k\right)$.}; in particular observe that such a condition ensures that $\mathcal Perf_X^{\geq 0}$ is a well-defined groupoid-valued functor: as a matter of fact the group $\mathrm{Ext}^i\left(\mathcal E,\mathcal E\right)$, where $\mathcal E\in\mathrm D\left(X\right)$ and $i<0$, parametrises $i^{\mathrm{th}}$-order autoequivalences of $\mathcal E$, thus perfect complexes with trivial negative $\mathrm{Ext}$ groups do not carry any higher homotopy, but only usual automorphisms. \\
By means of Derived Algebraic Geometry it is possible to outstandingly generalise Lieblich's result: indeed consider the functor
\begin{eqnarray} \label{RPerf}
\mathbb R\mathcal Perf_X:&\mathfrak{dgAlg}_k^{\leq 0}&\xrightarrow{\hspace*{2cm}}\mathfrak{sSet} \nonumber \\
&A&\longmapsto Map\left(\mathfrak{Perf}\left(X\right)^{\mathrm{op}},\hat A_{\mathrm{pe}}\right)
\end{eqnarray}
where $\mathfrak{Perf}\left(X\right)$ stands for the dg-category of perfect complexes on $X$, $\hat A_{\mathrm{pe}}$ for the dg-category of \emph{perfect $A$-modules} (see \cite{TVa} for more details) and $Map$ for the mapping space of the model category of dg-categories (see \cite{Tab1}, \cite{Toe3} and \cite{Toe2} for more details).
\begin{thm} \label{Toe-Vaq} \emph{(To\" en-Vaqui\' e)}
Functor \eqref{RPerf} is a locally geometric\footnote{Recall that a derived stack $\mathcal F$ is said to be \emph{locally geometric} if it is the union of open truncated derived geometric substacks.} derived stack over $k$ locally of finite presentation.
\end{thm}
\begin{proof}
See \cite{TVa} Section 3; see also \cite{Toe1} Section 3.2.4 and Section 4.3.5 for a quicker explanation.
\end{proof}
It is easily seen that there is a derived geometric $1$-substack of $\mathbb R\mathcal Perf_X$ whose underived truncation is equivalent to $\mathcal Perf_X^{\geq 0}$, so Theorem \ref{Lieb} is recovered as a corollary of To\" en and Vaqui\' e's work. \\
Theorem \ref{Toe-Vaq} is a very powerful and elegant result, which has been highly inspiring in recent research: just to mention a few significant instances, it is one of the key ingredients in \cite{Sim2} where Simpson constructed a locally geometric stack of perfect complexes equipped with a $\lambda$-connection, \cite{STVa} where Sch\" urg, To\" en and Vaqui\' e constructed a derived determinant map from the derived stack of perfect complexes to the derived Picard stack and studied various applications to Deformation Theory and Enumerative Geometry, \cite{PTVV} where Pantev, To\" en, Vaqui\' e and Vezzosi set Derived Symplectic Geometry. However the proof provided in \cite{TVa} is quite abstract and complicated: as a matter of fact To\" en and Vaqui\' e actually constructed a derived stack parametrising \emph{pseudo-perfect} objects (see \cite{TVa} for a definition) in a fixed dg-category of finite type (again see \cite{TVa} for more details) and then proved by hand -- i.e. without applying any representability result, but rather using just the definitions from \cite{TVe} -- that this is locally geometric and locally of finite type. Theorem \ref{Toe-Vaq} is then obtained just as an interesting application.\\
In this section we will apply the representability and smoothness results discussed in Section 2.1 to obtain a simpler and more concrete proof of Theorem \ref{Toe-Vaq}; actually we will follow the path marked by Pridham in \cite{Pr1}, where he develops general methods to study derived moduli of schemes and sheaves. In a way the approach we propose is the derived counterpart of Lieblich's one, as the latter is based on Artin Representability Theorem rather than the definition of (underived) Artin stack. Moreover we will give a rather explicit description of the derived geometric stacks determining the local geometricity of $\mathbb R\mathcal Perf_X$: again, such a picture is certainly present in To\" en and Vaqui\' e's work, but unravelling the language in order to clearly write down the relevant substacks might be non-trivial. Halpern-Leistner and Preygel have recently studied the stack $\mathbb R\mathcal Perf_X$ via representability as well, though their approach does not make use of Pridham's theory: for more details see \cite{Ha-LePr} Section 2.5. Other related work has been carried by Pandit, who showed in \cite{Pa} that the derived moduli stack of compact objects in a perfect symmetric monoidal infinity-category is locally geometric and locally of finite type, and Lowrey, who studied in \cite{Low} the derived moduli stack of pseudo-coherent complexes on a proper scheme.\\
Let $\mathcal X$ be a (possibly) derived scheme over $R$ and recall that the \emph{C\v ech nerve} of $\mathcal X$ associated to a fixed affine open cover $U:=\underset{\alpha}{\coprod}U_{\alpha}$ is defined to be the simplicial affine scheme
\begin{equation*}
\check{\mathcal X}:\qquad \xymatrix{\check{\mathcal X}_0\ar[r] & \check{\mathcal X}_1\ar@<2.5pt>[l]\ar@<-2.5pt>[l]\ar@<2.5pt>[r]\ar@<-2.5pt>[r] & \check{\mathcal X}_2 \ar[l]\ar@<5pt>[l]\ar@<-5pt>[l]\ar[r]\ar@<5pt>[r]\ar@<-5pt>[r] & \cdots\ar@<-7.5pt>[l]\ar@<-2.5pt>[l]\ar@<2.5pt>[l]\ar@<7.5pt>[l]}
\end{equation*}
where
\begin{equation*}
\check{\mathcal X}_m:=\underbrace{U\times^h_XU\times^h_X\cdots\times^h_XU}_{m+1\text{ times}}
\end{equation*}
while faces and degenerations are induced naturally by canonical projections and diagonal embeddings respectively.
Consider also the cosimplicial differential graded commutative $R$-algebra $O\left(\mathcal X\right)$ defined in level $m$ by
\begin{equation} \label{OUps}
O\left(\mathcal X\right)^m:=\Gamma\left(\check{\mathcal X}_m,\mathscr O_{\check{\mathcal X}_m}\right)
\end{equation}
whose cosimplicial structure maps are induced through the global section functor by the ones determining the simplicial structure of $\check X$. 
\begin{defn} \label{der mod}
Define a \emph{derived module} over $\mathcal X$ to be a cosimplicial $O\left(\mathcal X\right)$-module in complexes.
\end{defn}
We will denote by $\mathfrak{dgMod}\left(\mathcal X\right)$ the category of derived modules over $\mathcal X$; just unravelling Definition \ref{der mod} we see that an object $\mathcal M\in\mathfrak{dgMod}\left(\mathcal X\right)$ is made of cochain complexes $\mathcal M^m$ of $O\left(\mathcal X\right)^m$-modules related by maps
\begin{eqnarray*}
&\partial^i:\mathcal M^m\otimes^{\mathbb L}_{O\left(\mathcal X\right)^m}O\left(\mathcal X\right)^{m+1}\longrightarrow\mathcal M^{m+1}& \\
&\sigma^i:\mathcal M^m\otimes^{\mathbb L}_{O\left(\mathcal X\right)^m}O\left(\mathcal X\right)^{m-1}\longrightarrow\mathcal M^{m-1}&
\end{eqnarray*}
satisfying the usual cosimplicial identities. Observe that the projective model structures on cochain complexes we discussed in Section 1.2 induces a model structure on $\mathfrak{dgMod}\left(X\right)$, which we will still refer to as a projective model structure: in particular a morphism $f:\mathcal M\rightarrow\mathcal N$ in $\mathfrak{dgMod}\left(\mathcal X\right)$ is 
\begin{itemize}
\item a weak equivalence if $f^m:\mathcal M^m\rightarrow\mathcal N^m$ is a quasi-isomorphism; 
\item a fibration if $f^m:\mathcal M^m\rightarrow\mathcal N^m$ is degreewise surjective;
\item a cofibration if it has the left lifting property with respect to all fibrations (see \cite{Pr1} Section 4.1 for a rather explicit characterisation of them).
\end{itemize}
In the same way, the category $\mathfrak{dgMod}\left(\mathcal X\right)$ inherits a simplicial structure from the category of $R$-modules in complexes: more clearly for any $\mathcal M,\mathcal N\in\mathfrak{dgMod}\left(\mathcal X\right)$ consider the chain complex $\left(\mathrm{HOM}_{\mathcal X}\left(\mathcal M,\mathcal N\right),\delta\right)$ defined by the relations
\begin{eqnarray*} 
\mathrm{HOM}_{\mathcal X}\left(\mathcal M,\mathcal N\right)_n:=\mathrm{Hom}_{O\left(\mathcal X\right)}\left(\mathcal M,\mathcal N[-n]\right)\qquad\qquad\qquad\qquad \nonumber \\
\forall f\in {\mathrm{HOM}_{\mathcal X}\left(\mathcal M,\mathcal N\right)_n}\qquad\delta_n\left(f\right):=\bar d^n\circ f-\left(-1\right)^nf\circ d^n\in {\mathrm{HOM}_{\mathcal X}\left(\mathcal M,\mathcal N\right)_{n-1}}
\end{eqnarray*}
and define the Hom spaces just by taking good truncation and denormalisation, i.e. set
\begin{equation*}
\underline{\mathrm{Hom}}_{\mathfrak{dgMod}\left(\mathcal X\right)}\left(\mathcal M,\mathcal N\right):=\mathbf K\left(\tau_{\geq 0}\mathrm{HOM}_{\mathcal X}\left(\mathcal M,\mathcal N\right)\right).
\end{equation*}
\begin{defn}
A \emph{derived quasi-coherent sheaf} over $\mathcal X$ is a derived module $\mathcal M$ for which all face maps $\partial^i$ are weak equivalences.
\end{defn}
Let $\mathfrak{dgMod}\left(\mathcal X\right)_{\mathrm{cart}}$ to be the full subcategory of $\mathfrak{dgMod}\left(\mathcal X\right)$ consisting of derived quasi-coherent sheaves: this inherits a simplicial structure from the larger category and -- even if it has not enough limits and thus cannot be a model category -- it also inherits a reasonably well behaved subcategory of weak equivalences, so there is a homotopy category $\mathrm{Ho}\left(\mathfrak{dgMod}\left(\mathcal X\right)_{\mathrm{cart}}\right)$ of quasi-coherent modules over $\mathcal X$ simply obtained by localising $\mathfrak{dgMod}\left(\mathcal X\right)_{\mathrm{cart}}$ at weak equivalences.
\begin{rem}
The constructions above make sense in a much wider generality: as a matter of fact in \cite{Pr1} Pridham defined derived quasi-coherent modules over any \emph{homotopy derived Artin hypergroupoid} (see \cite{Pr4}) and through these objects he recovered the notion of homotopy-Cartesian module over a derived geometric stack which had previously been investigated by To\" en and Vezzosi in \cite{TVe}; also Corollary \ref{DMQCS} -- which is the main tool to deal with derived moduli of sheaves -- holds in this much vaster generality. We have chosen to discuss derived quasi-coherent modules only for derived schemes since our goal is to study perfect complexes on a proper scheme, for which the full power of Pridham's theory of Artin hypergroupoids is not really needed. In particular bear in mind that the C\v ech nerve of a derived scheme associated to an affine open cover is an example of homotopy Zariski $1$-hypergroupoid.
\end{rem}
From now on fix $R$ to be an ordinary (underived) $k$-algebra and $X$ to be a quasi-compact semi-separated scheme over $R$; note that in in \cite{Hut} H\" utterman showed that
\begin{equation*} 
\mathrm{Ho}\left(\mathfrak{dgMod}_{\mathrm{cart}}\left(X\right)\right)\simeq\mathrm D\left(\mathfrak{QCoh}\left(X\right)\right)
\end{equation*}
so in this case derived quasi-coherent modules are precisely what one would like them to be. \\
Now define the functor
\begin{eqnarray} \label{dCART}
d\mathrm{CART}_{X}:&\mathfrak{dg}_b\mathfrak{Nil}^{\leq 0}_R&\xrightarrow{\hspace*{1.75cm}}\mathfrak{sCat} \nonumber \\
&A&\longmapsto\left(\mathfrak{dgMod}_{\mathrm{cart}}\left(X\otimes^{\mathbb L}_RA\right)\right)^c
\end{eqnarray}
where $\left(\mathfrak{dgMod}_{\mathrm{cart}}\left(X\otimes^{\mathbb L}_RA\right)\right)^c$ is the full simplicial subcategory of $\mathfrak{dgMod}_{\mathrm{cart}}\left(X\otimes^{\mathbb L}_RA\right)$ on cofibrant objects, i.e. it is the (simplicial) category of cofibrant derived quasi-coherent modules on the derived scheme $X\otimes^{\mathbb L}_RA$. 
\begin{prop} \label{dCART homog smooth} \emph{(Pridham)}
Functor \eqref{dCART} is $2$-homogeneous and formally $2$-quasi-smooth.
\end{prop}
\begin{proof}
This is \cite{Pr1} Proposition 4.11, which relies on the arguments of \cite{Pr1} Proposition 3.7; we will discuss Pridham's proof here for the reader's convenience. \\
We first prove that $d\mathrm{CART}_{X}$ is $2$-homogeneous; let $A\rightarrow B$ be a square-zero extension and $C\rightarrow B$ a morphism in $\mathfrak{dg}_b\mathfrak{Nil}^{\leq 0}_R$ and fix $\mathscr F,\mathscr F'\in d\mathrm{CART}_{X}\left(A\times_BC\right)$. Since by definition $\mathscr F$ and $\mathscr F'$ are cofibrant (i.e. degreewise projective by Proposition \ref{DG cofibrant}) we immediately have that the commutative square of simplicial sets
\begin{equation*}
\tiny{\xymatrix{ \underline{\mathrm{Hom}}_{d\mathrm{CART}\left(A\times_BC\right)}\left(\mathscr F,\mathscr F'\right)\ar[r]\ar[d] & \underline{\mathrm{Hom}}_{d\mathrm{CART}\left(A\right)}\left(\mathscr F\otimes_{A\times_BC}A,\mathscr F'\otimes_{A\times_BC}A\right)\ar[d] \\
\underline{\mathrm{Hom}}_{d\mathrm{CART}\left(C\right)}\left(\mathscr F\otimes_{A\times_BC}C,\mathscr F'\otimes_{A\times_BC}C\right)\ar[r] & \underline{\mathrm{Hom}}_{d\mathrm{CART}\left(B\right)}\left(\mathscr F\otimes_{A\times_BC}B,\mathscr F'\otimes_{A\times_BC}B\right) }}
\end{equation*}
is actually a Cartesian diagram. Moreover fix $\mathscr F_A\in d\mathrm{CART}_{X}\left(A\right)$ and $\mathscr F_C\in d\mathrm{CART}_{X}\left(C\right)$ and let $\alpha:\mathscr F_A\otimes_AB\rightarrow\mathscr F_C\otimes_C B$ be an isomorphism; now define 
\begin{equation*}
\mathscr F:=\mathscr F_A\otimes_{\alpha,\mathscr F_C\otimes_CB}\mathscr F_C\simeq\mathscr F_C\otimes_{\alpha,\mathscr F_A\otimes B}\mathscr F_A
\end{equation*}
which is a derived quasi-coherent module over $X\otimes_R\left(A\otimes_BC\right)$. Clearly we have that 
\begin{equation*}
\mathscr F\otimes_{A\times_BC}A\simeq\mathscr F_A\qquad\qquad\mathscr F\otimes_{A\times_BC}C\simeq\mathscr F_C
\end{equation*}
and also observe that $\mathscr F$ is cofibrant, i.e. $\mathscr F\in d\mathrm{CART}_{X}\left(A\times_BC\right)$. This shows that $\underline{\mathrm{Hom}}_{d\mathrm{CART}_{X}}$ is homogeneous, which means that $d\mathrm{CART}_{X}$ is a $2$-homogeneous functor. \\
Now we prove that $d\mathrm{CART}_{X}$ is formally $2$-quasi-smooth; again let $I\hookrightarrow A\twoheadrightarrow B$ be a square-zero extension and pick $\mathscr F,\mathscr F'\in d\mathrm{CART}_{X}\left(A\right)$. Observe that, since $\mathscr F'$ is cofibrant as a quasi-coherent module over $X\otimes^{\mathbb L}_R A$, we have that the induced map $\mathscr F'\rightarrow \mathscr F'\otimes_AB$ is still a square-zero extension; furthermore if $A\rightarrow B$ is also a quasi-isomorphism, then so is $\mathscr F'\rightarrow \mathscr F'\otimes_AB$: as a matter of fact notice, as a consequence of Proposition \ref{DG cofibrant}, that 
\begin{equation*}
\mathrm{ker}\left(\mathscr F'\rightarrow \mathscr F'\otimes_AB\right)=\mathscr F'\otimes_AI. 
\end{equation*}
Now it follows that the natural chain map
\begin{equation*}
\mathrm{HOM}_{d\mathrm{CART}_{X}\left(A\right)}\left(\mathscr F,\mathscr F'\right)\longrightarrow\mathrm{HOM}_{d\mathrm{CART}_{X}\left(B\right)}\left(\mathscr F\otimes_A B,\mathscr F'\otimes_A B\right)
\end{equation*}
is degreewise surjective and a quasi-isomorphism whenever so is $A\rightarrow B$. Now, by just applying truncation and Dold-Kan denormalisation, we get that the morphism of simplicial sets
\begin{equation*}
\underline{\mathrm{Hom}}_{d\mathrm{CART}_{X}\left(A\right)}\left(\mathscr F,\mathscr F'\right)\longrightarrow\underline{\mathrm{Hom}}_{d\mathrm{CART}_{X}\left(B\right)}\left(\mathscr F\otimes_A B,\mathscr F'\otimes_A B\right)
\end{equation*}
is a fibration, which is trivial in case the square-zero extension $A\rightarrow B$ is a quasi-isomorphism. This shows that $\underline{\mathrm{Hom}}_{d\mathrm{CART}_{X}}$ is formally quasi-smooth, so in order to finish the proof we only need to prove that the base-change morphism
\begin{equation} \label{map 2-fib}
d\mathrm{CART}_{X}\left(A\right)\longrightarrow d\mathrm{CART}_{X}\left(B\right)
\end{equation}
is a $2$-fibration, which is trivial whenever the extension $A\rightarrow B$ is acyclic. The computations in \cite{Pr4} Section $7$ imply that obstructions to lifting a quasi-coherent module $\mathscr F\in d\mathrm{CART}_{X}\left(B\right)$ to $d\mathrm{CART}_{X}\left(A\right)$ lie in the group
\begin{equation*}
\mathrm{Ext}^2_{X\otimes_R^{\mathbb L}B}\left(\mathscr F,\mathscr F\otimes_B I\right).
\end{equation*}
so in particular if $H^*\left(I\right)=0$ then map \eqref{map 2-fib} is a trivial $2$-fibration. Now fix $\mathscr F\in d\mathrm{CART}_{X}\left(A\right)$, denote $\tilde{\mathscr F}:=\mathscr F\otimes_A B$ and let $\theta:\tilde{\mathscr F}\rightarrow\mathscr G$ be a homotopy equivalence in $d\mathrm{CART}_{X}\left(B\right)$. By cofibrancy, there exist a unique lift $\mathring{\mathscr G}$ of $\mathscr G$ to $A$ as a cosimplicial graded module and, in the same fashion, we can lift $\theta$ to a graded morphism $\mathring{\theta}:\mathscr F\rightarrow\mathring{\mathscr G}$: we want to prove that there also exist compatible lifts of the differential. The obstruction to lift the differential $d$ of $\mathscr G$ to a differential $\delta$ on $\mathring{\mathscr G}$ is given by a pair
\begin{equation*}
\left(u,v\right)\in\mathrm{HOM}^2_{X\otimes^{\mathbb L}_RB}\left(\mathscr G,\mathscr G\otimes_BI\right)\times\mathrm{HOM}^1_{X\otimes^{\mathbb L}_RB}\left(\tilde{\mathscr F},\mathscr G\otimes_BI\right)
\end{equation*}
satisfying $d\left(u\right)=0$ and $d\left(v\right)=u\circ\theta$. A different choice for $\left(\delta,\tilde{\theta}\right)$ would be of the form $\left(\delta+a,\tilde{\theta}+b\right)$, with
\begin{equation*}
\left(a,b\right)\in\mathrm{HOM}^1_{X\otimes^{\mathbb L}_RB}\left(\mathscr G,\mathscr G\otimes_BI\right)\times\mathrm{HOM}^0_{X\otimes^{\mathbb L}_RB}\left(\tilde{\mathscr F},\mathscr G\otimes_BI\right)
\end{equation*}
so that the pair $\left(u,v\right)$ is sent to $\left(u+d\left(a\right),v+d\left(b\right)+a\circ\theta\right)$. It follows that the obstruction to lifting $\theta$ and $\mathscr G$ lies in 
\begin{equation} \label{big obstr}
H^2\left(\mathrm{cone}\left(
\mathrm{HOM}_{X\otimes^{\mathbb L}_RB}\left(\mathscr G,\mathscr G\otimes_BI\right)\overset{\theta^*}{\xrightarrow{\hspace*{1cm}}}\mathrm{HOM}_{X\otimes^{\mathbb L}_RB}\left(\tilde{\mathscr F},\mathscr G\otimes_BI\right)
\right)\right).
\end{equation}
Since $\theta$ is a homotopy equivalence we have that $\theta^*$ is a quasi-isomorphism: in particular the cohomology group \eqref{big obstr} is $0$, which means that suitable lifts exist. This completes the proof.
\end{proof}
Proposition \ref{dCART homog smooth} is the key ingredient to build upon Pridham Nilpotent Representability Criterion an ad-hoc result to deal with moduli of sheaves.
\begin{cor} \label{DMQCS} (Pridham) 
Let 
\begin{equation*}
\mathbf M:\mathfrak{Alg}_{H^0\left(R\right)}\rightarrow\mathfrak{sCat}
\end{equation*}
be a presheaf satisfying the following conditions:
\begin{enumerate}
\item $\mathbf M$ is $n$-truncated;
\item $\mathbf M$ is open in the functor 
\begin{equation*}
A\longmapsto\pi^0\mathcal W\left(\mathfrak{dgMod}\left(X\otimes^{\mathbb L}_R A\right)_{\mathrm{cart}}\right)\qquad A\in\mathfrak{Alg}_{H^0\left(R\right)};
\end{equation*}
\item If $\left\{f_{\alpha}:A\rightarrow B_{\alpha}\right\}_{\alpha}$ is an \' etale cover in $\mathfrak{Alg}_{H^0\left(R\right)}$, then $\mathscr E\in\pi_0\pi^0\mathcal W\left(\mathfrak{dgMod}\left(X\otimes^{\mathbb L}_R A\right)_{\mathrm{cart}}\right)$ lies in the essential image of $\pi_0\mathbf M\left(A\right)$ whenever $\left(f_{\alpha}\right)^*\mathscr E$ is in the essential image of $\pi_0\mathbf M\left(B_{\alpha}\right)$ for all $\alpha$;
\item For all finitely generated $A\in\mathfrak{Alg}_{H^0\left(R\right)}$ and all $\mathscr E\in\mathbf M\left(A\right)$, the functors 
\begin{equation*}
\mathrm{Ext}^i_{X\otimes^{\mathbb L}_RA}\left(\mathscr E,\mathscr E\otimes^{\mathbb L}_A-\right):\mathfrak{Mod}_A\longrightarrow\mathfrak{Ab}
\end{equation*}
preserve filtered colimits $\forall i\neq 1$;
\item For all finitely generated integral domains $A\in\mathfrak{Alg}_{H^0\left(R\right)}$ and all $\mathscr E\in\mathbf M\left(A\right)$, the groups $\mathrm{Ext}^i_{X\otimes_R^{\mathbb L}A}\left(\mathscr E,\mathscr E\right)$ are finitely generated $A$-modules;
\item The functor 
\begin{equation*}
c(\pi_0\mathbf M):\mathfrak{Alg}_{H^0\left(R\right)}\longrightarrow\mathfrak{Set}
\end{equation*}
of components of the groupoid $\pi_0\mathbf M$ preserves filtered colimits;
\item For all complete discrete local Noetherian normal $H^0\left(R\right)$-algebras $A$, for all $\mathscr E\in\mathbf M\left(A\right)$ and for all $i>0$ the canonical maps
\begin{eqnarray*}
&c(\pi_0\mathbf M\left(A\right))\longrightarrow\underset{\underset{r}{\longleftarrow}}{\mathrm{lim}}\;c(\pi_0\mathbf M\left(\nicefrac{A}{\mathfrak m_A^r}\right))& \\
&\mathrm{Ext}^i_{X\otimes_R^{\mathbb L}A}\left(\mathscr E,\mathscr E\right)\longrightarrow\underset{\underset{r}{\longleftarrow}}{\mathrm{lim}}\;\mathrm{Ext}^i_{X\otimes_R^{\mathbb L}A}\left(\mathscr E,\nicefrac{\mathscr E}{\mathfrak m_A^r}\right)& \forall i<0
\end{eqnarray*}
are isomorphisms.
\end{enumerate}
Let
\begin{equation*}
\breve{\mathbf M}:\mathfrak{dg}_b\mathfrak{Nil}^{\leq 0}_R\longrightarrow\mathfrak{sCat}
\end{equation*}
be the full simplicial subcategory of $\mathcal W\left(d\mathrm{CART}_{X}\left(A\right)\right)$ consisting of objects $\mathscr F$ such that the complex $\mathscr F\otimes_A H^0\left(A\right)$ is weakly equivalent in $\mathfrak{dgMod}_{\mathrm{cart}}\left(X\otimes_R^{\mathbb L} H^0\left(A\right)\right)$ to an object of $\mathbf M\left(H^0\left(A\right)\right)$. Then the functor $\bar W\breve{\mathbf M}$ is (the restriction to $\mathfrak{dg}_b\mathfrak{Nil}^{\leq 0}_R$ of) a derived geometric $n$-stack.
\end{cor}
\begin{proof}
This is \cite{Pr1} Theorem 4.12; we just sketch the main ideas of the proof for the reader's convenience. We basically need to verify that the various conditions in the statement imply that the simplicial presheaf $\bar W\breve{\mathbf M}$ satisfies Pridham Nilpotent Representability Criterion (Theorem \ref{Nilp Rep}). \\
First observe that by Condition $(2)$ we have that
\begin{equation} \label{openness reps}
\breve{\mathbf M}\left(A\right)\approx\mathbf M\left(H^0\left(A\right)\right)\times^h_{\mathcal W\left(\mathfrak{dgMod}\left(X\otimes^{\mathbb L}_R H^0\left(A\right)\right)_{\mathrm{cart}}\right)} \mathcal W\left(\mathfrak{dgMod}\left(X\otimes^{\mathbb L}_R A\right)_{\mathrm{cart}}\right).\footnote{The symbol $\approx$ stands for {} ``weakly equivalent''.}
\end{equation}
As a matter of fact, the openness of $\mathbf M$ inside $\pi^0\mathcal W\left(\mathfrak{dgMod}\left(X\otimes^{\mathbb L}_R -\right)_{\mathrm{cart}}\right)$ says that the inclusion
\begin{equation*}
\mathbf M\hookrightarrow \pi^0\mathcal W\left(\mathfrak{dgMod}\left(X\otimes^{\mathbb L}_R -\right)_{\mathrm{cart}}\right)
\end{equation*}
is homotopy formally \' etale, thus we have the representation given by formula \eqref{openness reps}. \\
Then note that the proof of \cite{Pr4} Lemma 5.23 adapts to $\mathscr O_X$-modules, i.e. the assignment 
\begin{equation*}
A\longmapsto\mathfrak{dgMod}_{X\otimes^{\mathbb L}_RA}
\end{equation*}
provides us with a left Quillen hypersheaf, thus \cite{Pr4} Proposition 5.9 implies that $\bar W\mathcal W\left(d\mathrm{CART}_{X}\right)$ is an \' etale hypersheaf; now Condition $\left(3\right)$ and Proposition \ref{hypersheaf} ensure that $\bar W\mathcal W\left(\mathbf M\right)$ is a hypersheaf for the \' etale topology. \\ 
Also recall that the computations in \cite{Pr4} Section $7$ imply that
\begin{equation} \label{coho descr}
\mathrm D^i_{\mathscr E}\left(\bar W\breve{\mathbf M},M\right)\simeq\mathrm{Ext}^{i+1}_{X\otimes_A^{\mathbb L} M}\left(\mathscr E,\mathscr E\otimes^{\mathbb L}_AM\right)
\end{equation}
for all nilpotent dgca's $A\in\mathfrak{dg}_b\mathfrak{Nil}^{\leq 0}_R$, all complexes $\mathscr E\in\bar W\breve{\mathbf M}\left(A\right)$ and dg $A$-modules $M$. \\
Now Proposition \ref{dCART homog smooth} and Proposition \ref{qsmooth homog} tell us that Condition $\left(4\right)$ and Condition $\left(5\right)$ imply the homotopy-theoretic properties required by Pridham Nilpotent Representability Criterion, while the description of cohomology theories given by \eqref{coho descr} ensures the compatibility of such modules with filtered colimits and base-change. In the end the weak completeness condition given by Condition $\left(9\right)$ of Theorem \ref{Nilp Rep} follows from Condition $\left(7\right)$ through a few standard Mittag-Leffler computations: for more details see \cite{Pr1} Theorem 4.12 or the proof of Theorem \ref{Perf repr}, where similar calculations will be explicitly developed. 
\end{proof}
Now we are ready to study derived moduli of perfect complexes by means of Lurie-Pridham representability; consider the functor
\begin{eqnarray} \label{und prf cmplx}
&\mathbf M^n:\mathfrak{Alg}_R\longrightarrow\mathfrak{sCat}\qquad\qquad\qquad\qquad\qquad\qquad\qquad\qquad\qquad\qquad\qquad& \nonumber \\
&A\longmapsto\mathbf M^n\left(A\right):=\text{full simplicial subcategory}\qquad\qquad\qquad& \nonumber \\ 
&\qquad\qquad\qquad\qquad\qquad\;\;\;\text{of perfect complexes $\mathcal E$ of $\left(\mathscr O_X\otimes^{\mathbb L}_R A\right)$-modules}& \nonumber \\ 
&\qquad\qquad\qquad\qquad\;\text{such that $\mathrm{Ext}_{X\otimes^{\mathbb L}_R A}^i\left(\mathcal E,\mathcal E\right)=0$ for $i<-n$}&
\end{eqnarray}
which classifies perfect $\mathscr O_X$-modules in complexes with trivial $\mathrm{Ext}$ groups in higher negative degrees.
\begin{thm} \label{Perf repr}
In the above notations, assume that the scheme $X$ is also proper; then functor \eqref{und prf cmplx} induces a derived geometric $n$-stack $\mathbb R\mathcal Perf_X^n$.
\end{thm}
\begin{proof}
We have to prove that functor \eqref{und prf cmplx} satisfies the conditions of Corollary \ref{DMQCS}. \\
First of all, notice that the vanishing condition on higher negative $\mathrm{Ext}$ groups guarantees that the simplicial presheaf $\mathbf M^n$ is $n$-truncated, which is exactly Condition $\left(1\right)$. \\
Now we look at Condition $\left(2\right)$, hence we need to prove the openness of $\mathbf M^n$ as a subfunctor of $\pi^0\mathcal W\left(\mathfrak{dgMod}\left(X\otimes^{\mathbb L}_R-\right)_{\mathrm{cart}}\right)$; it is immediate to see that $\mathbf M^n\left(A\right)$ is a full simplicial subcategory of $\pi^0\mathcal W\left(\mathfrak{dgMod}\left(X\otimes^{\mathbb L}_RA\right)_{\mathrm{cart}}\right)$, so we only need to check that the map
\begin{equation} \label{formal} 
\mathbf M^n\hookrightarrow\pi^0\mathcal W\left(\mathfrak{dgMod}\left(X\otimes^{\mathbb L}_R-\right)_{\mathrm{cart}}\right) 
\end{equation}
is homotopy formally \' etale, i.e. that the morphism of formal groupoids\footnote{Notice that (homotopy) formal \' etaleness is a local property, so we can restrict map \eqref{formal} to formal objects.}
\begin{equation*} 
\pi_0\mathbf M^n\hookrightarrow\pi_0\pi^0\mathcal W\left(\mathfrak{dgMod}\left(X\otimes^{\mathbb L}_R-\right)_{\mathrm{cart}}\right) 
\end{equation*}
is formally \' etale. By classical Formal Deformation Theory this amounts to check that the map induced by morphism \eqref{formal} is an isomorphism on tangent spaces and an injection on obstruction spaces (see for example \cite{Ser} Section 2.1 and \cite{Man2} Section V.8), so fix a square-zero extension $I\hookrightarrow A\twoheadrightarrow B$ and a perfect complex $\mathscr E\in\mathbf M^n\left(B\right)$. By Lieblich's work (see \cite{Lie} Section 3) we have that
\begin{itemize}
\item the tangent space to the functor $\pi_0\mathbf M^n$ at $\mathscr E$ is given by the group $\mathrm{Ext}_{X\otimes^{\mathbb L}_RA}^1\left(\mathscr E,\mathscr E\otimes^{\mathbb L}_B I\right)$;
\item a functorial obstruction space for $\pi_0\mathbf M^n$ at $\mathscr E$ is given by the group $\mathrm{Ext}_{X\otimes^{\mathbb L}_RA}^2\left(\mathscr E,\mathscr E\otimes^{\mathbb L}_B I\right)$.
\end{itemize}
On the other hand, it is well known (for instance see the proof of \cite{Pr1} Theorem 4.12) that
\begin{itemize}
\item the tangent space to the functor $\pi_0\pi^0\mathcal W\left(\mathfrak{dgMod}\left(X\otimes^{\mathbb L}_R-\right)_{\mathrm{cart}}\right)$ at $\mathscr E$ is given by the group $\mathrm{Ext}_{X\otimes^{\mathbb L}_RA}^1\left(\mathscr E,\mathscr E\otimes^{\mathbb L}_B I\right)$;
\item a functorial obstruction space for $\pi_0\pi^0\mathcal W\left(\mathfrak{dgMod}\left(X\otimes^{\mathbb L}_R-\right)_{\mathrm{cart}}\right)$ at $\mathscr E$ is given by the group $\mathrm{Ext}_{X\otimes^{\mathbb L}_RA}^2\left(\mathscr E,\mathscr E\otimes^{\mathbb L}_B I\right)$.
\end{itemize}
It follows that the group homomorphism induced by map \eqref{formal} on first-order deformations and obstruction theories is just the identity, so Condition $\left(2\right)$ holds. \\
Now let us look at Condition $\left(3\right)$: take an \' etale cover $\left\{f_{\alpha}:A\rightarrow B_{\alpha}\right\}_{\alpha}$ in $\mathfrak{Alg}_{R}$ and let $\mathscr E$ be an object in $\pi_0\pi^0\mathcal W\left(\mathfrak{dgMod}\left(X\otimes^{\mathbb L}_RA\right)_{\mathrm{cart}}\right)$ such that the derived modules $\left(f_{\alpha}\right)^*\mathscr E$ over $X\otimes^{\mathbb L}_RB_{\alpha}$ are perfect; then the derived quasi-coherent module $\mathscr E$ has to be perfect as well, because perfectness is a local property which is preserved under pull-back. It follows that Condition $\left(3\right)$ holds. \\
In order to check Condition $\left(4\right)$, fix a finitely generated $R$-algebra $A$ and a perfect complex $\mathscr E$ of $\left(\mathscr O_X\otimes_R^{\mathbb L}A\right)$-modules and consider an inductive system $\left\{B_{\alpha}\right\}_{\alpha}$ of $A$-algebras. The perfectness assumption on $\mathscr E$ allows us to substitute this with a bounded complex $\mathscr F$ of flat $\left(\mathscr O_X\otimes^{\mathbb L}_RA\right)$-modules, so we get that $\mathrm{Ext}^i_{X\otimes^{\mathbb L}_RA}\left(\mathscr E,\mathscr E\otimes_A^{\mathbb L}-\right)$ preserves filtered colimits if and only if so does $\mathrm{Ext}^i_{X\otimes^{\mathbb L}_RA}\left(\mathscr E,\mathscr F\otimes_A-\right)$, which is just the classical $\mathrm{Ext}$ functor. Now a few standard results in Homological Algebra imply the following canonical isomorphisms $\forall i\geq 0$
\begin{equation*}
\mathrm{Ext}^i_{X\otimes^{\mathbb L}_RA}\left(\mathscr E,\mathscr F\otimes_A\underset{\underset{\alpha}{\longrightarrow}}{\lim}\;B_{\alpha}\right)\simeq\mathrm{Ext}^i_{X\otimes^{\mathbb L}_RA}\left(\mathscr E,\underset{\underset{\alpha}{\longrightarrow}}{\lim}\;\mathscr F\otimes_AB_{\alpha}\right)\simeq\underset{\underset{\alpha}{\longrightarrow}}{\lim}\;\mathrm{Ext}^i_{X\otimes^{\mathbb L}_RA}\left(\mathscr E,\mathscr F\otimes_AB_{\alpha}\right).
\end{equation*}
In particular in the first isomorphism we are using the fact that filtered colimits commute with exact functors (and so is the tensor product as $\mathscr F$ is flat in each degree), while in the second one we are using the fact that filtered colimits commute with all $\mathrm{Ext}$ functors, since $\mathscr E$ is a finitely presented object as by perfectness this is locally quasi-isomorphic to a bounded complex of vector bundles. Ultimately the key idea in this argument is that the assumptions on the complexes we are classifying allow us to compute the $\mathrm{Ext}$ groups by choosing a {} ``projective resolution'' in the first variable and a {} ``flat resolution'' in the second one, so that all necessary finiteness conditions to make $\mathrm{Ext}^i_{X\otimes^{\mathbb L}_RA}$ and $\underset{\underset{\alpha}{\longrightarrow}}{\mathrm{lim}}$ commute are verified (see \cite{We} Section 2.6). It follows that Condition $\left(4\right)$ holds.\\
In order to check Condition $\left(5\right)$, fix a finitely generated $R$-algebra $A$ and a perfect complex $\mathscr E$ of $\left(\mathscr O_X\otimes_R^{\mathbb L}A\right)$-modules and again choose $\mathscr F$ to be a bounded complex of flat $\left(\mathscr O_X\otimes^{\mathbb L}_RA\right)$-modules being quasi-isomorphic to $\mathscr E$. Consider the derived endomorphism complex of $\mathscr E$ over $X\otimes^{\mathbb L}_RA$: we have that
\begin{equation*}
\mathbb R\mathcal Hom_{\mathscr O_X\otimes_R^{\mathbb L}A}\left(\mathscr E,\mathscr E\right)\approx\left(\mathscr E\right)^{\vee}\otimes_{\mathscr O_X\otimes^{\mathbb L}_R A}\mathscr F
\end{equation*}
where 
\begin{equation*}
\mathscr E^{\vee}:=\mathbb R\mathcal Hom_{\mathscr O_X\otimes_R^{\mathbb L}A}\left(\mathscr E,\mathscr O_X\right).
\end{equation*}
Notice that, again, we have computed the complex $\mathbb R\mathcal Hom_{\mathscr O_X\otimes_R^{\mathbb L}A}\left(\mathscr E,\mathscr E\right)$ by choosing a {} ``flat resolution'' in the second entry and a {} ``projective resolution'' in the first one; now consider the cohomology sheaves
\begin{equation*}
\mathcal Ext^i_{\mathscr O_X\otimes_R^{\mathbb L}A}\left(\mathscr E,\mathscr E\right):=\mathscr H^i\left(\mathbb R\mathcal Hom_{\mathscr O_X\otimes_R^{\mathbb L}A}\left(\mathscr E,\mathscr E\right)\right)
\end{equation*}
and note that these are coherent $\mathscr O_X\otimes_R^{\mathbb L}A$-modules. The local-to-global spectral sequence
\begin{equation} \label{spectral}
H^p\left(X,\mathcal Ext^q_{\mathscr O_X\otimes_R^{\mathbb L}A}\left(\mathscr E,\mathscr E\right)\right)\Longrightarrow\mathrm{Ext}_{X\otimes_R^{\mathbb L}A}^{p+q}\left(\mathscr E,\mathscr E\right)
\end{equation}
relates the cohomology of the $\mathcal Ext$ sheaves to the $\mathrm{Ext}$ groups and is well-known to converge: since the sheaves $\mathcal Ext^i_{\mathscr O_X\otimes_R^{\mathbb L}A}\left(\mathscr E,\mathscr E\right)$ are coherent and finitely many, formula \eqref{spectral} implies that the groups $\mathrm{Ext}_{X\otimes_R^{\mathbb L}A}^{p+q}\left(\mathscr E,\mathscr E\right)$ are finitely generated as $A$-modules, thus Condition $\left(5\right)$ holds. \\
Now we look at Condition $\left(6\right)$; fix an inductive system $\left\{A_{\alpha}\right\}_{\alpha}$ of $R$-algebras and let $A:=\underset{\underset{\alpha}{\longrightarrow}}{\mathrm{lim}}\;A_{\alpha}$: we need to show that 
\begin{equation} \label{components}
c\left(\pi_0\mathbf M^n\left(A\right)\right)=\underset{\underset{\alpha}{\longrightarrow}}{\mathrm{lim}}\;c\left(\pi_0\mathbf M^n\left(A_{\alpha}\right)\right)
\end{equation}
where for any $R$-algebra $B$
\begin{equation*}
c\left(\pi_0\mathbf M^n\left(B\right)\right):=\left\{\text{isomorphism classes of perfect complexes of $\left(\mathscr O_X\otimes_R^{\mathbb L}B\right)$-modules}\right\}.
\end{equation*}
Because being a perfect complex is local property, it suffices to show that formula \eqref{components} holds locally, i.e replacing $X$ with an open affine subscheme $U$; in particular, as flat modules are locally free, observe that a class $\left[M\right]\in c\left(\pi_0\mathbf M^n\left(B\right)\right)$ is locally determined by an equivalence class of bounded complexes
\begin{equation} \label{sequence}
\xymatrix{ M_1\ar[r]^d & M_2\ar[r]^d & \cdots\ar[r]^d & M_s}
\end{equation}
where $s$ is some natural number and $M_i$ is a free $\mathscr O_X\left(U\right)\otimes_R^{\mathbb L}B$-module for all $i$; again we have used the property that perfect complexes are quasi-isomorphic to bounded and degreewise flat ones. Now denote by $i_k$ the rank of the module $M_k$ in representative \eqref{sequence} and consider the scheme defined for all $B\in\mathfrak{Alg}_R$ through the functor of points
\begin{equation} \label{subscheme}
S\left(B\right):=\left\{\left(D_i\right)\in\underset{k=1,\ldots,s-1}{\prod}\mathrm{Mat}_{i_k,i_{k+1}}\left(\mathscr O_X\left(U\right)\otimes_R^{\mathbb L}B\right)\text{ s.t. }D_i^2=0\right\}.
\end{equation}
Formula \eqref{subscheme} determines a closed subscheme of 
\begin{equation*}
\underset{k=1,\ldots,s-1}{\prod}\mathrm{Mat}_{i_k,i_{k+1}}\left(\mathscr O_X\left(U\right)\otimes_R^{\mathbb L}B\right)
\end{equation*}
and provides a local description of $c\left(\pi_0\mathbf M^n\left(B\right)\right)$; clearly
\begin{equation} \label{commute}
\underset{k=1,\ldots,s-1}{\prod}\mathrm{Mat}_{i_k,i_{k+1}}\left(\mathscr O_X\left(U\right)\otimes_R^{\mathbb L}A\right)\simeq\underset{\underset{\alpha}{\longrightarrow}}{\mathrm{lim}}\;\underset{k=1,\ldots,s-1}{\prod}\mathrm{Mat}_{i_k,i_{k+1}}\left(\mathscr O_X\left(U\right)\otimes_R^{\mathbb L}A_{\alpha}\right)
\end{equation}
and since the subscheme $S\hookrightarrow\underset{k=1,\ldots,s-1}{\prod}\mathrm{Mat}_{i_k,i_{k+1}}$ is defined by finitely many equations, formula \eqref{commute} descends to $S\left(A\right)$, meaning that
\begin{equation} \label{final}
S\left(A\right)\simeq\underset{\underset{\alpha}{\longrightarrow}}{\mathrm{lim}}\;S\left(A_{\alpha}\right).
\end{equation}
Formula \eqref{final} implies formula \eqref{components}, so Condition $\left(6\right)$ holds. \\
Lastly, we have to check Condition $\left(7\right)$, so fix a complete discrete local Noetherian $R$-algebra $A$ and a perfect complex $\mathscr E$ of $\mathscr O_X\otimes_R^{\mathbb L}A$; again the assumptions on $\mathscr E$ allow us to substitute it with a bounded complex $\mathscr F$ of flat $\mathscr O_X\otimes^{\mathbb L}_RA$-modules. \\
We first prove the compatibility of the $\mathrm{Ext}$ functors; the properties of $A$ imply that the canonical morphism $A\longrightarrow \hat A$ to the pronilpotent completion
\begin{equation}
\hat A:=\underset{\underset{r}{\longleftarrow}}{\lim}\;\nicefrac{A}{\mathfrak m_A^r}
\end{equation}
is an isomorphism, which we can use to induce $\forall i>0$ a canonical isomorphism 
\begin{equation}
\mathrm{Ext}^i_{X\otimes^{\mathbb L}_RA}\left(\mathscr E,\mathscr F\right)\tilde{\longrightarrow}\mathrm{Ext}^i_{X\otimes^{\mathbb L}_RA}\left(\mathscr E,\underset{\underset{r}{\longleftarrow}}{\lim}\;\nicefrac{\mathscr F}{\mathfrak m_A^r}\right).
\end{equation}
Again, we compute the $\mathrm{Ext}$ groups by using $\mathscr E$ (which is degreewise projective) in the first variable and $\mathscr F$ (which is degreewise flat) in the second variable. The obstruction for the functors $\mathrm{Ext}^i_{X\otimes^{\mathbb L}_RA}$ to commute with the inverse limit $\underset{\underset{r}{\longleftarrow}}{\mathrm{lim}}$ lies in the derived functor $\underset{\underset{r}{\longleftarrow}}{\mathrm{lim}}^1$; however notice that the completeness assumption on $A$ ensures that the tower $\nicefrac{A}{\mathfrak m_A^r}\rightarrow\nicefrac{A}{\mathfrak m_A^{r+1}}$ satisfies the Mittag-Leffler condition (see \cite{We} Section 3.5), and so does the induced tower $\nicefrac{\mathscr F}{\mathfrak m_A^r}\rightarrow\nicefrac{\mathscr F}{\mathfrak m_A^{r+1}}$ (see \cite{Pr1} Section 4.2 for details). In particular we get
\begin{equation}
\underset{\underset{r}{\longleftarrow}}{\mathrm{lim}}^1\;\mathrm{Ext}_{X\otimes_RA}^{i-1}\left(\mathscr F,\nicefrac{\mathscr F}{\mathfrak m_A^r}\right)=0
\end{equation}
which implies
\begin{equation}
\mathrm{Ext}^i_{X\otimes_RA}\left(\mathscr F,\mathscr F\right)\tilde{\longrightarrow}\underset{\underset{r}{\longleftarrow}}{\lim}\; \mathrm{Ext}^i_{X\otimes_RA}\left(\mathscr F,\nicefrac{\mathscr F}{\mathfrak m_A^r}\right)\qquad\qquad\forall i\neq 1.
\end{equation}
At last, we show the compatibility condition on components, i.e. we want to prove that the push-forward map
\begin{equation} \label{push comp}
c\left(\pi_0\mathbf M^n\left(A\right)\right)\longrightarrow\underset{\underset{r}{\longleftarrow}}{\mathrm{lim}}\,c\left(\pi_0\mathbf M^n\left(\nicefrac{A}{\mathbf m_A^r}\right)\right)
\end{equation}
is bijective. This basically means to show that any inverse system 
\begin{equation*}
\left\{\mathscr E_r\text{ s.t. }\mathscr E_r\text{ perfect complex of }\mathscr O_X\otimes^{\mathbb L}_R\nicefrac{A}{\mathfrak m_A^r\text{-modules}}\right\}_{r\in\mathbb N}
\end{equation*}
determines uniquely a perfect $\mathscr O_X\otimes^{\mathbb L}_RA$-module in complexes via map \eqref{push comp}; such a statement is precisely the version of Grothendieck Existence Theorem for perfect complexes: for a proof see \cite{Lu3} Theorem 3.2.2 or \cite{Ha-LePr} Section 3.\\
It follows that Condition $\left(7\right)$ holds as well, so this completes the proof. 
\end{proof}
\begin{rem}
Consider the derived stack $\mathbb R\mathcal Perf_X$ defined by formula \eqref{RPerf}; clearly for all $n\geq 0$ the derived geometric stack $\mathbb R\mathcal Perf_X^n$ is an open substack of $\mathbb R\mathcal Perf_X$ and moreover
\begin{equation*}
\mathbb R\mathcal Perf_X\simeq\underset{n}{\bigcup}\;\mathbb R\mathcal Perf_X^n
\end{equation*}
so we recover the local geometricity of the stack $\mathbb R\mathcal Perf_X$ studied by To\" en and Vaqui\' e in \cite{TVa}.
\end{rem}
\subsection{Derived Moduli of Filtered Perfect Complexes}
This section is devoted to the main result of this paper, that is the construction of a derived moduli stack $\mathbb R\mathcal Filt_X$ classifying filtered perfect complexes of $\mathscr O_X$-modules over some reasonable $k$-scheme $X$; (local) geometricity of such a stack will be ensured by some quite natural cohomological finiteness conditions given in terms of the Rees construction (see Section 1.4): actually the very homotopy-theoretical features of the Rees functor collected in Theorem \ref{Rprops} will allow us to mimic most of the results and arguments of Section 2.2, which deal with the corresponding unfiltered situation. \\
In full analogy with what we did in Section 2.2, associate to any given derived scheme $\mathcal X$ over $R$ the cosimplicial differential graded commutative $R$-algebra $O\left(\mathcal X\right)$ defined by formula \eqref{OUps}.
\begin{defn} \label{filtmod}
Define a \emph{filtered derived module} over $\mathcal X$ to be a cosimplicial filtered $O\left(\mathcal X\right)$-module in complexes. 
\end{defn}
More concretely Definition \ref{filtmod} says that a filtered derived module over $\mathcal X$ is a pair $\left(\mathcal M,F\right)$ made of filtered cochain complexes $\left(\mathcal M^m,F\right)$ of $O\left(\mathcal X\right)^m$-modules related by maps
\begin{eqnarray*}
&\partial^i:F^p\mathcal M^m\otimes^{\mathbb L}_{O\left(\mathcal X\right)^m}O\left(\mathcal X\right)^{m+1}\longrightarrow F^p\mathcal M^{m+1}& \\
&\sigma^i:F^p\mathcal M^m\otimes^{\mathbb L}_{O\left(\mathcal X\right)^m}O\left(\mathcal X\right)^{m-1}\longrightarrow F^p\mathcal M^{m-1}&
\end{eqnarray*}
satisfying the usual cosimplicial identities and such that the diagrams 
\begin{equation*}
\tiny{\xymatrix{\mathcal M_m\otimes^{\mathbb L}_{O\left(\mathcal X\right)^m}O\left(\mathcal X\right)^{m+1}\ar[r]\ar@<5pt>[r]\ar@<-5pt>[r] & \mathcal M_{m+1} \\
F^1\mathcal M_m\otimes^{\mathbb L}_{O\left(\mathcal X\right)^m}O\left(\mathcal X\right)^{m+1}\ar@{^{(}->}[u]\ar[r]\ar@<5pt>[r]\ar@<-5pt>[r] & F^1\mathcal M_{m+1}\ar@{^{(}->}[u] \\
F^2\mathcal M_m\otimes^{\mathbb L}_{O\left(\mathcal X\right)^m}O\left(\mathcal X\right)^{m+1}\ar@{^{(}->}[u]\ar[r]\ar@<5pt>[r]\ar@<-5pt>[r] & F^2\mathcal M_{m+1}\ar@{^{(}->}[u] \\
\vdots\ar@{^{(}->}[u] & \vdots\ar@{^{(}->}[u]}\qquad\qquad
\xymatrix{\mathcal M_{m-1} & \mathcal M_m\otimes^{\mathbb L}_{O\left(\mathcal X\right)^m}O\left(\mathcal X\right)^{m-1}\ar@<2.5pt>[l]\ar@<-2.5pt>[l] \\
F^1\mathcal M_{m-1}\ar@{^{(}->}[u] & F^1\mathcal M_m\otimes^{\mathbb L}_{O\left(\mathcal X\right)^m}O\left(\mathcal X\right)^{m-1}\ar@{^{(}->}[u]\ar@<2.5pt>[l]\ar@<-2.5pt>[l] \\
F^2\mathcal M_{m-1}\ar@{^{(}->}[u] & F^2\mathcal M_m\otimes^{\mathbb L}_{O\left(\mathcal X\right)^m}O\left(\mathcal X\right)^{m-1}\ar@{^{(}->}[u]\ar@<2.5pt>[l]\ar@<-2.5pt>[l] \\
\vdots\ar@{^{(}->}[u] & \vdots\ar@{^{(}->}[u]} }
\end{equation*}
commute; in other words a derived filtered module $\left(\mathcal M,F\right)$ is just a nested sequence
\begin{equation*}
\xymatrix{\cdots\ar@{^{(}->}[r] & F^{p+1}\mathcal M\ar@{^{(}->}[r] & F^p\mathcal M\ar@{^{(}->}[r] & \cdots\ar@{^{(}->}[r] & F^2\mathcal M\ar@{^{(}->}[r] & F^1\mathcal M\ar@{^{(}->}[r] & F^0\mathcal M=:\mathcal M}
\end{equation*}
in $\mathfrak{dgMod}\left(\mathcal X\right)$. Notice that a filtered derived module is equipped with three different indexings, one coming from the filtration, one from the differential graded structure and the last one from the cosimplicial structure: a morphism of derived filtered modules will be an arrow preserving all of them, so there is a category of derived filtered modules on $\mathcal X$, which we will denote by $\mathfrak{FdgMod}\left(\mathcal X\right)$. \\
Just like the unfiltered situation analysed in Section 2.2, observe that the projective model structure on filtered cochain complexes given by Theorem \ref{fdg model thm} induces a projective model structure on $\mathfrak{FdgMod}\left(\mathcal X\right)$; in particular a morphism $f:\left(\mathcal M,F\right)\rightarrow\left(\mathcal N,F\right)$ in $\mathfrak{FdgMod}\left(\mathcal X\right)$ is 
\begin{itemize}
\item a weak equivalence if $F^pf^m:F^p\mathcal M^m\rightarrow F^p\mathcal N^m$ is a quasi-isomorphism; 
\item a fibration if $F^pf^m:F^p\mathcal M^m\rightarrow F^p\mathcal N^m$ is degreewise surjective;
\item a cofibration if it has the left lifting property with respect to all fibrations.
\end{itemize}
There is also a natural simplicial structure on the category $\mathfrak{FdgMod}\left(\mathcal X\right)$ again coming from the simplicial structure on $\mathfrak{FdgMod}_R$: more clearly for any $\left(\mathcal M,F\right),\left(\mathcal N,F\right)\in\mathfrak{FdgMod}\left(\mathcal X\right)$ consider the chain complex $\left(\mathrm{HOM}_{\mathcal X}\left(\left(\mathcal M,F\right),\left(\mathcal N,F\right)\right),\delta\right)$ defined by the relations
\begin{eqnarray} \label{filt HOM} 
&\mathrm{HOM}_{\mathcal X}\left(\left(\mathcal M,F\right),\left(\mathcal N,F\right)\right)_n:=\mathrm{Hom}_{O\left(\mathcal X\right)}\left(\left(\mathcal M,F\right),\left(\mathcal N[-n],F\right)\right)& \nonumber \\
&\forall\left(f,F\right)\in {\mathrm{HOM}_{\mathcal X}\left(\left(\mathcal M,F\right),\left(\mathcal N,F\right)\right)_n}\qquad\quad\delta_n\left(\left(f,F\right)\right)\in {\mathrm{HOM}_{\mathcal X}\left(\left(\mathcal M,F\right),\left(\mathcal N,F\right)\right)_{n-1}}& \nonumber\\
&\text{defined by } F^p\left(\delta_n\left(\left(f,F\right)\right)\right):=F^p\bar d^n\circ F^pf-\left(-1\right)^nF^pf\circ F^pd^n&
\end{eqnarray}
and define the Hom spaces just by taking good truncation and denormalisation, i.e. set
\begin{equation*}
\underline{\mathrm{Hom}}_{\mathfrak{FdgMod}\left(\mathcal X\right)}\left(\left(\mathcal M,F\right),\left(\mathcal N,F\right)\right):=\mathbf K\left(\tau_{\geq 0}\mathrm{HOM}_{\mathcal X}\left(\left(\mathcal M,F\right),\left(\mathcal N,F\right)\right)\right).
\end{equation*}
In a similar way, notice that the $\mathrm{HOM}$ complex for filtered derived modules defined by formula \eqref{filt HOM} sheafifies, so we have a well-defined $\mathcal Hom$-sheaf bifunctor
\begin{equation*}
\mathcal Hom_{\mathscr O_{\mathcal X,\bullet}}:\mathfrak{FdgMod}^{op}\left(\mathcal X\right)\times\mathfrak{FdgMod}\left(\mathcal X\right)\longrightarrow\mathfrak{dgMod}\left(\mathcal X\right)
\end{equation*}
and consequently a derived $\mathcal Hom$ sheaf, given by the bifunctor
\begin{eqnarray*}
\mathbb R\mathcal Hom_{\mathscr O_{\mathcal X,\bullet}}:&\mathfrak{FdgMod}^{op}\left(\mathcal X\right)\times\mathfrak{FdgMod}\left(\mathcal X\right)&\longrightarrow\qquad\quad\mathfrak{dgMod}\left(\mathcal X\right) \\
&\left(\left(\mathcal M,F\right),\left(\mathcal N,F\right)\right)&\longmapsto\mathcal Hom\left(\mathbf Q\left(\left(\mathcal M,F\right)\right),\left(\mathcal N,F\right)\right)
\end{eqnarray*}
where $\mathbf Q\left(\left(\mathcal M,F\right)\right)$ is a functorial cofibrant replacement for $\left(\mathcal M,F\right)$. We can also define $\mathcal Ext$ sheaves for the category $\mathfrak{FdgMod}\left(\mathcal X\right)$ by denoting for all $\left(\mathcal M,F\right),\left(\mathcal N,F\right)\in\mathfrak{FdgMod}\left(\mathscr O_{\mathcal X,\bullet}\right)$
\begin{equation*}
\mathcal Ext^i_{\mathscr O_{\mathcal X,\bullet}}\left(\left(\mathcal M,F\right),\left(\mathcal N,F\right)\right):=\mathcal H^i\left(\pi^0\mathcal X,\mathbb R\mathcal Hom_{\mathscr O_{\mathcal X,\bullet}}\left(\left(\mathcal M,F\right),\left(\mathcal N,F\right)\right)\right).
\end{equation*}
The Rees construction given by formula \eqref{Rees defn} readily extends to this context, as well; more formally consider the derived scheme $\mathcal X\left[t\right]$ over $R\left[t\right]$ whose C\v ech nerve is defined in simplicial degree $m$ by the structure sheaf
\begin{equation*}
\mathscr O_{\check{\mathcal X}\left[t\right]_m}:=\mathscr O_{\check{\mathcal X}_m}\left[t\right] 
\end{equation*}
so that its cosimplicial differential graded commutative $R\left[t\right]$-algebra of global sections is given in cosimplicial level $m$ by
\begin{equation*}
O\left(\mathcal X\left[t\right]\right)^m:=\Gamma\left(\check{\mathcal X}_m,\mathscr O_{\check{\mathcal X}_m}\right)\left[t\right].
\end{equation*}
Again, there is a natural $\mathbb G_m$-action on the derived scheme $\mathcal X\left[t\right]$ defined on rings of functions in level $m$ as
\begin{eqnarray} \label{Gm sheaf}
\mathbb G_m\times O\left(\mathcal X\left[t\right]\right)^m&\xrightarrow{\hspace*{1cm}}&O\left(\mathcal X\left[t\right]\right)^m \nonumber \\
\left(\lambda,\varrho\left(t\right)\right)\quad\quad\;\;&\longmapsto&\varrho\left(\lambda^{-1}t\right)
\end{eqnarray} 
therefore there is a category $\mathbb G_m\text{-}\mathfrak{dgMod}\left(\mathcal X\left[t\right]\right)$ of graded derived modules over $\mathcal X\left[t\right]$, where the extra grading is induced by the action given by the formula \eqref{Gm sheaf}. Clearly the $\mathbb G_m$-equivariant projective model structure determined by Theorem \ref{Gm dg model thm} extends to the category $\mathbb G_m\text{-}\mathfrak{dgMod}\left(\mathcal X\left[t\right]\right)$: in particular a morphism $f:\mathcal M\rightarrow\mathcal N$ in $\mathbb G_m\text{-}\mathfrak{dgMod}\left(\mathcal X\left[t\right]\right)$ is 
\begin{itemize}
\item a weak equivalence if $f^m:\mathcal M^m\rightarrow\mathcal N^m$ is a $\mathbb G_m$-equivariant quasi-isomorphism; 
\item a fibration if $f^m:\mathcal M^m\rightarrow\mathcal N^m$ is a $\mathbb G_m$-equivariant degreewise surjection;
\item a cofibration if it has the left lifting property with respect to all fibrations.
\end{itemize}
Now define the Rees module associated to the filtered derived module $\left(\mathcal M,F\right)$ over $\mathcal X$ to be the derived module $\mathrm{Rees}\left(\left(\mathcal M,F\right)\right)$ over $\mathcal X\left[t\right]$ determined in cosimplicial level $m$ by the (bigraded) complex of $O\left(\mathcal X\left[t\right]\right)^m$-modules
\begin{equation} \label{Rees sheaf}
\mathrm{Rees}\left(\left(\mathcal M,F\right)\right):=\underset{p=0}{\overset{\infty}{\bigoplus}}F^p\mathcal M^m\cdot t^{-p}.
\end{equation}
The construction \eqref{Rees sheaf} is clearly natural in all entries, so we get a functor
\begin{equation} \label{Rees sheaf funct}
\mathrm{Rees}:\mathfrak{FdgMod}\left(\mathcal X\right)\longrightarrow\mathbb G_m\text{-}\mathfrak{dgMod}\left(\mathcal X\left[t\right]\right).
\end{equation}
which is immediately seen to have -- \emph{mutatis mutandis} -- all properties stated by Theorem \ref{Rprops}. In particular, for all $\left(\mathcal M,F\right),\left(N,F\right)\in\mathfrak{FdgMod}\left(\mathcal X\right)$ define the groups
\begin{equation} \label{filt ext}
\mathrm{Ext}^{n-i}_{\mathcal X}\left(\left(\mathcal M,F\right),\left(\mathcal N,F\right)\right):=\pi_i\underline{\mathrm{Hom}}_{\mathfrak{FdgMod}\left(\mathcal X\right)}\left(\left(\mathcal M,F\right),\left(\mathcal N\left[-n\right],F\right)\right)
\end{equation}
and observe that Theorem \ref{Rprops}.2 implies
\begin{equation} \label{rees ext}
\mathrm{Ext}^i_{\mathcal X}\left(\left(\mathcal M,F\right),\left(\mathcal N,F\right)\right)=\mathrm{Ext}^i_{\mathcal X\left[t\right]}\left(\mathrm{Rees}\left(\left(\mathcal M,F\right)\right),\mathrm{Rees}\left(\left(\mathcal N,F\right)\right)\right)^{\mathbb G_m}.
\end{equation}
\begin{rem} \label{filt loc-to-glob}
Because of formula \eqref{rees ext} and the exactness of the functor $\left(-\right)^{\mathbb G_m}$, we have that the local-to-global spectral sequence extends to the filtered context, i.e there is a convergent spectral sequence
\begin{equation*}
H^p\left(\pi^0\mathcal X,\mathcal Ext^q_{\mathscr O_{\mathcal X,\bullet}}\left(\left(\mathcal M,F\right),\left(\mathcal M,F\right)\right)\right)\Longrightarrow\mathrm{Ext}_{\mathcal X}^{p+q}\left(\left(\mathcal M,F\right),\left(\mathcal M,F\right)\right).
\end{equation*}
\end{rem}
\begin{defn}
Define a \emph{filtered derived quasi-coherent sheaf} over $\mathcal X$ to be a filtered derived module $\left(\mathcal M,F\right)$ for which 
and $F^p\mathcal M\in\mathfrak{dgMod}_{\mathrm{cart}}\left(\mathcal X\right)$ for all $p$.
\end{defn}
Denote by $\mathfrak{FdgMod}_{\mathrm{cart}}\left(\mathcal X\right)$ the full subcategory of $\mathfrak{FdgMod}\left(\mathcal X\right)$ consisting of filtered quasi-coherent derived sheaves: the homotopy-theoretic properties of $\mathfrak{FdgMod}\left(\mathcal X\right)$ induce a simplicial structure and a well-behaved subcategory of weak equivalences on it.
\begin{rem} \label{Rees FdCART}
The Rees functor \eqref{Rees sheaf funct} respects quasi-coherence, meaning that it restricts to a functor
\begin{equation*}
\mathrm{Rees}:\mathfrak{FdgMod}_{\mathrm{cart}}\left(\mathcal X\right)\longrightarrow\mathbb G_m\text{-}\mathfrak{dgMod}_{\mathrm{cart}}\left(\mathcal X\left[t\right]\right).
\end{equation*}
which obviously still maps weak equivalences to weak equivalences.
\end{rem} 
Now our goal is to study derived moduli of filtered derived quasi-coherent sheaves by means of Lurie-Pridham representability: in order to reach this we will literally follow the strategy described in Section 2.2 when tackling moduli of unfiltered complexes; in particular we will prove filtered analogues of Proposition \ref{dCART homog smooth}, Corollary \ref{DMQCS} and Theorem \ref{Perf repr}. In the following, given any filtered derived quasi-coherent sheaf $\left(\mathscr E,F\right)$ over some derived geometric stack denote by $\hat F$ the filtration induced by (derived) base-change and by $\tilde F$ the one induced on quotients.\\
From now on fix $R$ to be an ordinary (underived) $k$-algebra and $X$ to be a quasi-compact semi-separated scheme over $R$; define the functor
\begin{eqnarray} \label{FdCART}
Fd\mathrm{CART}_{X}:&\mathfrak{dg}_b\mathfrak{Nil}^{\leq 0}_R&\xrightarrow{\hspace*{1.75cm}}\mathfrak{sCat} \nonumber \\
&A&\longmapsto\left(\mathfrak{FdgMod}_{\mathrm{cart}}\left(X\otimes^{\mathbb L}_RA\right)\right)^c
\end{eqnarray}
where again $\left(\mathfrak{FdgMod}_{\mathrm{cart}}\left(X\otimes^{\mathbb L}_RA\right)\right)^c$ is the full simplicial subcategory of $\mathfrak{FdgMod}_{\mathrm{cart}}\left(X\otimes^{\mathbb L}_RA\right)$ on cofibrant objects.
\begin{lemma} \label{poly square-zero}
Let $f:A\rightarrow B$ a square-zero extension in $\mathfrak{dgAlg}^{\leq 0}_R$; then the induced morphism
\begin{eqnarray*}
&\mathsf f:\qquad A\left[t\right]&\longrightarrow \quad\quad B\left[t\right] \\
&A_n\left[t\right]\ni\underset{i}{\sum}a_it^i&\overset{\mathsf f_n}{\longmapsto}\underset{i}{\sum}f\left(a_i\right)t^i\in B_n\left[t\right]
\end{eqnarray*}
is a square-zero extension in $\mathfrak{dgAlg}^{\leq 0}_{R\left[t\right]}$. Moreover $\mathsf f$ is acyclic whenever so is $f$.
\end{lemma}
\begin{proof}
Denote $I:=\ker\left(f\right)$; then $\ker\left(\mathsf f\right)=I\left[t\right]$, where
\begin{equation*}
I\left[t\right]:\qquad\xymatrix{\cdots\ar[r]^{\mathsf d} & I_2\left[t\right]\ar[r]^{\mathsf d} & I_1\left[t\right]\ar[r]^{\mathsf d} & I_0\left[t\right]}.
\end{equation*}
In particular $I\left[t\right]^2=0\Leftrightarrow I^2=0$ and $H^i\left(I\left[t\right]\right)=0\Leftrightarrow H^i\left(I\right)=0$.
\end{proof}
\begin{prop} \label{FdCART homog smooth}
Functor \eqref{FdCART} is $2$-homogeneous and formally $2$-quasi-smooth.
\end{prop}
\begin{proof}
The argument of Proposition \ref{dCART homog smooth} applies to this context as well, we sketch the main adjustments. \\
In order to verify that $Fd\mathrm{CART}_{X}$ is $2$-homogeneous take a square-zero extension $A\rightarrow B$ and a morphism $C\rightarrow B$  in $\mathfrak{dg}_b\mathfrak{Nil}^{\leq 0}_R$ and fix $\left(\mathscr E,F\right),\left(\mathscr E',F\right)\in Fd\mathrm{CART}_{X}\left(A\times_BC\right)$. Cofibrancy of such pairs -- which by Proposition \ref{fdg cofibrant} implies filtration-levelwise degreewise projectiveness -- ensures that the commutative square of simplicial sets
\begin{equation*}
\tiny{\xymatrix@C=11pt{\underline{\mathrm{Hom}}_{Fd\mathrm{CART}\left(A\times_BC\right)}\left(\left(\mathscr E,F\right),\left(\mathscr E',F\right)\right)\ar[r]\ar[d] & \underline{\mathrm{Hom}}_{Fd\mathrm{CART}\left(A\right)}\left(\left(\mathscr E\otimes_{A\times_BC}A,\hat F\right),\left(\mathscr E'\otimes_{A\times_BC}A,\hat F\right)\right)\ar[d] \\
\underline{\mathrm{Hom}}_{Fd\mathrm{CART}\left(C\right)}\left(\left(\mathscr E\otimes_{A\times_BC}C,\hat F\right),\left(\mathscr E'\otimes_{A\times_BC}C,\hat F\right)\right)\ar[r] & \underline{\mathrm{Hom}}_{Fd\mathrm{CART}\left(B\right)}\left(\left(\mathscr E\otimes_{A\times_BC}B,\hat F\right),\left(\mathscr E'\otimes_{A\times_BC}B,\hat F\right)\right) }}
\end{equation*}
is actually Cartesian. Then fix $\left(\mathscr E_A,F_A\right)\in Fd\mathrm{CART}_{X}\left(A\right)$ and $\left(\mathscr E_C,F_C\right)\in Fd\mathrm{CART}_{X}\left(C\right)$, let $\alpha:\left(\mathscr E_A\otimes_AB,\hat F_A\right)\rightarrow\left(\mathscr E_C\otimes_C B,\hat F_C\right)$ be a filtered isomorphism and define 
\begin{equation} \label{ess surj}
\left(\mathscr E,F\right):=\left(\mathscr E_A\otimes_{\alpha,\mathscr E_C\otimes_CB}\mathscr E_C,F_A\otimes F_C\right)\simeq\left(\mathscr E_C\otimes_{\alpha,\mathscr E_A\otimes B}\mathscr E_A,F_C\otimes F_A\right).
\end{equation}
The filtered derived module $\left(\mathscr E,F\right)$ is actually a cofibrant filtered derived quasi-coherent sheaf on $X\otimes_R\left(A\otimes_BC\right)$, namely $\left(\mathscr E,F\right)\in Fd\mathrm{CART}_X\left(A\times_BC\right)$; we also have that 
\begin{equation*}
\left(\mathscr E,F\right)\otimes_{A\times_BC}A\simeq\left(\mathscr E_A,F_A\right)\qquad\qquad\left(\mathscr E,F\right)\otimes_{A\times_BC}C\simeq\left(\mathscr E_C,F_C\right)
\end{equation*}
and this completes the proof that $Fd\mathrm{CART}_{X}$ is a $2$-homogeneous functor. \\
Now we want to prove that the functor $Fd\mathrm{CART}_X$ is formally $2$-quasi-smooth, so we start by showing that $\underline{\mathrm{Hom}}_{Fd\mathrm{CART}_X}$ is formally quasi-smooth; for this reason take a square-zero extension $A\rightarrow B$ in $\mathfrak{dg}_b\mathfrak{Nil}_R^{\leq 0}$ and consider the induced $R\left[t\right]$-linear morphism $A\rightarrow B$, as done in Lemma \ref{poly square-zero}. Let $\left(\mathscr E,F\right),\left(\mathscr E',F\right)\in Fd\mathrm{CART}\left(A\right)$ and look at the induced morphism of simplicial sets
\begin{equation} \label{filt hom homog}
\underline{\mathrm{Hom}}_{Fd\mathrm{CART}\left(A\right)}\left(\left(\mathscr E,F\right),\left(\mathscr E',F\right)\right)\longrightarrow\underline{\mathrm{Hom}}_{Fd\mathrm{CART}\left(B\right)}\left(\left(\mathscr E\otimes_AB,\hat F\right),\left(\mathscr E'\otimes_AB,\hat F\right)\right).
\end{equation}
By Theorem \ref{Rprops}.2, map \eqref{filt hom homog} is a (trivial) fibration if and only if the map
\begin{equation*}
\tiny{\xymatrix{\underline{\mathrm{Hom}}_{d\mathrm{CART}\left(A\left[t\right]\right)}\left(\mathrm{Rees}\left(\left(\mathscr E,F\right)\right),\mathrm{Rees}\left(\left(\mathscr E',F\right)\right)\right)^{\mathbb G_m}\ar[d] \\ 
\underline{\mathrm{Hom}}_{d\mathrm{CART}\left(B\left[t\right]\right)}\left(\mathrm{Rees}\left(\left(\mathscr E\otimes_AB,\hat F\right)\right),\mathrm{Rees}\left(\left(\mathscr E'\otimes_AB,\hat F\right)\right)\right)^{\mathbb G_m}}}
\end{equation*}
is a (trivial) fibration, which in turn is equivalent to say that
\begin{equation} \label{apply dCART}
\tiny{\xymatrix{\underline{\mathrm{Hom}}_{d\mathrm{CART}\left(A\left[t\right]\right)}\left(\mathrm{Rees}\left(\left(\mathscr E,F\right)\right),\mathrm{Rees}\left(\left(\mathscr E',F\right)\right)\right)\ar[d] \\ 
\underline{\mathrm{Hom}}_{d\mathrm{CART}\left(B\left[t\right]\right)}\left(\mathrm{Rees}\left(\left(\mathscr E\otimes_AB,\hat F\right)\right),\mathrm{Rees}\left(\left(\mathscr E'\otimes_AB,\hat F\right)\right)\right)}}
\end{equation}
is a (trivial) fibration, as functor $\left(-\right)^{\mathbb G_m}$ is exact. Now by Lemma \ref{poly square-zero} we have that the morphism of $R\left[t\right]$-algebras $A\left[t\right]\rightarrow B\left[t\right]$ is a square-zero extension that is acyclic whenever so is $A\rightarrow B$, while Proposition \ref{dCART homog smooth} ensures that map \eqref{apply dCART} is a fibration which is trivial if $A\left[t\right]\rightarrow B\left[t\right]$ is acyclic: these observations conclude the proof of the formal quasi-smoothness of $\underline{\mathrm{Hom}}_{Fd\mathrm{CART}_X}$. In order to finish the proof, it only remains to show that the base-change morphism
\begin{equation} \label{filt map 2-fib}
Fd\mathrm{CART}_{X}\left(A\right)\longrightarrow Fd\mathrm{CART}_{X}\left(B\right)
\end{equation}
is a $2$-fibration, which should be trivial whenever the square-zero extension $A\rightarrow B$ is acyclic. Note first that the computations in \cite{Pr4} Section $7$, together with the definition of Ext groups for filtered derived modules given by formula \eqref{filt ext} and the isomorphism provided by formula \eqref{rees ext}, imply that obstructions to lifting a filtered quasi-coherent module $\left(\mathscr E,F\right)\in Fd\mathrm{CART}_{X}\left(B\right)$ to $Fd\mathrm{CART}_{X}\left(A\right)$ lie in the group
\begin{equation*}
\mathrm{Ext}^2_{X\otimes_R^{\mathbb L}B}\left(\left(\mathscr E,F\right),\left(\mathscr E\otimes_B I,\hat F\right)\right)\simeq\mathrm{Ext}^2_{\left(X\otimes_R^{\mathbb L}B\right)\left[t\right]}\left(\mathrm{Rees}\left(\left(\mathscr E,F\right)\right),\mathrm{Rees}\left(\left(\mathscr E\otimes_B I,\hat F\right)\right)\right)^{\mathbb G_m}.
\end{equation*}
so if $H^*\left(I\right)=0$ then map \eqref{filt map 2-fib} is a trivial $2$-fibration. Now fix $\left(\mathscr E,F\right)\in Fd\mathrm{CART}_X\left(A\right)$, $\left(\mathscr H,G\right)\in Fd\mathrm{CART}_X\left(B\right)$ and let $\theta:\left(\mathscr E\otimes_AB,\hat F\right)\rightarrow\left(\mathscr H,G\right)$ be a homotopy equivalence in $ Fd\mathrm{CART}_X\left(B\right)$: we want to prove that $\theta$ lifts to a homotopy equivalence $\check{\theta}:\left(\mathscr E,F\right)\rightarrow{\left(\check{\mathscr H},\check G\right)}$ in $Fd\mathrm{CART}_X\left(A\right)$. Apply the Rees functor \eqref{Rees sheaf funct} to all data: by Theorem \ref{Rprops} we end up to be given a homotopy equivalence
\begin{equation*}
\mathrm{Rees}\left(\theta\right):\mathrm{Rees}\left(\left(\mathscr E\otimes_AB,\hat F\right)\right)\longrightarrow\mathrm{Rees}\left(\left(\mathscr H,G\right)\right)
\end{equation*}
in $d\mathrm{CART}_X\left(B\left[t\right]\right)$ which by Proposition \ref{dCART homog smooth} lifts to a homotopy equivalence in $d\mathrm{CART}_X\left(A\left[t\right]\right)$; in particular this ensures that suitable lifts $\check{\theta}$ of the homotopy equivalence $\theta$ do exist, again by Theorem \ref{Rprops}. This completes the proof.
\end{proof}
We can build upon Proposition \ref{FdCART homog smooth} a filtered version of Corollary \ref{DMQCS}. 
\begin{cor} \label{FDMQCS} 
Let 
\begin{equation*}
\mathbf M:\mathfrak{Alg}_{H^0\left(R\right)}\rightarrow\mathfrak{sCat}
\end{equation*}
be a presheaf satisfying the following conditions:
\begin{enumerate}
\item $\mathbf M$ is a $n$-truncated hypersheaf;
\item $\mathbf M$ is open in the functor 
\begin{equation*}
A\longmapsto\pi^0\mathcal W\left(\mathfrak{FdgMod}\left(X\otimes^{\mathbb L}_R A\right)_{\mathrm{cart}}\right)\qquad A\in\mathfrak{Alg}_{ H^0\left(R\right)};
\end{equation*}
\item If $\left\{f_{\alpha}:A\rightarrow B_{\alpha}\right\}_{\alpha}$ is an \' etale cover, then $\left(\mathscr E,F\right)\in\pi_0\pi^0\mathcal W\left(\mathfrak{FdgMod}\left(X\otimes^{\mathbb L}_R A\right)_{\mathrm{cart}}\right)$ lies in the essential image of $\pi_0\mathbf M\left(A\right)$ whenever $\left(\left(f_{\alpha}\right)^*\mathscr E,\hat F\right)$ is in the essential image of $\pi_0\mathbf M\left(B_{\alpha}\right)$ for all $\alpha$;
\item For all finitely generated $A\in\mathfrak{Alg}_{H^0\left(R\right)}$ and all $\left(\mathscr E,F\right)\in\mathbf M\left(A\right)$, the functors 
\begin{equation*}
\mathrm{Ext}^i_{X\otimes^{\mathbb L}_RA}\left(\left(\mathscr E,F\right),\left(\mathscr E\otimes^{\mathbb L}_A-,\hat F\right)\right):\mathfrak{Mod}_A\longrightarrow\mathfrak{Ab}
\end{equation*}
preserve filtered colimits $\forall i\neq 1$;
\item For all finitely generated integral domains $A\in\mathfrak{Alg}_{H^0\left(R\right)}$ and all $\left(\mathscr E,F\right)\in\mathbf M\left(A\right)$, the groups $\mathrm{Ext}^i_{X\otimes_R^{\mathbb L}A}\left(\left(\mathscr E,F\right),\left(\mathscr E,F\right)\right)$ are finitely generated $A$-modules;
\item The functor 
\begin{equation*}
c(\pi_0\mathbf M):\mathfrak{Alg}_{H^0\left(R\right)}\longrightarrow\mathfrak{Set}
\end{equation*}
of components of the groupoid $\pi_0\mathbf M$ preserves filtered colimits;
\item For all complete discrete local Noetherian normal $H^0\left(R\right)$-algebras $A$, all $\left(\mathscr E,F\right)\in\mathbf M\left(A\right)$ and all $i>0$ the canonical maps
\begin{eqnarray*}
&c(\pi_0\mathbf M\left(A\right))\longrightarrow\underset{\underset{r}{\longleftarrow}}{\mathrm{lim}}\;c(\pi_0\mathbf M\left(\nicefrac{A}{\mathfrak m_A^r}\right))& \\
&\mathrm{Ext}^i_{X\otimes_R^{\mathbb L}A}\left(\left(\mathscr E,F\right),\left(\mathscr E,F\right)\right)\longrightarrow\underset{\underset{r}{\longleftarrow}}{\mathrm{lim}}\;\mathrm{Ext}^i_{\mathcal F\otimes_R^{\mathbb L}A}\left(\left(\mathscr E,F\right),\left(\nicefrac{\mathscr E}{\mathfrak m_A^r},\tilde F\right)\right)&\forall i<0
\end{eqnarray*}
are isomorphisms.
\end{enumerate}
Let
\begin{equation*}
\breve{\mathbf M}:\mathfrak{dg}_b\mathfrak{Nil}^{\leq 0}_R\longrightarrow\mathfrak{sCat}
\end{equation*}
be the full simplicial subcategory of $\mathcal W\left(Fd\mathrm{CART}_{X}\left(A\right)\right)$ consisting of objects $\left(\mathscr F,F\right)$ for which the pair $\left(\mathscr F\otimes_A H^0\left(A\right),\hat F\right)$ is weakly equivalent in $\mathfrak{FdgMod}_{\mathrm{cart}}\left(X\otimes_R^{\mathbb L} H^0\left(A\right)\right)$ to an object of $\mathbf M\left(H^0\left(A\right)\right)$. Then the functor $\bar W\breve{\mathbf M}$ is (the restriction to $\mathfrak{dg}_b\mathfrak{Nil}^{\leq 0}_R$ of) a derived geometric $n$-stack.
\end{cor}
\begin{proof}
The same arguments used to prove Corollary \ref{DMQCS} carry over to this context, using Proposition \ref{FdCART homog smooth} in place of Proposition \ref{dCART homog smooth} and observing -- as done in the proof of Proposition \ref{FdCART homog smooth} itself -- that
\begin{eqnarray*}
\mathrm D^i_{\left(\mathscr E,F\right)}\left(\bar W\breve{\mathbf M},M\right)&\simeq&\mathrm{Ext}^{i+1}_{X\otimes_A^{\mathbb L} M}\left(\left(\mathscr E,F\right),\left(\mathscr E\otimes^{\mathbb L}_AM,\hat F\right)\right) \\
&\simeq&\mathrm{Ext}_{\left(X\otimes_A^{\mathbb L} M\right)\left[t\right]}\left(\mathrm{Rees}\left(\left(\mathscr E,F\right)\right),\mathrm{Rees}\left(\left(\mathscr E\otimes^{\mathbb L}_AM,\hat F\right)\right)\right)^{\mathbb G_m}.
\end{eqnarray*}
Also Condition $\left(2\right)$ tells us that
\begin{equation*} 
\breve{\mathbf M}\left(A\right)\approx\mathbf M\left(H^0\left(A\right)\right)\times^h_{\mathcal W\left(\mathfrak{FdgMod}\left(X\otimes^{\mathbb L}_R H^0\left(A\right)\right)_{\mathrm{cart}}\right)} \mathcal W\left(\mathfrak{FdgMod}\left(X\otimes^{\mathbb L}_R A\right)_{\mathrm{cart}}\right)
\end{equation*}
which is the filtered analogue of formula \eqref{openness reps}. \\
The only claim which still needs to be verified is the one saying that $\bar W\mathcal W\left(\mathbf M\right)$ is an \' etale hypersheaf: observe that, by combining Condition $\left(3\right)$ and Proposition \ref{hypersheaf}, this amounts to check that $\bar W\mathcal W\left(Fd\mathrm{CART}_X\right)$ is a hypersheaf for the \' etale topology, thus fix an \' etale hypercover $B\rightarrow B^{\bullet}$ and consider the induced map
\begin{equation} \label{holim FdCART}
\bar W\mathcal W\left(Fd\mathrm{CART}_X\right)\left(B\right)\longrightarrow\underset{\underset{n}{\longleftarrow}}{\mathrm{holim}}\left(\bar W\mathcal W\left(Fd\mathrm{CART}_X\right)\left(B^{\bullet}\right)\right).
\end{equation}
Let us apply the Rees construction to map \eqref{holim FdCART}: by Remark \ref{Rees FdCART} the Rees functor \eqref{Rees sheaf funct} descends to quasi-coherent modules and as a consequence of Theorem \ref{Rprops} it preserves cofibrant objects, so map \eqref{holim FdCART} becomes
\begin{equation} \label{holim dCARTt}
\bar W\mathcal W\left(d\mathrm{CART}_{X\left[t\right]}\right)\left(B\left[t\right]\right)\longrightarrow\underset{\underset{n}{\longleftarrow}}{\mathrm{holim}}\left(\bar W\mathcal W\left(d\mathrm{CART}_{X\left[t\right]}\right)\left(B\left[t\right]^{\bullet}\right)\right)
\end{equation}
and map \eqref{holim dCARTt} is actually a weak equivalence because $\bar W\mathcal W\left(d\mathrm{CART}_{X\left[t\right]}\right)$ is a hypersheaf for the \' etale topology over $\mathfrak{dgAlg}_{R\left[t\right]}^{\leq 0}$, as observed in the proof of Corollary \ref{DMQCS}. Arguing backwards, this implies that $\bar W\mathcal W\left(Fd\mathrm{CART}_X\right)$ is itself an \' etale hypersheaf, and this completes the proof.
\end{proof}
Now we are ready to discuss derived moduli of filtered perfect complexes; for this reason consider the functor
\begin{eqnarray} \label{und filt prf cmplx} 
&\mathbf M^n_{\mathrm{filt}}:\mathfrak{Alg}_R\longrightarrow\mathfrak{sCat}\qquad\qquad\qquad\qquad\qquad\qquad\qquad\qquad\qquad\qquad\qquad\qquad\qquad\qquad& \nonumber \\
&A\longmapsto\mathbf M^n_{\mathrm{filt}}\left(A\right):=\text{full simplicial subcategory}\qquad\qquad\qquad\qquad\qquad& \nonumber \\ 
&\qquad\qquad\qquad\qquad\qquad\qquad\quad\text{of filtered complexes $\left(\mathscr E,F\right)$ of $\mathscr O_X\otimes^{\mathbb L}_RA$-modules such that:}& \nonumber \\ 
&a)\;\text{$F$ is bounded below}& \nonumber \\
&\quad\; b)\;\text{$F^p\mathscr E$ is perfect for all $p$}& \nonumber \\ 
&\qquad\qquad\qquad\quad\quad\;\;\;\, c)\;\text{$\mathrm{Ext}_{X\otimes_R^{\mathbb L}A}^i\left(\left(\mathscr E,F\right),\left(\mathscr E,F\right)\right)=0$ for $i<-n$}&
\end{eqnarray}
which classifies filtered perfect $\mathscr O_X$-modules in complexes with trivial $\mathrm{Ext}$ groups in higher negative degrees.
\begin{thm} \label{filt Perf repr}
In the above notations, assume that the scheme $X$ is also proper; then functor \eqref{und filt prf cmplx} induces a derived geometric $n$-stack $\mathbb R\mathcal Filt_X^n$.
\end{thm}
\begin{proof}
We have to prove that functor \eqref{und filt prf cmplx} satisfies the conditions of Corollary \ref{FDMQCS}: again our strategy consists of adapting the proof of Theorem \ref{Perf repr} to the filtered case by means of the homotopy-theoretical properties of the Rees construction. \\
First of all, notice the vanishing assumption about the Ext groups given by Axiom $\left(c\right)$ corresponds exactly to the $n$-truncation of the presheaf $\mathbf M^n_{\mathrm{filt}}$, which gives us Condition $\left(1\right)$. \\
As regards Condition $\left(2\right)$, let us show the openness of $\mathbf M_{\mathrm{filt}}^n$ inside $\pi^0\mathcal W\left(\mathfrak{FdgMod}\left(X\otimes^{\mathbb L}_R-\right)_{\mathrm{cart}}\right)$, which essentially amounts to prove that the morphism of formal groupoids
\begin{equation} \label{filt form etale} 
\pi_0\mathbf M_{\mathrm{filt}}^n\hookrightarrow\pi_0\pi^0\mathcal W\left(\mathfrak{dgMod}\left(X\otimes^{\mathbb L}_R-\right)_{\mathrm{cart}}\right) 
\end{equation}
is formally \' etale. Fix a square-zero extension $I\hookrightarrow A\twoheadrightarrow B$ and an object $\left(\mathscr E,F\right)\in\mathbf M_{\mathrm{filt}}^n\left(B\right)$, then look at the maps which morphism \eqref{filt form etale} induces on tangent and obstruction spaces; by combining the results in \cite{Lie} Section 3 and \cite{Pr1} Theorem 4.12 about the Deformation Theory of perfect complexes and quasi-coherent modules respectively with the homotopy-theoretical features of the Rees functor established by Theorem \ref{Rprops} and formula \eqref{rees ext} we have that
\begin{itemize}
\item the tangent space to the functor $\pi_0\mathbf M_{\mathrm{filt}}^n$ at $\left(\mathscr E,F\right)$ is given by
\begin{equation*} 
\mathrm{Ext}_{X\otimes^{\mathbb L}_RA}^1\left(\left(\mathscr E,F\right),\left(\mathscr E\otimes^{\mathbb L}_B I,\hat F\right)\right)\simeq\mathrm{Ext}_{\left(X\otimes^{\mathbb L}_RA\right)\left[t\right]}^1\left(\mathrm{Rees}\left(\left(\mathscr E,F\right)\right),\mathrm{Rees}\left(\left(\mathscr E\otimes^{\mathbb L}_B I,\hat F\right)\right)\right)^{\mathbb G_m}
\end{equation*}
\item a functorial obstruction space for $\pi_0\mathbf M_{\mathrm{filt}}^n$ at $\left(\mathscr E,F\right)$ is given by 
\begin{equation*} 
\mathrm{Ext}_{X\otimes^{\mathbb L}_RA}^2\left(\left(\mathscr E,F\right),\left(\mathscr E\otimes^{\mathbb L}_B I,\hat F\right)\right)\simeq\mathrm{Ext}_{\left(X\otimes^{\mathbb L}_RA\right)\left[t\right]}^2\left(\mathrm{Rees}\left(\left(\mathscr E,F\right)\right),\mathrm{Rees}\left(\left(\mathscr E\otimes^{\mathbb L}_B I,\hat F\right)\right)\right)^{\mathbb G_m}
\end{equation*}
\item the tangent space to the functor $\pi_0\pi^0\mathcal W\left(\mathfrak{FdgMod}\left(X\otimes^{\mathbb L}_R-\right)_{\mathrm{cart}}\right)$ at $\left(\mathscr E,F\right)$ is given by the group 
\begin{equation*} 
\mathrm{Ext}_{X\otimes^{\mathbb L}_RA}^1\left(\left(\mathscr E,F\right),\left(\mathscr E\otimes^{\mathbb L}_B I,\hat F\right)\right)\simeq\mathrm{Ext}_{\left(X\otimes^{\mathbb L}_RA\right)\left[t\right]}^1\left(\mathrm{Rees}\left(\left(\mathscr E,F\right)\right),\mathrm{Rees}\left(\left(\mathscr E\otimes^{\mathbb L}_B I,\hat F\right)\right)\right)^{\mathbb G_m}
\end{equation*}
\item a functorial obstruction space for $\pi_0\pi^0\mathcal W\left(\mathfrak{FdgMod}\left(X\otimes^{\mathbb L}_R-\right)_{\mathrm{cart}}\right)$ at $\left(\mathscr E,F\right)$ is given by 
\begin{equation*} 
\mathrm{Ext}_{X\otimes^{\mathbb L}_RA}^2\left(\left(\mathscr E,F\right),\left(\mathscr E\otimes^{\mathbb L}_B I,\hat F\right)\right)\simeq\mathrm{Ext}_{\left(X\otimes^{\mathbb L}_RA\right)\left[t\right]}^2\left(\mathrm{Rees}\left(\left(\mathscr E,F\right)\right),\mathrm{Rees}\left(\left(\mathscr E\otimes^{\mathbb L}_B I,\hat F\right)\right)\right)^{\mathbb G_m}
\end{equation*}
\end{itemize}
so the group homomorphisms induced on first-order deformations and obstruction theories is just identities, therefore Condition $\left(2\right)$ holds. \\
In terms of Condition $\left(3\right)$, notice that the argument showing the analogous claim in the proof of Theorem \ref{Perf repr} also holds in this context, since the filtered complexes we are parametrising are perfect in each level of the filtration; thus Condition $\left(3\right)$ holds. \\
In order to check Condition $\left(4\right)$, fix a finitely generated $R$-algebra $A$ and a pair $\left(\mathscr E,F\right)\in\mathbf M^n_{\mathrm{filt}}\left(A\right)$ and consider an inductive system $\left\{B_{\alpha}\right\}_{\alpha}$ of $A$-algebras. Since $F^m\mathscr E$ is perfect for any $m$, we can choose a {} ``flat'' resolution (see Theorem \ref{Perf repr} for more explanation) $\left(\mathscr F,\dot F\right)$ for the filtered complex $\left(\mathscr E,F\right)$; therefore there is a chain of isomorphisms
\begin{eqnarray*}
&\mathrm{Ext}^i_{X\otimes_RA}\left(\left(\mathscr E,F\right),\left(\mathscr E\otimes^{\mathbb L}_A\underset{\underset{\alpha}{\longrightarrow}}{\mathrm{lim}}\;B_{\alpha},\hat F\right)\right)& \\
&\mathrm{Ext}^i_{X\otimes_RA}\left(\left(\mathscr E,F\right),\left(\mathscr F\otimes_A\underset{\underset{\alpha}{\longrightarrow}}{\mathrm{lim}}\;B_{\alpha},\hat{\dot F}\right)\right)& \\
&\simeq\mathrm{Ext}^i_{\left(X\otimes_RA\right)\left[t\right]}\left(\mathrm{Rees}\left(\mathscr E,F\right),\mathrm{Rees}\left(\mathscr F\otimes_A\underset{\underset{\alpha}{\longrightarrow}}{\mathrm{lim}}\;B_{\alpha},\hat{\dot F}\right)\right)^{\mathbb G_m} & \\
&\simeq\mathrm{Ext}^i_{\left(X\otimes_RA\right)\left[t\right]}\left(\mathrm{Rees}\left(\mathscr E,F\right),\mathrm{Rees}\left(\underset{\underset{\alpha}{\longrightarrow}}{\mathrm{lim}}\left(\mathscr F\otimes_AB_{\alpha},\hat{\dot F}_{\alpha}\right)\right)\right)^{\mathbb G_m} & \\
&\simeq\mathrm{Ext}^i_{\left(X\otimes_RA\right)\left[t\right]}\left(\mathrm{Rees}\left(\mathscr E,F\right),\underset{\underset{\alpha}{\longrightarrow}}{\mathrm{lim}}\;\mathrm{Rees}\left(\mathscr F\otimes_AB_{\alpha},\hat{\dot F}_{\alpha}\right)\right)^{\mathbb G_m} & \\
&\simeq\left(\underset{\underset{\alpha}{\longrightarrow}}{\mathrm{lim}}\;\mathrm{Ext}^i_{\left(X\otimes_RA\right)\left[t\right]}\left(\mathrm{Rees}\left(\mathscr E,F^{\bullet}\right),\mathrm{Rees}\left(\mathscr F\otimes_AB_{\alpha},\hat{\dot F}_{\alpha}\right)\right)\right)^{\mathbb G_m} & \\
&\simeq\underset{\underset{\alpha}{\longrightarrow}}{\mathrm{lim}}\;\mathrm{Ext}^i_{\left(X\otimes_RA\right)\left[t\right]}\left(\mathrm{Rees}\left(\mathscr E,F\right),\mathrm{Rees}\left(\mathscr F\otimes_AB_{\alpha},\hat{\dot F}_{\alpha}\right)\right)^{\mathbb G_m} & \\
&\simeq\underset{\underset{\alpha}{\longrightarrow}}{\mathrm{lim}}\;\mathrm{Ext}^i_{X\otimes_RA}\left(\left(\mathscr E,F\right),\left(\mathscr F\otimes_AB_{\alpha},\hat{\dot F}_{\alpha}\right)\right) & \\
&\simeq\underset{\underset{\alpha}{\longrightarrow}}{\mathrm{lim}}\;\mathrm{Ext}^i_{X\otimes_RA}\left(\left(\mathscr E,F\right),\left(\mathscr E\otimes^{\mathbb L}_AB_{\alpha},\hat F_{\alpha}\right)\right) &
\end{eqnarray*}
where we have used the various properties collected in Theorem \ref{Rprops}, the induced description of the Ext groups determined by formula \eqref{rees ext}, the exactness of the functor $\left(-\right)^{\mathbb G_m}$ and the filtration-levelwise degreewise flatness of the representative $\left(\mathscr F,\dot F\right)$. It follows that Condition $\left(4\right)$ holds.\\
The way we prove Condition $\left(5\right)$ is exactly the same utilised to show the corresponding claim in Theorem \ref{Perf repr}: indeed, note that such an argument carries over to this context, provided that we use the {} ``filtered version'' of the local-to-global spectral sequence given by Remark \ref{filt loc-to-glob} in place of the classical one; thus Condition $\left(5\right)$ holds. \\
Now we look at Condition $\left(6\right)$; fix an inductive system $\left\{A_{\alpha}\right\}_{\alpha}$ of $R$-algebras and let $A:=\underset{\underset{\alpha}{\longrightarrow}}{\mathrm{lim}}\;A_{\alpha}$: we need to show that 
\begin{equation} \label{filt components}
c\left(\pi_0\mathbf M_{\mathrm{filt}}^n\left(A\right)\right)=\underset{\underset{\alpha}{\longrightarrow}}{\mathrm{lim}}\;c\left(\pi_0\mathbf M_{\mathrm{filt}}^n\left(A_{\alpha}\right)\right)
\end{equation}
where for any $R$-algebra $B$
\begin{equation} \label{filt descr comp}
c\left(\pi_0\mathbf M_{\mathrm{filt}}^n\left(B\right)\right):=\left\{\text{isomorphism classes of filtered perfect complexes of $\left(\mathscr O_X\otimes_R^{\mathbb L}B\right)$-modules}\right\}.
\end{equation}
According to formula \eqref{filt descr comp} an element in $\underset{\underset{\alpha}{\longrightarrow}}{\mathrm{lim}}\;c\left(\pi_0\mathbf M_{\mathrm{filt}}^n\left(A_{\alpha}\right)\right)$ consists of a direct system of classes
\begin{equation} \label{filt class}
\tiny{\left\{\left[\vcenter{\xymatrix{\mathscr E_{\alpha} \\ 
F^1\mathscr E_{\alpha}\ar@{^{(}->}[u] \\
F^2\mathscr E_{\alpha}\ar@{^{(}->}[u] \\
\vdots\ar@{^{(}->}[u] }} \right]\right\}_{\alpha} }
\end{equation}
where for all $p$ and all $\alpha$ $F^p\mathscr E_{\alpha}$ is a perfect complex of $\mathscr O_X\otimes^{\mathbb L}_R A_{\alpha}$-modules. In the proof of Theorem \ref{Perf repr} we have shown that each system $\left\{\left[F^p\mathscr E_{\alpha}\right]\right\}_{\alpha}$ determines uniquely an isomorphism class of perfect $\mathscr O_X\otimes_R^{\mathbb L}A$-module in complexes and notice that inclusions are preserved under inductive limits, thus the object described by formula \eqref{filt class} determines a unique class in $c\left(\pi_0\mathbf M_{\mathrm{filt}}^n\left(A\right)\right)$, which means that formula \eqref{filt descr comp} is verified. It follows that Condition $\left(6\right)$ holds.\\
Lastly, we have to check Condition $\left(7\right)$, so fix a complete discrete local Noetherian $R$-algebra $A$ and a pair $\left(\mathscr E,F\right)\in\mathbf M^n_{\mathrm{filt}}\left(A\right)$
. \\
Consider for all $i<0$ the canonical map
\begin{equation} \label{filt ext comp}
\mathrm{Ext}^i_{X\otimes^{\mathbb L}_RA}\left(\left(\mathscr E,F\right),\left(\mathscr E,F\right)\right)\longrightarrow\underset{\underset{r}{\longleftarrow}}{\mathrm{lim}}\;\mathrm{Ext}^i_{\mathcal F\otimes_R^{\mathbb L}A}\left(\left(\mathscr E,F\right),\left(\nicefrac{\mathscr E}{\mathfrak m_A^r},\tilde F\right)\right)
\end{equation}
induced by the isomorphism
\begin{equation}
\hat A:=\underset{\underset{r}{\longleftarrow}}{\lim}\;\nicefrac{A}{\mathfrak m_A^r}
\end{equation}
to the pronilpotent completion of $A$. Now by formula \eqref{rees ext} we see that map \eqref{filt ext comp} is the same as the group morphism
\begin{equation*}
\tiny{\xymatrix{\mathrm{Ext}^i_{\left(X\otimes^{\mathbb L}_RA\right)\left[t\right]}\left(\mathrm{Rees}\left(\left(\mathscr E,F\right)\right),\mathrm{Rees}\left(\left(\mathscr E,F\right)\right)\right)^{\mathbb G_m}\ar[d] \\
\underset{\underset{r}{\longleftarrow}}{\mathrm{lim}}\;\mathrm{Ext}^i_{\left(\mathcal F\otimes_R^{\mathbb L}A\right)\left[t\right]}\left(\mathrm{Rees}\left(\left(\mathscr E,F\right)\right),\mathrm{Rees}\left(\left(\nicefrac{\mathscr E}{\mathfrak m_A^r},\tilde F\right)\right)\right)^{\mathbb G_m}}}
\end{equation*}
which is an isomorphism, as follows by combining the exactness of the functor $\left(-\right)^{\mathbb G_m}$ and the computations in the proof of Theorem \ref{Perf repr}.\\
At last, the compatibility condition on the components is easily checked by using techniques similar to the ones utilised to verify Condition $\left(6\right)$. As a matter of fact take any inverse system  
\begin{equation} \label{filt class}
\tiny{\left\{\left[\vcenter{\xymatrix{\mathscr E_{r} \\ 
F^1\mathscr E_{r}\ar@{^{(}->}[u] \\
F^2\mathscr E_{r}\ar@{^{(}->}[u] \\
\vdots\ar@{^{(}->}[u] }} \right]\right\}_r }
\end{equation}
of filtered perfect complexes of $\mathscr O_X\otimes^{\mathbb L}_R\nicefrac{A}{\mathfrak m^r_A}$-modules and note that the proof of the corresponding statement in Theorem \ref{Perf repr} allows us to lift each level $F^p\mathscr E_r$ to a perfect complex of $\mathscr O_X\otimes^{\mathbb L}_RA$-modules; moreover countable limits preserve inclusions: this concludes the verification of the claim. \\
It follows that Condition $\left(7\right)$ holds, so the proof is complete.
\end{proof}
Now define
\begin{equation*}
\mathbb R\mathcal Filt_X:=\underset{n}{\bigcup}\;\mathbb R\mathcal Filt_X^n
\end{equation*}
which is the simplicial presheaf parametrising filtered perfect complexes over $X$: Theorem \ref{filt Perf repr} ensures that $\mathbb R\mathcal Filt_X$ is a locally geometric derived stack over $R$; this comment provides the ultimate comparison between moduli of complexes and moduli of filtered complexes.
\begin{rem}
An interesting derived substack of $\mathbb R\mathcal Filt_X$ is the \emph{stack of submodules} over $X$, which we denote by $\mathbb R\mathcal Sub_X$; this is the simplicial presheaf over $\mathfrak{dgAlg}_k^{\leq 0}$ parametrising filtered perfect $\mathscr O_X$-modules in complexes $\left(\mathscr E,F\right)$ such that the filtration $F$ has length $2$: in other words the functor $\mathbb R\mathcal Sub_X$ classifies pairs made of a perfect complex $\mathscr E$ and a subcomplex $F^1\mathscr E$, which is perfect as well. $\mathbb R\mathcal Sub_X$ is clearly a derived substack of $\mathbb R\mathcal Filt_X$ and it is also locally geometric; as a matter of fact consider the simplicial presheaf $\mathbb R\mathcal Sub^n_X$ parametrising filtered $\mathscr O_X$-modules in complexes for which the filtration has length $2$ and the relevant higher Ext groups vanish: then the arguments and techniques explained in this section show that $\mathbb R\mathcal Sub^n_X$ is a derived geometric $n$-stack over $R$ and moreover we have that
\begin{equation*}
\mathbb R\mathcal Sub_X=\bigcup_n\mathbb R\mathcal Sub^n_X
\end{equation*}
thus $\mathbb R\mathcal Sub_X$ is locally geometric.
\end{rem}
\subsection{Homotopy Flag Varieties and Derived Grassmannians}
In this last section we will see how the ideas and notions discussed in Section 2.2 and Section 2.3 allow us to construct homotopy versions of Grassmannians and flag varieties.
Throughout all this section fix our base scheme $X$ to be the point $\mathrm{Spec}\left(k\right)$ and let $V$ be a bounded complex of finite-dimensional $k$-vector spaces; furthermore for all $n\in\mathbb N$ consider the derived stacks
\begin{eqnarray*}
&\mathbb R\mathcal Perf_k:=\mathbb R\mathcal Perf_{\mathrm{Spec}\left(k\right)} \qquad\mathbb R\mathcal Perf^n_k:=\mathbb R\mathcal Perf^n_{\mathrm{Spec}\left(k\right)}& \\
&\mathbb R\mathcal Filt_k:=\mathbb R\mathcal Filt_{\mathrm{Spec}\left(k\right)} \qquad \mathbb R\mathcal Filt^n_k:=\mathbb R\mathcal Filt^n_{\mathrm{Spec}\left(k\right)} \\
&\mathbb R\mathcal Sub_k:=\mathbb R\mathcal Sub_{\mathrm{Spec}\left(k\right)} \qquad \mathbb R\mathcal Sub^n_k:=\mathbb R\mathcal Sub^n_{\mathrm{Spec}\left(k\right)}
\end{eqnarray*}
which respectively parametrise:
\begin{itemize}
\item cochain complexes of $k$-vector spaces; 
\item filtered cochain complexes of $k$-vector spaces;
\item pairs made of a cochain complex of $k$-vector spaces and a subcomplex.
\end{itemize}
\begin{defn} \label{big Grass}
Define the \emph{big total derived Grassmannian} over $k$ associated to $V$ to be the derived stack given by the homotopy fibre
\begin{equation*}
\mathcal{DGRASS}_k\left(V\right):=\underset{\longleftarrow}{\mathrm{holim}}\left(\mathbb R\mathcal Sub_k\doublerightarrow{\left[F\hookrightarrow W\right]\mapsto W}{\mathrm{const}_V}\mathbb R\mathcal Perf_k\right)
\end{equation*}
where the top map is the natural forgetful morphism while {} ``$\mathrm{const}_V$'' stands for the constant morphism sending any pair $\left[F\hookrightarrow W\right]$ to $V$.
\end{defn}
\begin{rem}
The derived stack $\mathcal{DGRASS}_k\left(W\right)$ is locally geometric: as a matter of fact we have that
\begin{equation*}
\mathcal{DGRASS}_k\left(V\right)=\bigcup_n\mathcal{DG}rass^n_k\left(V\right)
\end{equation*}
where
\begin{equation} \label{DGrass loc geom}
\mathcal{DGRASS}^n_k\left(V\right):=\underset{\longleftarrow}{\mathrm{holim}}\left(\mathbb R\mathcal Sub^n_k\doublerightarrow{\left[F\hookrightarrow W\right]\mapsto W}{\mathrm{const}_V}\mathbb R\mathcal Perf^n_k\right)
\end{equation}
and formula \eqref{DGrass loc geom} shows in particular that $\mathcal{DGRASS}^n_k\left(V\right)$ is a derived geometric $n$-stack over $k$.
\end{rem}
There is a more concrete realisation of the big total derived Grassmannian associated to $V$: indeed consider the functorial simplicial category
\begin{eqnarray} \label{DGrass real}
\forall A\in\mathfrak{dgAlg}_k^{\leq0}\quad DGRASS_k\left(V\right)\left(A\right):=&\text{full simplicial subcategory made of sequences}& \nonumber \\
& U\hookrightarrow W\overset{\varphi}{\rightarrow} V\otimes A& \nonumber \\
&\text{of cofibrant $A$-modules in complexes}& \nonumber \\ 
&\text{where $\varphi$ is a quasi-isomorphism}&
\end{eqnarray}
and observe that $\mathcal{DGRASS}_k\left(V\right)=\bar WDGRASS_k\left(V\right)$. \\
Similarly, we can construct a preliminary derived notion of flag variety.
\begin{defn}
Define the \emph{big homotopy flag variety} over $k$ associated to $V$ to be the derived stack given by the homotopy fibre
\begin{equation*}
\mathcal{DFLAG}_k\left(V\right):=\underset{\longleftarrow}{\mathrm{holim}}\left(\mathbb R\mathcal Filt_k\doublerightarrow{\left(W,F\right)\mapsto W}{\mathrm{const}_V}\mathbb R\mathcal Perf_k\right)
\end{equation*}
where the top map is the natural forgetful morphism while {} ``$\mathrm{const}_V$'' denotes again the constant morphism sending any filtered complex to $V$.
\end{defn}
\begin{rem}
The derived stack $\mathcal{DFLAG}_k\left(V\right)$ is locally geometric: as a matter of fact we have that
\begin{equation*}
\mathcal{DFLAG}_k\left(V\right)=\bigcup_n\mathcal{DFLAG}^n_k\left(V\right)
\end{equation*}
where
\begin{equation} \label{DFlag loc geom}
\mathcal{DFLAG}^n_k\left(V\right):=\underset{\longleftarrow}{\mathrm{holim}}\left(\mathbb R\mathcal Filt^n_k\doublerightarrow{\left(W,F\right)\mapsto W}{\mathrm{const}_V}\mathbb R\mathcal Perf^n_k\right)
\end{equation}
and formula \eqref{DFlag loc geom} shows in particular that $\mathcal{DFLAG}^n_k\left(V\right)$ is a derived geometric $n$-stack over $k$.
\end{rem}
There is a concrete realisation of the big homotopy flag variety given by equations similar to the ones supplied in formula \eqref{DGrass real} in the case of the big derived Grassmannian; as a matter of fact define the functorial simplicial category
\begin{eqnarray} 
\forall A\in\mathfrak{dgAlg}_k^{\leq0}\;\;\; DFLAG_k\left(V\right)\left(A\right):=&\text{full simplicial subcategory made of pairs }\left(\left(W,F\right),\varphi\right),& \nonumber \\
&\text{with $\left(W,F\right)$ a filtered cofibrant $A$-module in complexes}& \nonumber \\ 
&\text{and $\varphi:W\rightarrow V\otimes A$ a quasi-isomorphism} \nonumber&
\end{eqnarray}
and observe that $\mathcal{DFLAG}_k\left(V\right)=\bar WDFLAG_k\left(V\right)$.
\begin{rem} \label{count to class grass}
Assume that $V$ is concentrated in degree $0$ and consider the classical total Grassmannian variety
\begin{eqnarray} \label{class grass defn}
&\mathrm{Grass}\left(V\right):=\underset{i=0}{\overset{\mathrm{dim}\,V}{\coprod}} \mathrm{Grass}\left(i,V\right)& \\
&\mathrm{Grass}\left(i,V\right):=\left\{W\subseteq V\text{ s.t. }\mathrm{dim}\left(W\right)=i\right\}.\nonumber&
\end{eqnarray}
We would like that the stack $\mathcal{DGRASS}_k\left(V\right)$ were a derived enhancement of the variety \eqref{class grass defn}, but unfortunately this is not the case as $\mathcal{DGRASS}_k\left(V\right)$ is far too large: for instance, we have that $\mathcal{DGRASS}_k\left(0\right)\approx\mathbb R\mathcal Perf_k$; analogous statements will hold for the flag variety $\mathrm{Flag}\left(V\right)$.
\end{rem}
Remark \ref{count to class grass} tells us that the big total derived Grassmannian and the big homotopy flag variety are not derived enhancements of $\mathrm{Grass}\left(V\right)$ and $\mathrm{Flag}\left(V\right)$; anyhow hereinafter we will show that the two latter varieties can be realised respectively as the underived truncations of natural open substacks of $\mathcal{DGRASS}_k\left(V\right)$ and $\mathcal{DFLAG}_k\left(V\right)$.\\ 
Consider the (underived) functorial simplicial category
\begin{eqnarray*}
\forall A\in\mathfrak{Alg}_k\quad DGrass_k\left(V\right)\left(A\right):=&\text{full simplicial subcategory of }DGRASS_k\left(V\right)\left(A\right)& \\
&\text{made of sequences }U\hookrightarrow W\overset{\phi}{\rightarrow} V\otimes A& \\
&\text{for which } H^i\left(U\right) \text{ is flat over } A\\
&\text{and the induced morphism}& \\
&H^i\left(U\right)\rightarrow H^i\left(V\right)\otimes A& \\
&\text{is injective for all $i$.}&
\end{eqnarray*}
as well as its enhancement
\begin{eqnarray*}
\forall A\in\mathfrak{dgAlg}_k^{\leq0}\quad\widetilde{DGrass}_k\left(V\right)\left(A\right):=&\scriptstyle{DGRASS_k\left(V\right)\left(A\right)\times^{\left(2\right)}_{DGRASS_k\left(V\right)\left(H^0\left(A\right)\right)}DGrass_k\left(V\right)\left(H^0\left(A\right)\right)}&\\
=&\text{full simplicial subcategory of } DGRASS_k\left(V\right)\left(A\right)& \\
&\text{made of sequences }U\hookrightarrow W\overset{\phi}{\rightarrow} V\otimes A& \\
&\text{weakly equivalent to an object in $DGrass_k\left(V\right)\left(H^0\left(A\right)\right)$}& \nonumber \\
&\text{after tensorisation with $H^0\left(A\right)$ over $A$}&
\end{eqnarray*}
\begin{defn} \label{actual Gr}
For any cochain complex $V$ define the \emph{derived total Grassmannian associated to $V$} to be
\begin{equation*}
\mathcal{DG}rass_k\left(V\right):=\bar W\widetilde{DGrass}_k\left(V\right).
\end{equation*}
\end{defn}
\begin{prop} \label{comp gras 3}
$\mathcal{DG}rass_k\left(V\right)$ is an open derived substack of $\mathcal{DGRASS}_k\left(V\right)$.
\end{prop}
\begin{proof}
We want to show that the inclusion
\begin{equation*}
i:\mathcal{DG}rass_k\left(V\right)\hookrightarrow\mathcal{DGRASS}_k\left(V\right)
\end{equation*}
is \' etale, which in turn amounts to prove that the induced map of formal groupoids
\begin{equation*}
\pi_{\leq0}DGrass_k\left(V\right)\longrightarrow\pi_{\leq0}DGRASS_k\left(V\right)
\end{equation*}
is formally \' etale. \\
Let $I\hookrightarrow A\twoheadrightarrow B$ be a square-zero extension in $\mathfrak{Alg}_k$ and pick a triple $\left[S\hookrightarrow W\rightarrow V\otimes B\right]$ in $\pi_{\leq0}DGrass_k\left(V\right)\left(B\right)$ -- i.e. such that the induced morphism $H^i\left(S\right)\rightarrow H^i\left(V\right)\otimes B$ stays injective for all $i$ -- and take $\left[S'\hookrightarrow W'\rightarrow V\otimes A\right]$ in $DGRASS_k\left(V\right)\left(A\right)$ such that $S'\otimes_AB\approx S$ and $W'\otimes_A B\approx W$; we need to show that the cohomology map $H^i\left(S'\right)\rightarrow H^i\left(V\right)\otimes A$ is injective. By taking long exact sequence in cohomology we end up with a morphism of complexes
\begin{eqnarray} \label{diag final hope}
\xymatrix{\cdots\ar[r]& H^i\left(S\right)\otimes_BI\ar[r]\ar@{_{(}->}[d] & H^i\left(S'\right)\ar[r]\ar[d] & H^i\left(S\right)\ar[r]\ar@{_{(}->}[d] & \cdots \\
\cdots\ar[r] & H^i\left(V\right)\otimes I\ar[r] & H^i\left(V\right)\otimes A\ar[r] & H^i\left(V\right)\otimes B\ar[r] & \cdots} \nonumber \\
\end{eqnarray}
in which the horizontal arrows are exact. Let $v\in H^i\left(S'\right)$ an element mapping to $0$ in $H^i\left(V\right)\otimes A$ and hence to $0$ in $H^i\left(V\right)\otimes B$; the injectivity of the map $H^i\left(S\right)\rightarrow H^i\left(V\right)\otimes B$ implies that 
\begin{equation*}
v\in\mathrm{ker}\left(H^i\left(S'\right)\rightarrow H^i\left(S\right)\right)\simeq\mathrm{Im}\left(H^i\left(S\right)\otimes_BI\rightarrow H^i\left(S'\right)\right)
\end{equation*}
so let $w$ be a preimage of $v$ inside $H^i\left(S\right)\otimes_B I$. Now let us walk along the commutative square on the left-hand side of diagram \eqref{diag final hope}: we know that the (vertical) map 
\begin{equation*}
H^i\left(S\right)\otimes_BI\rightarrow H^i\left(V\right)\otimes B\otimes_BI\simeq H^i\left(V\right)\otimes I
\end{equation*}
is injective and notice furthermore that $H^i\left(V\right)$ is flat over $k$; as $I\hookrightarrow  A$, it follows then that the (horizontal) map
\begin{equation*}
H^i\left(V\right)\otimes I\simeq H^i\left(V\right)\otimes B\otimes_BI\rightarrow H^i\left(V\right)\otimes A 
\end{equation*}
is injective, as well. As a result, we have that $w$ is mapped to $0$ in $H^i\left(V\right)\otimes A$ via the composite of two injections, therefore $w=0$ and $v=0$. This means that the map $H^i\left(S'\right)\rightarrow H^i\left(V\right)\otimes A$ is injective, which concludes the proof.
\end{proof}
\begin{thm} \label{comp gras 4}
There is an isomorphism 
\begin{equation} \label{bingo}
\pi^0\pi_{\leq0}\mathcal{DG}rass_k\left(V\right)\simeq\mathrm{Grass}\left(H^*\left(V\right)\right)
\end{equation}
where the right-hand side in formula \eqref{bingo} is the product of the classical total Grassmannians associated to the vector spaces $H^i\left(V\right)$; in particular if $V$ is concentrated in degree $0$ than $\mathcal{DG}rass_k\left(V\right)$ is a derived enhancement of the classical total Grassmannian associated to $V$.
\end{thm}
\begin{proof}
We want to show that $\pi^0\pi_{\leq0}\mathcal{DG}rass_k\left(V\right)$ is the same as the functor of points represented by the variety $\mathrm{Grass}\left(H^*\left(V\right)\right)$, which is
\begin{eqnarray*}
\mathbf{Grass}\left(H^*\left(V\right)\right)&:\mathfrak{Alg}_k&\xrightarrow{\hspace*{3cm}}\mathfrak{Set} \\
&A&\longmapsto\left\{W\hookrightarrow H^*\left(V\right)\otimes A\text{ s.t. } W\text{ cofibrant}\right\}.
\end{eqnarray*}
Notice that, for all $A\in\mathfrak{Alg}_k$ we have that
\footnotesize{\begin{equation*}
\pi^0\pi_{\leq0}\mathcal{DG}rass_k\left(V\right)\left(A\right):=\dfrac{\left\{T\hookrightarrow V\otimes A\text{ s.t. } T\;\text{perfect,}\, H^*\left(T\right)\,A\text{-flat,}\,H^*\left(T\right)\rightarrow H^*\left(V\right)\otimes A\;\text{injective}\right\}}{\left\{\text{quasi-isomorphism}\right\}}.
\end{equation*}}\normalsize
Taking cohomology induces a natural bijection between the sets $\pi^0\pi_{\leq0}\mathcal{DG}rass_k\left(V\right)\left(A\right)$ and $\mathbf{Grass}\left(H^*\left(V\right)\right)\left(A\right)$. Indeed consider $\left[W\hookrightarrow H^*\left(V\right)\otimes A\right]\in\mathbf{Grass}\left(H^*\left(V\right)\right)\left(A\right)$: all we need to show is the existence and unicity of a quasi-isomorphism class
\begin{equation*}
\left[T\hookrightarrow V\otimes A\right]\in\pi^0\pi_{\leq0}\mathcal{DG}rass_k\left(V\right)\left(A\right) 
\end{equation*}
whose cohomology is $\left[W\hookrightarrow H^*\left(V\right)\otimes A\right]$; now this follows directly from the observation that the complex $T$ is made of locally free modules in each degree, since it is perfect with flat cohomology.
\end{proof}
The constructions and results described by Definition \ref{actual Gr}, Proposition \ref{comp gras 3} and Theorem \ref{comp gras 4} for Grassmannians readily extend to the more general case of flag varieties. \\
Consider the (underived) functorial simplicial category
\begin{eqnarray*}
\forall A\in\mathfrak{Alg}_k\quad DFlag_k\left(V\right)\left(A\right):=&\text{full simplicial subcategory of }DFLAG_k\left(V\right)\left(A\right)& \\
&\text{made of pairs } \left(\left(W,F\right),\varphi\right)& \\
&\text{for which } H^j\left(F^iW\right) \text{ is flat over } A & \\ 
&\text{and the induced morphisms}& \\
&H^j\left(F^iW\right)\rightarrow H^j\left(F^{i-1}W\right)\rightarrow H^*\left(V\right)\otimes A& \\
&\text{are injective for all $i,j$.}&
\end{eqnarray*}
as well as its enhancement
\begin{eqnarray*}
\forall A\in\mathfrak{dgAlg}_k^{\leq0}\quad\widetilde{DFlag}_V\left(A\right):=&\scriptstyle{DFLAG_k\left(V\right)\left(A\right)\times^{\left(2\right)}_{DFLAG_k\left(V\right)\left(H^0\left(A\right)\right)}DFlag_k\left(V\right)\left(H^0\left(A\right)\right)}&\\
=&\text{full simplicial subcategory of }DFLAG_k\left(V\right)\left(A\right)& \\
&\text{made of pairs }\left(\left(W,F\right),\varphi\right)& \\
&\text{weakly equivalent to an object in $DFlag_k\left(V\right)\left(H^0\left(A\right)\right)$}& \nonumber \\
&\text{after tensorisation with $H^0\left(A\right)$ over $A$}&
\end{eqnarray*}
\begin{defn} \label{actual G}
Define the \emph{homotopy flag variety associated to $V$} to be
\begin{equation*}
\mathcal{DF}lag_k\left(V\right):=\bar WDFlag_k\left(V\right).
\end{equation*}
\end{defn}
\begin{prop} 
$\mathcal{DF}lag_k\left(V\right)$ is an open derived substack of $\mathcal{DFLAG}_k\left(V\right)$.
\end{prop}
\begin{proof}
The proof of Proposition \ref{comp gras 3} carries over to this context.
\end{proof}
\begin{thm}
The homotopy flag variety associated to $V$ is a derived enhancement of the classical total flag variety attached to $H^*\left(V\right)$, i.e. 
\begin{equation*}
\pi^0\pi_{\leq0}\mathcal{DF}lag_k\left(V\right)\simeq\mathrm{Flag}\left(H^*\left(V\right)\right).
\end{equation*}
\end{thm}
\begin{proof}
The proof of Theorem \ref{comp gras 4} carries over to this context. 
\end{proof}
\begin{rem}
In this paper we have focused on the study of the global theory of Grassmannians and flag varieties in Derived Algebraic Geometry, ending up with the construction of $\mathcal{DG}rass_k\left(V\right)$ and $\mathcal{DF}lag_k\left(V\right)$. The infinitesimal picture of these stacks -- including the computation of their tangent complexes -- will be analysed in \cite{dN}.
\end{rem}

\section*{Notations and conventions}
\begin{itemize}
\item $\mathrm{diag}\left(-\right)=$ diagonal of a bisimplicial set
\item $k=$ fixed field of characteristic $0$, unless otherwise stated
\item If $A$ is a (possibly differential graded) local ring, $\mathfrak m_A$ will be its unique maximal (possibly differential graded) ideal
\item $R=$ (possibly differential graded) commutative unital ring or $k$-algebra, unless otherwise stated
\item If $\left(M,d\right)$ is a cochain complex (in some suitable category) then $\left(M\left[n\right],d_{\left[n\right]}\right)$ will be the cochain complex such that $M\left[n\right]^j:=M^{j+n}$ and $d_{\left[n\right]}^j=d^{j+n}$
\item $\mathbb G_m=$ multiplicative group scheme over $k$
\item $X=$ semi-separated quasi-compact (possibly proper) scheme over $R$ or $k$ of finite dimension, unless otherwise stated
\item $\mathcal X=$ derived scheme over $R$
\item $\mathscr O_X=$ structure sheaf of $X$
\item $\mathbb L^{\mathcal F/R}=$ (absolute) cotangent complex of the derived geometric stack $\mathcal F$ over $R$
\item $\mathrm D\left(X\right)=$ derived category of $X$
\item $\Delta=$ category of finite ordinal numbers
\item $\mathfrak{Alg}_k=$ category of commutative associative unital algebras over $k$
\item $\mathfrak{Alg}_R=$ category of commutative associative unital algebras over $R$
\item $\mathfrak{Alg}_{H^0\left(R\right)}=$ category of commutative associative unital algebras over $H^0\left(R\right)$
\item $\mathfrak{Ch}_{\geq 0}\left(\mathfrak{Vect}_k\right)=$ model category of chain complexes of $k$-vector spaces 
\item $\mathfrak{dgAlg}^{\leq 0}_k=$ model category of (cochain) differential graded commutative algebras over $k$ in non-positive degrees
\item $\mathfrak{dgAlg}^{\leq 0}_R=$ model category of (cochain) differential graded commutative algebras over $R$ in non-positive degrees
\item $\mathfrak{dgAlg}^{\leq 0}_{R\left[t\right]}=$ model category of (cochain) differential graded commutative algebras over $R\left[t\right]$ in non-positive degrees
\item $\mathfrak{dgArt}^{\leq 0}_k=$ model category of (cochain) differential graded local Artin algebras over $k$ in non-positive degrees
\item $\mathfrak{dgMod}_R=$ model category of $R$-modules in (cochain) complexes
\item $\mathfrak{dgMod}\left(\mathcal X\right)=$ model category of derived modules over $\mathcal X$
\item $\mathfrak{dgMod}\left(\mathcal X\right)_{\mathrm{cart}}=$ $\infty$-category of derived quasi-coherent sheaves over $\mathcal X$
\item $\mathfrak{dg}_b\mathfrak{Nil}^{\leq 0}_R=$ $\infty$-category of bounded below differential graded commutative $R$-algebras in non-positive degrees such that the canonical map $A\rightarrow H^0\left(A\right)$ is nilpotent
\item $\mathfrak{dg}_b\mathfrak{Nil}^{\leq 0}_{H^0\left(R\right)}=$ $\infty$-category of bounded below differential graded commutative $H^0\left(R\right)$-algebras in non-positive degrees such that the canonical map $A\rightarrow H^0\left(A\right)$ is nilpotent
\item $\mathfrak{dgVect}_k^{\leq 0}=$ model category of (cochain) differential graded vector spaces over $k$ in non-positive degrees
\item $\mathfrak{FdgMod}_R=$ model category of filtered $R$-modules in (cochain) complexes
\item $\mathfrak{FdgMod}\left(\mathcal X\right)=$ model category of filtered derived modules over $\mathcal X$
\item $\mathfrak{FdgMod}\left(\mathcal X\right)_{\mathrm{cart}}=$ $\infty$-category of filtered derived quasi-coherent sheaves over $\mathcal X$
\item $\mathbb G_m$-$\mathfrak{dgMod}_{R\left[t\right]}=$ model category of graded $R\left[t\right]$-modules in (cochain) complexes
\item $\mathbb G_m$-$\mathfrak{dgMod}\left(\mathcal X\left[t\right]\right)=$ model category of graded derived modules over $\mathcal X\left[t\right]$\footnote{The definition of the derived scheme $\mathcal X\left[t\right]$ is given in the body of the paper (Section 2.3).}
\item $\mathfrak{Grpd}=$ 2-category of groupoids
\item $\mathfrak{Mod}_R=$ category of $R$-modules
\item $\mathfrak{Perf}\left(X\right)=$ dg-category of perfect complexes of $\mathscr O_X$-modules
\item $\mathfrak{QCoh}\left(X\right)=$ category of quasi-coherent sheaves over $X$
\item $\mathfrak{Set}=$ category of sets
\item $\mathfrak{sAlg}_k=$ model category of simplicial commutative associative unital algebras over $k$ 
\item $\mathfrak{sCat}=$ model category of simplicial categories
\item $\mathfrak{sSet}=$ simplicial model category of simplicial sets
\item $\mathfrak{sVect}_k=$ model category of simplicial vector spaces over $k$
\item $\mathfrak{Vect}_k=$ category of vector spaces over $k$
\end{itemize}

\end{document}